\documentclass[12pt,a4paper]{amsart}

\usepackage{amsmath,amsthm,amssymb,mathtools,amsfonts,mathrsfs}
\usepackage{xspace,xcolor}
\usepackage[breaklinks,colorlinks,citecolor=teal,linkcolor=teal,urlcolor=teal,pagebackref,hyperindex]{hyperref}
\usepackage[alphabetic]{amsrefs}
\usepackage{tikz-cd}
\usepackage[all]{xy}
\usepackage{bbm}
\usepackage{MnSymbol}
\usepackage{stmaryrd}

\usepackage[bbgreekl]{mathbbol}

\DeclareSymbolFontAlphabet{\mathbb}{AMSb}
\DeclareSymbolFontAlphabet{\mathbbl}{bbold}

\newtheorem{theorem}{Theorem}[section]
\newtheorem{prop}[theorem]{Proposition}
\newtheorem{proposition}[theorem]{Proposition}
\newtheorem{corollary}[theorem]{Corollary}

\newtheorem{thm}[theorem]{Theorem}
\newtheorem{lemma}[theorem]{Lemma}

\theoremstyle{definition} 
\newtheorem{defn}[theorem]{Definition}
\newtheorem{definition}[theorem]{Definition}

\newtheorem{remark}[theorem]{Remark}

\newcommand{\qu}{/\kern-.7ex/}
\newcommand{\lqu}{\backslash \kern-.7ex \backslash}
\newcommand{\on}{\operatorname} 
\newcommand{\Aut}{\on{Aut}}

\newcommand{\NE}{\on{NE}}

\newcommand{\bt}{\mathbf{t}}

\title[Relative MS, theta and Gamma]{Relative mirror symmetry, theta functions and the Gamma conjecture}

\author{Fenglong You}
\address{School of Mathmatical Sciences \\ University of Nottingham, \\University Park \\Nottingham, NG7 2RD \\United Kingdom.}
\email{fenglong.you@nottingham.ac.uk}

\thanks{}


\begin{document}
\date{\today}

\begin{abstract} 
Let $X$ be a Fano variety, and $D\subset X$ be an snc anticanonical divisor. We study relative mirror symmetry for the log Calabi--Yau pair $(X,D)$. (1) We prove a relative mirror theorem for snc pairs without assuming the divisors are nef. (2) We study theta functions associated with the pair $(X,D)$. (3) We introduce functions on the mirror that are obtained from the higher-degree part of the big relative quantum cohomology. As an application, we use these new ingredients in relative mirror symmetry to prove a version of the mirror symmetric Gamma conjecture for $X$ for $\mathcal O_X$ and $\mathcal O_{\on{pt}}$ in this setting, where the Landau--Ginzburg potential is defined as a sum of theta functions. 
\end{abstract}

\maketitle 

\tableofcontents

\section{Introduction}

\subsection{Overview} The main objective of this article is to study relative mirror symmetry of log Calabi--Yau pairs extending \cite{TY20b}, \cite{TY20c}, \cite{You24}, then use it to prove the mirror symmetric Gamma conjecture  for Fano varieties \cite{GI19}, \cite{Iritani09} under this set-up. We would like to clarify that we consider the mirror symmetric Gamma conjecture for the Fano variety itself, not a ``relative'' Gamma conjecture for the log Calabi--Yau pair.

\subsubsection{The Gamma conjecture}

We first give a brief review of the Gamma conjecture. The (non-mirror-symmetric) Gamma conjecture for Fano varieties \cite{GGI} is the relationship between topological data (the Gamma class) of a Fano variety $X$ and rational curve counts (genus zero Gromov--Witten invariants) of $X$.  

For a Fano variety $X$, the Gamma class $\hat{\Gamma}_X$ is a characteristic class defined by
\[
\hat{\Gamma}_X=\prod_{i=1}^n\Gamma(1+\delta_i)=\exp\left(-\gamma c_1(TX)+\sum_{k=2}^\infty (-1)^k\zeta(k)(k-1)!\on{ch}_k(TX)\right),
\]
where $\delta_i$ are Chern roots of $TX$ such that
\[
c(TX)=\prod_{i=1}^n(1+\delta_i);
\]
the Euler constant $\gamma$ is
\[
\gamma=\lim_{k\rightarrow \infty}\left(\sum_{i=1}^k\frac{1}{i}-\log k \right); 
\]
and $\zeta(k)$ is the Riemann zeta function.

We recall that the (small) $J$-function of $X$ is
\[
J_X(\tau_{0,2},z)=e^{\tau_{0,2}/z}\left(1+\sum_{\beta\in \NE(X),\beta \neq 0}e^{\tau_{0,2}\cdot\beta}\sum_{\phi_i}\left\langle \frac{\phi^i}{z(z-\psi)}\right\rangle_{0,1,\beta}^X\phi_i\right),
\]
where $\tau_{0,2}\in H^2(X)$, $\{\phi_i\}$ and $\{\phi^i\}$ are dual bases of $H^*(X)$. Gamma conjecture I in \cite{GGI} is the following identity:
\[
[\hat{\Gamma}_X]=\lim_{t\rightarrow +\infty}[J_X(c_1\log t, 1)].
\]

The Gamma conjecture I has been proved in various cases. See, for example, \cite{GI19}, \cite{Iritani09}, \cite{Iritani11}, \cite{GZ}, \cite{GGI}, \cite{HKLY}.

Recently, a counterexample to the Gamma conjecture was discovered in \cite{GHIKLS}. Note that this is not a counterexample to the mirror symmetric Gamma conjecture. The mirror symmetric Gamma conjecture remains as a valid conjecture and we will study it here. It may be interesting to see if our method to the mirror symmetric Gamma conjecture will be useful for the proof of the (modified) non-mirror-symmetric Gamma conjecture.

\subsubsection{Mirror symmetric Gamma conjecture}

Now we consider the mirror symmetric Gamma conjecture \cite{GI19}. The mirror of a Fano variety $X$ of complex dimension $n$ is a Landau--Ginzburg model $(\check X,W)$, where 
\[
W=W(x_1,\ldots,x_n): \check X\rightarrow \mathbb C
\]
 is called the superpotential. Under some assumptions \cite{GI19}, mirror symmetric Gamma conjecture is the identity
\begin{align}\label{iden-gamma-conj}
    \int_{(\mathbb R_{>0})^n} e^{-W/z}\frac{dx_1\cdots dx_n}{x_1\cdots x_n}=\int_X(z^{c_1}z^{\deg/2}J_X(0,-z))\cup \hat{\Gamma}_X,
\end{align}
for $z>0$, where $\deg$ is an endomorphism of $H^*(X)$ that sends $\phi$ to 
$\deg(\phi_i)\phi_i$. 

One can consider more general cycles and integrands in the mirror symmetric Gamma conjecture. Given a coherent sheaf $V$ on $X$, one may consider the homological mirror $C_V\subset \check{X}$ of $V$. Given an $n$-form $\check{\varphi}_{\tau_{0,2}}(z) e^{-W_{\tau_{0,2}}/z}\omega$ that is mirror to a class $\varphi\in H^*(X)$, the mirror symmetric Gamma conjecture is of the form
\begin{align}\label{conj-homo}
    \int_{C_V}\check{\varphi}_{\tau_{0,2}}(-z) e^{-W_{\tau_{0,2}}/z}\omega=\int_X (z^{c_1}z^{\deg/2}\mathbb J_{X}(\tau_{0,2},-z)\varphi)\cup \hat{\Gamma}_X \on{Ch}(V),
\end{align}
where
\[
\mathbb J_X(\tau_{0,2},z)\varphi=e^{\tau_{0,2}/z}\left(\varphi+\sum_{\beta\neq 0} e^{\tau_{0,2}\cdot\beta}\sum_{\phi_i}\langle \varphi,\frac{\phi^i}{z-\psi}\rangle_{0,2,\beta}^X\phi_i\right)
\]
and, 
\[
\on{Ch}(V):=(2\pi \sqrt{-1})^{\deg/2}\on{ch}(V)=\sum_{k\geq 0}(2\pi \sqrt{-1})^k\on{ch}_k(V)
\]
is a modified Chern character.

Mirror symmetric Gamma conjecture has been studied in various settings. See, for example, \cite{Iritani09}, \cite{Iritani11}, \cite{GI19}, \cite{AGIS}, \cite{AFW}, \cite{FWZ}, \cite{Wang20}.

\subsubsection{Key ingredients in relative mirror symmetry}
In this article, we study the mirror symmetric Gamma conjecture via a different approach. Key ingredients for the proof of the mirror symmetric Gamma conjecture are the following:
\begin{itemize}
    \item A mirror construction for a log Calabi--Yau pair $(X,D)$. More precisely, we consider a variant of Gross--Siebert's intrinsic mirror symmetry \cite{GS19} studied in \cite{TY20c} using orbifold invariants of root stacks instead of punctured logarithmic invariants.  As we use different invariants, we probably lose the geometric mirror duality of the Gross--Siebert program. However, we do not claim or pursue any geometric mirror duality with this variation. Therefore, one may prefer not to think about this mirror construction as part of the Gross--Siebert program. Instead, one may simply think of it as a way of constructing the mirrors with ideas borrowing from the intrinsic mirror symmetry. This is the version of relative mirror symmetry that we consider for the mirror symmetric Gamma conjecture. 
    \item Functions in this mirror construction. The Landau--Ginzburg potential in this setting is defined in terms of the sum of theta functions. This is the potential $W$ that we will use in (\ref{conj-homo}). We define generalizations of the theta functions.  In general, we need to consider the higher degree part of the big relative quantum cohomology. To our knowledge, these functions, defined in Section \ref{sec:func-mirror}, have not appeared in the literature before (although similar ideas may be known to some experts). In particular, the orbifold theta functions studied in \cite{You24} are included as special cases.  
    \item A relative mirror theorem of a pair $(X,D)$. We prove a relative mirror theorem for simple normal crossings (snc) pairs generalizing the result in \cite{TY20b} by removing the nefness assumption on the divisors. This is a non-trivial generalization that requires a different method. The relative mirror theorem (Corollary \ref{cor-extended-rel-I-func}) gives a relation between the genus zero Gromov--Witten invariants of the snc pair $(X,D)$  and the genus zero Gromov--Witten invariants (the $J$-function) of the Fano $X$. This serves as a bridge that connects the integrals of functions on the mirror (the LHS of (\ref{conj-homo})) to the $J$-function and the Gamma class of the Fano variety $X$ (the RHS of (\ref{conj-homo})).  
\end{itemize}

We explain these ingredients in more detail in the following sections. These results are of independent interest. Mirror symmetric Gamma conjecture may be considered as the first major application of these results. We expect more applications of these techniques in mirror symmetry and Gromov--Witten theory.

\subsection{Relative mirror symmetry}

Following \cite{Givental98}, the mirror of a Fano variety $X$ is a Landau--Ginzburg model $(\check{X}, W)$, where $W:\check{X}\rightarrow \mathbb C$ is called the superpotential. Mirror symmetry for Fano varieties usually requires the existence of anticanonical divisors. The mirror of a Fano variety depends on the choice of an anticanonical divisor. Therefore,  instead of considering mirror symmetry for only the Fano variety $X$ itself, we consider relative mirror symmetry of a log Calabi--Yau pair $(X,D)$, where $D=\sum_{i=1}^m D_i\in |-K_X|$ is an snc anticanonical divisor of $X$. 

We first assume that every stratum of $D$ contains a zero-dimensional stratum. Following the intrinsic mirror symmetry of \cite{GS19}, and a variant of it in \cite{TY20c} using orbifold invariants of the root stacks, the mirror of the pair $(X,D)$ is a Landau--Ginzburg model $(\check{X},W)$, where $\check{X}$ is the spectrum of the degree zero part of the relative quantum cohomology of the pair $(X,D)$. In fact, $\check{X}$ is considered the mirror of the complement $X\setminus D$. In Section 0.4 of the first arXiv version of \cite{GHK}, the superpotential is
\[
W:=\sum_{i=1}^m \vartheta_{[D_i]},
\]
where theta functions $ \vartheta_{[D_i]}$'s are global functions on $\check{X}$ which are certain elements of the canonical basis of the degree zero part of the relative quantum cohomology (the mirror algebra).

If $D$ does not have zero-dimensional strata, one needs to consider a maximally unipotent monodromy degeneration of the log Calabi--Yau manifold $X\setminus D$. The intrinsic mirror $\check{X}$ of $X\setminus D$ can be constructed using this degeneration. This is called the relative case in \cite{GS19}. We will focus on the case where $D$ has zero-dimensional strata and explain how the computation works in the relative case in the end of this paper.

\subsection{Functions on the mirror}

Theta functions through orbifold invariants are defined for smooth pairs  in \cite{You22}  and for snc pairs in \cite{You24}. Logarithmic theta functions are considered in \cite{GS21} and \cite{GRZ}, but we will not use them here.   

The theta function $\vartheta_{\vec s}$ corresponds to the degree zero class $[1]_{\vec s}\in H^0(D_{\vec s})$, where $\vec s\in B(\mathbb Z)$. Theta functions satisfy the product rule that comes from the restriction of the relative quantum product to the degree zero part. In Section \ref{sec:theta-func}, we consider a more general definition of theta functions that satisfy the product rule of the relative quantum cohomology with a parameter $\tau$. 
As in \cite{You24}, orbifold theta functions are defined using mid-age invariants of root stacks.
\begin{definition}[= Definition \ref{def-theta-tau}]
For an snc log Calabi--Yau pair $(X,D)$, let $\tau\in \mathfrak H$ be a cohomology class in the state space $\mathfrak H$ of the Gromov--Witten theory of the pair $(X,D)$ with $\deg^0(\tau)\leq 1$.  The theta functions are defined as follows. Fix $\vec s \in B(\mathbb Z)\setminus \{0\}$. Let $\sigma_{\on{max}}\in \Sigma(X)$ be a maximal cone of $\Sigma(X)$ such that $p\in \sigma_{\on{max}}$, then
    \begin{align*}
   \vartheta_{\tau,\vec s}(p)&:=
  \sum_{\beta} \sum_{\vec k\in {\mathbb Z}^m} N_{\tau,\vec s, \vec{\mathbf b}, -\vec{\mathbf{b}}+\vec k}^\beta t^{\beta} x^{\vec k},
    \end{align*}
with
\[
N_{\tau,\vec s, \vec{\mathbf b}, -\vec{\mathbf{b}}+\vec k}^\beta=\sum_{l\geq 0}\frac{1}{l!}\langle [1]_{\vec s},\tau,\ldots,\tau, [1]_{\vec{\mathbf b}}, [\on{pt}]_{-\vec{\mathbf{b}}+\vec k}\rangle_{0,l+3,\beta}^{(X,D)},
\]
and
\[
 x^{\vec k}=x_1^{k_1}\cdots x_m^{k_m},
\]
where the last two markings are mid-age markings along $D_i$ for $i\in I_\sigma$ for $\sigma=\sigma_{\on{max}}$. 
\end{definition}

In Definition \ref{def-theta-tau-mirror}, we let $\tau$ be the relative mirror map. We use it to define theta functions $\tilde{\vartheta}_{\tau,\vec p}$ that encode the relative mirror map in the definition of theta functions. We call these functions theta functions with quantum corrections. 

The more general version of the mirror symmetric Gamma conjecture (\ref{conj-homo}) also involves the mirror of a class $\varphi\in H^*(X)$. Therefore, we also define more general functions on the mirror.
\begin{definition}[=Definition \ref{def-hat-varphi}]
For an snc log Calabi--Yau pair $(X,D)$, fix $[\varphi]_{\vec s}\in \mathfrak H_{\vec s}$. Let $\sigma_{\on{max}}\in \Sigma(X)$ be a maximal cone of $\Sigma(X)$ such that $p\in \sigma_{\on{max}}$, then
    \begin{align}
   \check{\varphi}_{\tau,\vec s}(p)&:=
  \sum_{\beta} \sum_{\vec k\in {\mathbb Z}^m, l\geq 0}  \frac{1}{l!}\langle [\varphi]_{\vec s},\tau,\cdots,\tau,[1]_{\vec{\mathbf b}}, [\on{pt}]_{-\vec{\mathbf{b}}+\vec k}\rangle_{0,3+l,\beta} t^{\beta}x^{\vec k},
    \end{align}
where 
\[
\tau=\tau_{0,2}+\tau^\prime\in \mathfrak H \text{ and } \tau_{0,2}\in H^2(X)\subset \mathfrak H_{\vec 0}=H^*(X);
\]
\[
\vec d:=(D_1\cdot\beta,\cdots,D_m\cdot\beta);
\]
the invariants here are orbifold invariants with mid-ages along the divisors $D_{\mathbf b_i}$ such that $D_{\vec {\mathbf b}}=D_{\sigma_{\on{max}}}$ is a lowest dimensional stratum in $D_{\vec k}$.
\end{definition}

We then explore the properties of these functions in Section \ref{sec:func-mirror}.

\begin{remark}
We remark that $\check{\varphi}$ is not exactly $\check{\varphi}(z)$ in the mirror symmetric Gamma conjecture. Let $\varphi(z)$ be the mirror map associated with $\varphi$ in the relative mirror theorem of $(X,D)$. Then $\varphi(z)$ is a polynomial in $z$ and the coefficients take values in the state space $\mathfrak H$ of the pair $(X,D)$. The function $\check{\varphi}(z)$ is defined as the mirror function for $\varphi(z)$. 
\end{remark}

\subsection{Relative mirror theorem}

A mirror theorem for an snc pair $(X,D)$ was proved in \cite{TY20b} with the assumption that all $D_i$'s are nef. In many cases, we need to be able to compute orbifold invariants of $(X,D)$ without this assumption. In particular, a general version of the relative mirror theorem is necessary in order to prove the mirror symmetric Gamma conjecture via relative mirror symmetry.

We start with the case where $D$ is smooth. Let 
\[
P:=\mathbb P(\mathcal O_X(-D)\oplus \mathcal O_X)
\]
 and $X_\infty\subset P$ be the infinity divisor. Let
\[
p: P_{X_\infty,r} \rightarrow P
\]
be the $r$-th root stack of $P$ along $X_\infty$.  The orbifold/relative mirror theorem of \cite{FTY} uses a hypersurface construction to consider root stacks as a hypersurface in the toric stack bundle $P_{X_\infty,r}$, then apply the mirror theorem for toric stack bundles \cite{JTY} and the orbifold quantum Lefschetz principle \cite{Tseng}. This requires the nefness assumption on the divisor. 

However, the orbifold quantum Lefschetz principle \cite{Tseng} cannot be applied when $D$ is not nef and requires a different approach to prove the mirror formula. The new idea presented in Section \ref{sec:rel-mirror-deg} is to degenerate both $X_{D,r}$ and $P_{X_\infty,r}$. We then use the degeneration formula to study the equivariant twisted Gromov--Witten invariants of toric stack bundles and then obtain a mirror theorem for root stacks without the nefness assumption on the divisors. 

When considering the degeneration of $X$ to the normal cone of $D$, we have the $\mathbb P^1$-bundle $$Y:=\mathbb P(N_{D/X}\oplus \mathcal O_D)$$ and its zero divisor $D_0\subset Y$ such that $X$ degenerates into $X\cup_{D_0}Y$. Let $I_{(Y,D_0)}(\tau,z)$ be the $I$-function of $(Y,D_0)$ obtained as a limit of the $I$-function of the orbifold $Y_{D_0,r_0}$. The $I$-function of the orbifold $Y_{D_0,r_0}$ is defined in \cite{JTY} since $Y_{D_0,r_0}$ is a weighted-$\mathbb P^1$-bundle over $D$. Unlike the nef case, we need to include the possibly non-trivial mirror map between $J_{(Y,D_0)}$ and $I_{(Y,D_0)}$. This mirror map is the same as the mirror map between the equivariant $J$-function and the equivariant $I$-function of $N_{D/X}$. We refer to this as the mirror map in $D$. This is trivial when $D$ is nef, so we do not see it in \cite{FTY}. 

Then we prove a mirror theorem for root stacks in Theorem \ref{thm-quan-lef}.
Using the iterative construction in \cite{TY20b}, we can see that the result also holds for snc pairs and we obtain the $I$-function for the multi-root stack $X_{D,\vec r}$.
\if{
\begin{corollary}[=Corollary \ref{cor-extended-orb-I-func}]\label{intro-cor-orb-I-func}
The $S$-extended $I$-function 
\[
I_{X_{D,\vec r},\on{amb}}^S(Q,q, x_{\vec a},\tau,z)
\]
 for $X_{D,\vec r}\subset P_{m,X_{\infty},\vec r}$ is
\begin{equation}
\begin{split}
&e^{\sum_{i=1}^m h_i\log q_i/z}\sum_{\beta\in \on{NE}(X)} \sum_{(l_1,\ldots,l_k)\in (\mathbb Z_{\geq 0})^k}\frac{J_{X, \beta}(\tau,z)Q^{\beta}q^{D\cdot\beta}\prod_{i=1}^kx_{\vec a_i}^{l_i}}{z^{\sum_{i=1}^k l_i}\prod_{i=1}^k (l_i!)}\\
&\cdot\prod_{j=1}^m\frac{\prod_{ a\leq D_j\cdot\beta}(D_j+az)\prod_{\langle a \rangle =\langle (D_j\cdot\beta-\sum_{i=1}^k l_ia_{ij})/r_j\rangle, a\leq 0}(D_j/r_j+az)}{\prod_{ a\leq 0}(D_j+az)\prod_{\langle a \rangle =\langle (D_j\cdot\beta-\sum_{i=1}^k l_ia_{ij})/r_j\rangle, a\leq (D_j\cdot\beta-\sum_{i=1}^k l_ia_{ij})/r_j}(D_j/r_j+az)}\\
& \cdot[1]_{( (-D_1\cdot\beta+\sum_{i=1}^k l_ia_{i1})/r_1,\ldots, (-D_m\cdot\beta+\sum_{i=1}^k l_ia_{im})/r_m)}\cup_{i=1}^m h_i,
\end{split}
\end{equation}
where we refer to Corollary \ref{cor-extended-orb-I-func} for the notation.
The $S$-extended $I$-function $I_{X_{D,r},\on{amb}}^S(Q,q,x_{\vec a},\tau,z)$, after applying the mirror map in $D$ , lies in $\iota_*\mathcal L_{X_{D,r}}$, where $\mathcal L_{X_{D,r}}$ is Givental's Lagrangian cone of $X_{D,r}$.
\end{corollary}
}\fi
Taking the large $\vec r$-limit, we obtain a relative mirror theorem for the snc pair $(X,D)$. 

Let 
\[
S=\{\vec a_1,\ldots, \vec a_k\},
\]
where $\vec a_i\in (\mathbb Z_{\geq 0})^m$. 
The $S$-extended $I$-function 
\[
I_{(X,D)}^S(Q,q,x_{\vec a},\tau_D,z)
\]
for $(X,D)$ is
\begin{equation}\label{Intro-I-func}
\begin{split}
&e^{\sum_{i=1}^m D_i\log q_i/z}\sum_{\beta\in \on{NE}(X)} \sum_{(l_1,\ldots,l_k)\in (\mathbb Z_{\geq 0})^k}\frac{J_{X, \beta}(\tau_D,z)Q^{\beta}q^{D\cdot\beta}\prod_{i=1}^kx_{\vec a_i}^{l_i}}{z^{\sum_{i=1}^k l_i}\prod_{i=1}^k (l_i!)}\\
&\cdot\prod_{j=1}^m\frac{\prod_{ a\leq D_j\cdot\beta}(D_j+az)}{\prod_{ a\leq 0}(D_j+az)}\prod_{j: D_j\cdot\beta-\sum_{i=1}^k l_ia_{ij}>0}\frac{1}{(D_j+(D_j\cdot\beta-\sum_{i=1}^k l_ia_{ij})z)}\\
& \cdot [1]_{ (-D_1\cdot\beta+\sum_{i=1}^k l_ia_{i1},\ldots, -D_m\cdot\beta+\sum_{i=1}^k l_ia_{im})}.
\end{split}
\end{equation}

When $X$ is Fano, we have
\begin{corollary}[=Corollary \ref{cor-extended-rel-I-func}]\label{intro-cor-rel-I-func}
The $S$-extended $I$-function $I_{(X,D)}^S(Q,q,x_{\vec a},\tau,z)$ for $(X,D)$ is
\begin{equation}\label{intro-extended-rel-I-func-Fano}
\begin{split}
&e^{\sum_{i=1}^m D_i\log q_i/z}\sum_{\beta\in \on{NE}(X)} \sum_{(l_1,\ldots,l_k)\in (\mathbb Z_{\geq 0})^k}\frac{J_{X, \beta}(\tau,z)Q^{\beta}q^{D\cdot\beta}\prod_{i=1}^kx_{\vec a_i}^{l_i}}{z^{\sum_{i=1}^k l_i}\prod_{i=1}^k (l_i!)}\\
&\cdot\prod_{j=1}^m\frac{\prod_{ a\leq D_j\cdot\beta}(D_j+az)}{\prod_{ a\leq 0}(D_j+az)}\prod_{j: D_j\cdot\beta-\sum_{i=1}^k l_ia_{ij}>0}\frac{1}{(D_j+(D_j\cdot\beta-\sum_{i=1}^k l_ia_{ij})z)}\\
& \cdot [1]_{ (-D_1\cdot\beta+\sum_{i=1}^k l_ia_{i1},\ldots, -D_m\cdot\beta+\sum_{i=1}^k l_ia_{im})}.
\end{split}
\end{equation}
The $S$-extended $I$-function $I_{(X,D)}^S(Q,q,x_{\vec a},\tau,z)$ lies in  Givental's Lagrangian cone  $\mathcal L_{(X,D)}$  of $(X,D)$.
\end{corollary}

\begin{remark}
In general, the relative mirror theorem involves the mirror maps in $D_i$ in the degeneration formula, which we do not describe explicitly in general. There are nice cases where these mirror maps are trivial, and the $I$-functions are simply (\ref{Intro-I-func}). When $X$ is Fano and $\tau\in H^2(X)$, then the mirror maps in $D_i$ are trivial. Although there can still be relative mirror maps by examining the coefficients of the relative $I$-functions. 
In this case, we can also write the $J$-function of $X$ in terms of invariants of $(X,D)$ and hypergeometric modifications coming from $D$. This is stated in Theorem \ref{thm-mirror-X} which may be considered as a mirror theorem for Fano varieties.

In this paper, we mainly consider the Fano case. More details of other cases will be studied elsewhere. See for example, \cite{You25b}.
\end{remark}

\begin{remark}
Based on our limited knowledge, we do not know about a similar mirror theorem of this form that is obtained from degeneration. However, it should not be surprising that the degeneration technique will be useful for the mirror theorem for relative invariants. For example, to include relative invariants with negative contact orders, we had to use the degeneration and rubber invariants of $Y$ \cite{FWY}. This is also related to the conjecture about the relation between degeneration and mirror symmetry \cite{DHT},\cite{DKY} and \cite{DKY21}. We will explore it further in the future.
\end{remark}

\subsection{Mirror symmetric Gamma conjecture}

With these ingredients that we developed above, we first prove the mirror symmetric Gamma conjecture for $V=\mathcal O_{\on{pt}}$.

\begin{theorem}[=Theorem \ref{thm-gamma-conj-pt}]\label{thm-gamma-pt}
Let $X$ be a Fano variety with an snc anticanonical divisor $D$ such that the dual intersection complex $B$ is pure-dimensional with $\dim_{\mathbb R} B=\dim_{\mathbb C} X$. Let $\tau=\sum_{i=1}^m [1]_{\vec e_i}+ \tau_{0,2}$ and $\tau_{0,2}\in H^2(X)$. Mirror symmetric Gamma conjecture for $\mathcal O_{\on{pt}}$ holds:

\begin{align}\label{intro-gamma-conj-pt}
\int_{\Gamma_{c}}\check{\varphi}_\tau(-z) e^{-W_{\tau}/z}\omega=\int_X \left(z^{c_1}z^{\deg/2} \mathbb{J}_{X}(\tau_{0,2},-z)\varphi\right)\cup \hat{\Gamma}_X \on{Ch}(\mathcal O_{\on{pt}}),
\end{align}
where $\Gamma_c$ is the compact cycle which is the homological mirror of $\mathcal O_{\on{pt}}$.
\end{theorem}

As $\check{\varphi}_\tau(-z)$ is not quite explicit in general, one may prefer to state a clearer identity for Theorem \ref{thm-gamma-pt} when there is no class $\varphi$:
\begin{align}\label{intro-gamma-conj-pt-no-phi}
\int_{\Gamma_{c}} e^{-W_{\tau}/z}\omega=\int_X \left(z^{c_1}z^{\deg/2} J_{X}(\tau_{0,2},-z)\right)\cup \hat{\Gamma}_X \on{Ch}(\mathcal O_{\on{pt}}).
\end{align}

After regularization, we also obtain the following result between regularized quantum periods and classical periods generalizing \cite{TY20b}*{Theorem 1.14}. 
\begin{theorem}[=Theorem \ref{theorem-theta=x}]
Let $X$ be a Fano variety and $D\subset X$ be an snc anticanonical divisor of $X$ such that $B$ is pure-dimensional with $\dim_{\mathbb R}B=\dim_{\mathbb C}X=n$. Then the regularized quantum period of $X$ coincides with the classical period of $W$. 
\end{theorem}

We then also prove the mirror symmetric Gamma conjecture for Fano varieties for the structure sheaf $\mathcal O_X$. We first state it without the $\varphi$-class:

\begin{theorem}\label{thm-gamma-struc-sheaf}
Let $X$ be a Fano variety with an snc anticanonical divisor $D$ such that the dual intersection complex $B$ is pure-dimensional with $\dim_{\mathbb R} B=\dim_{\mathbb C} X$. Let $\tau=\sum_{i=1}^m [1]_{\vec e_i}+ \tau_{0,2}$ and $\tau_{0,2}\in H^2(X)$. Mirror symmetric Gamma conjecture for $X$ holds for $\mathcal O_X$:
\begin{align}\label{gamma-conj-struc-sheaf}
\int_{\Gamma_{\mathbb R}} e^{-W_{\tau}/z}\omega=\int_X \left(z^{c_1}z^{\deg/2} J_{X}(\tau_{0,2},-z)\right)\cup \hat{\Gamma}_X.
\end{align}
\end{theorem}

When we include the class $\varphi$ we have

\begin{theorem}\label{thm-gamma-struc-sheaf-phi}
Let $X$ be a Fano variety with an snc anticanonical divisor $D$ such that the dual intersection complex $B$ is pure-dimensional with $\dim_{\mathbb R} B=\dim_{\mathbb C} X$.  Mirror symmetric Gamma conjecture for $X$ holds for $\mathcal O_X$:
\begin{align}\label{gamma-conj-struc-sheaf-phi}
\int_{\Gamma_{\mathbb R}}\check{\varphi}_\tau(-z) e^{-W_{\tau}/z}\omega=\int_X \left(z^{c_1}z^{\deg/2} \mathbb{J}_{X}(\tau_{0,2},-z)\varphi\right)\cup \hat{\Gamma}_X.
\end{align}
\end{theorem}

\begin{remark}
In Identity (\ref{gamma-conj-struc-sheaf}), we have $V=\mathcal O_X$ and the real cycle $\Gamma_{\mathbb R}$ is the homological mirror to $\mathcal O_X$. We can also allow $V$ to be a direct sum of line bundles. Then Identity (\ref{conj-homo}) should follow from Identity (\ref{gamma-conj-struc-sheaf}) by the monodromy transform as in \cite{Iritani09}, \cite{Iritani11}.
\end{remark}

\begin{remark}
We note that although the LHS of (\ref{gamma-conj-struc-sheaf}) is defined using the information of the divisor $D$,  the RHS of (\ref{gamma-conj-struc-sheaf}) does not depend on $D$. 
\end{remark}

\begin{remark}
    The function $\check{\varphi}_\tau(-z)$ involves the mirror map associated with the class $\varphi$. The mirror map can be very complicated in general, and we do not compute it explicitly. 
\end{remark}

\begin{remark}
Suppose a Fano variety $Y$ is a complete intersection in a Fano variety $X$ that has a divisor $D$ satisfying the assumption of Theorem \ref{thm-gamma-struc-sheaf}. Then we first apply  computation to prove the mirror symmetric Gamma conjecture for $X$ under this set-up. Then the mirror symmetric Gamma conjecture also holds for $Y$ via a Laplace transform as explained in \cite{Iritani11} and \cite{Iritani23gamma}.  
\end{remark}

When the relative mirror map is not trivial, we can define theta functions that include the relative mirror map. 
When computing the oscillatory integrals, we need to pass from invariants of $(X,D)$ to invariants of $X$ which requires the relative mirror map. Therefore, if we additionally impose quantum corrections in the definition of the theta functions, we will not need an additional change of variables in the formula. This is discussed in Section \ref{sec:mirror-map}.


In Section \ref{sec:relative-case}, when $D$ does not have zero-dimensional strata, we consider a maximally unipotent monodromy degeneration of $(X, D)$. We explain how the computation also works in this setting.

\begin{thm}\label{thm-mum-deg}
If A Fano variety $X$ with an snc anticanonical divisor $D$ admits a maximally unipotent monodromy degeneration, then the mirror symmetric Gamma conjecture for $X$ holds for $\mathcal O_{\on{pt}}$ and $\mathcal O_X$.
\end{thm}

\begin{remark}
To prove the mirror symmetric Gamma conjecture with a class $\varphi\in H^*(X)$, we also need to consider the higher degree part of the relative quantum cohomology. At this moment, under the current approach, one may not replace orbifold Gromov--Witten invariants of root stacks by punctured logarithmic Gromov--Witten invariants because only the degree zero part of the relative quantum cohomology is considered in \cite{GS19}. Of course, it is expected that one can define the higher-degree part of relative quantum cohomology using punctured logarithmic Gromov--Witten invariants. Another important tool is a relative mirror theorem in Section \ref{sec:rel-mirror-deg} and \cite{TY20b} that relates invariants of the pair $(X,D)$ with the $J$-function of the Fano variety $X$. One may also expect a similar mirror theorem for punctured logarithmic invariants. Then, one can try to use a similar approach to prove mirror symmetric Gamma conjecture via intrinsic mirror symmetry considered in \cite{GS19}. 
\end{remark}

\begin{remark}
We also would like to remark that, besides borrowing the mirror construction and theta functions from the Gross--Siebert program, the techniques and tools in this article is very different from those that we normally see in the Gross--Siebert program. Hence, one may prefer not to think about our approach as a part of the Gross--Siebert program. One can try to use completely different techniques and tools to prove the mirror symmetric Gamma conjecture via the actual intrinsic mirror symmetry. However, this is not the scope of this paper.  
\end{remark}

\section*{Acknowledgement}

The author would like to thank Hiroshi Iritani for numerous inspiring discussions, detailed explanations of the Gamma conjecture, and helpful comments on the drafts of this paper. The author also thanks Bohan Fang, Mark Gross, Weiqiang He, Jianxun Hu, Huazhong Ke, Changzheng Li and Zhitong Su for helpful discussions. 

Part of the work was done during the Workshop on Mirror symmetry and Related Topics, Kyoto 2024. The author would like to thank Kyoto University and Hiroshi Iritani for their hospitality. 

The author would also like to thank Andreas Malmendier, Michael T. Schultz and Longting Wu for bringing mirror symmetric Gamma conjecture to the author's attention during private communications. 

This project has received funding from the European Union’s Horizon 2020 research and innovation programme under the Marie Skłodowska -Curie grant agreement 101025386.

\section{Orbifold Gromov--Witten theory of snc pairs}
We give a brief review of the orbifold Gromov--Witten theory of multi-root stacks $X_{D,\vec r}$ and the relative mirror theorem when $D_i$ are nef. 

\subsection{Orbifold Gromov--Witten theory of root stacks }\label{sec:GW-multi-root}
The large $\vec r$-limit is called the formal Gromov--Witten theory of infinite root stacks $X_{D,\infty}$ in \cite{TY20c}. We will simply call it the orbifold Gromov--Witten theory of the pair $(X,D)$.

Let $X$ be a smooth projective variety and
\[
D=D_1+\ldots+D_m
\]
be an snc divisor with irreducible components $D_1,\ldots, D_m$. For any index set $I\subseteq\{1,\ldots, m\}$, we define 
\[
D_I:=\cap_{i\in I} D_i.
\]
Let 
\[
\vec k=(k_1,\ldots,k_m)\in \mathbb Z^m.
\]
We define
\[
I_{\vec k}:=\{i:k_i\neq 0\}\subseteq \{1,\ldots,m\}.
\]

For 
\[
\vec r=(r_1,\ldots,r_m)\in (\mathbb Z_{\geq 0})^m
\]
with $r_i$'s pairwise coprime, we consider the multi-root stack
\[
X_{D,\vec r}:=X_{(D_1,r_1),\ldots,(D_m,r_m)}.
\]
Let $$\overline{M}_{0,\{\vec k^j\}_{j=1}^l, \beta}(X_{D,\vec r})$$ be the moduli space of genus zero orbifold stable maps to $X_{D,\vec r}$ with degree $\beta\in H_2(X)$ and $l$-markings where the ages of the markings are determined by $\{\vec k^j\}_{j=1}^l$ as follows.
\begin{itemize}
    \item The vectors
\[
\vec k^j=(k_1^j,\ldots,k_m^j)\in (\mathbb Z)^m, \text{ for } j=1,2,\ldots,l,
\]
satisfy the following condition:
\[
\sum_{j=1}^l k_i^j=\int_\beta[D_i], \text{ for } i\in\{1,\ldots,m\}.
\]  
We also use the notation
\[
\sum_{j=1}^l\vec k^j=D\cdot \beta=(D_1\cdot\beta,\ldots,D_m\cdot\beta).
\]
These vectors are used to record contact orders of the markings with respect to divisors $D_1,\ldots,D_m$. In the language of orbifold Gromov--Witten theory,
the $j$-th marking maps to the twisted sector $D_{I_{\vec s^j}}$ of the inertia stack of $X_{D,\vec r}$ with age
\[
\sum_{i: k^j_i>0} \frac{k^j_i}{r_i}+\sum_{i: k^j_i<0} \left(1+\frac{k^j_i}{r_i}\right),
\]
see \cite{TY20c}. 
\item Define
\[
k_{i,-}:=\#\{j: k_i^j<0\}, \text{ for } i=1,2,\ldots, m,
\]
to be the number of markings that have negative contact orders with the divisor $D_i$.
\end{itemize}
 By \cite{TY20c}*{Corollary 16}, the virtual cycle $$[\overline{M}_{0,\{\vec k^j\}_{j=1}^l, \beta}(X_{D,\vec r})]^{\on{vir}}$$ of genus zero orbifold Gromov-Witten theory of $X_{D,\vec r}$, after multiplying by $\left(\prod_{i=1}^m r_i^{k_{i,-}}\right)$, is independent of $r_i$ for $r_i$ sufficiently large. Following \cite{TY20c}*{Definition 18}, the formal Gromov-Witten invariants of $X_{D,\infty}$ are defined as limits of the corresponding genus zero orbifold Gromov-Witten invariants of $X_{D,\vec r}$.

Let \begin{itemize}
    \item $\gamma_j\in H^*(D_{I_{\vec k^j}})$, for $j\in\{1,2,\ldots,l\}$;
    \item $a_j\in \mathbb Z_{\geq 0}$, for $j\in \{1,2,\ldots,l\}$.
\end{itemize}
The formal genus zero orbifold Gromov-Witten invariants of $X_{D,\infty}$ are defined as
\begin{align}\label{inv-X-D}
    \left\langle [\gamma_1]_{\vec k^1}\bar{\psi}^{a_1},\ldots, [\gamma_l]_{\vec k^l}\bar{\psi}^{a_l} \right\rangle_{0,\{\vec k^j\}_{j=1}^l,\beta}^{X_{D,\infty}}:=\left(\prod_{i=1}^m r_i^{k_{i,-}}\right)\left\langle \gamma_1\bar{\psi}^{a_1},\ldots, \gamma_l\bar{\psi}^{a_l} \right\rangle_{0,\{\vec k^j\}_{j=1}^l,\beta}^{X_{D,\vec r}}
\end{align}
for sufficiently large $\vec r$.

The genus zero invariant (\ref{inv-X-D}) is zero unless it satisfies the virtual dimension constraint:
\begin{align}\label{vir-dim}
    \dim_{\mathbb C}X-3+l+\int_\beta c_1(T_X)-\int_\beta [D]=\sum_{j=1}^l \deg^{0}([\gamma_j]_{\vec k^j})+\sum_{j=1}^l a_j,
\end{align}
where if $\alpha\in \mathfrak H_{\vec k}$ is a cohomology class of real degree $d$, then,
\begin{align}\label{deg-0}
\deg^{0}([\alpha]_{\vec k})=d/2+\#\{i:k_i<0\}.
\end{align}

We will simply use the notation $\langle \cdots \rangle^{(X,D)}$ instead of $\langle \cdots\rangle^{X_{D,\infty}}$ when there is no ambiguity. As the contact orders $\{\vec k^j\}$ are specified in the insertions, we will also sometimes use the notation $\langle \cdots \rangle_{0,l,\beta}^{(X,D)}$ instead of $\langle \cdots \rangle_{0,\{\vec k^j\}_{j=1}^l,\beta}^{(X,D)}$. When the target and the topological data are clear, we may sometimes omit the superscripts and subscripts in the notation.

We also recall some notation for an snc pair $(X,D)$. Let $S$ be the dual intersection complex of $D$. That means $S$ is the simplicial complex with vertices $v_1,\ldots, v_m$ and simplices $\langle v_{i_1},\ldots, v_{i_p}\rangle$ corresponding to nonempty intersections $D_{i_1}\cap\cdots\cap D_{i_p}$. Let $B$ be the cone over $S$, $\Sigma(X)$ be the induced simplicial fan in $B$ and, $B(\mathbb Z)$ be the set of integer points of $B$. For a cone $\sigma\in \Sigma(X)$, we write $D_\sigma$ for the corresponding stratum of $D$. We write $i\in I_\sigma\subset \{1,\ldots,m\}$ if $D_i\supset D_\sigma$. Therefore, $D_\sigma=\cap_{i\in I_\sigma} D_i$.

\subsection{Mid-age invariants}

We consider the moduli space $$\overline{M}_{0,\{\vec k^j\}_{j=1}^l, \vec k_a,\vec k_b,\beta}(X_{D,\vec r})$$ of genus zero orbifold stable maps to $X_{D,\vec r}$ with degree $\beta\in H_2(X)$ and $(l+2)$-markings where ages of the first $l$ markings are determined by $\{\vec k^j\}_{j=1}^l$ as before. The last two markings are mid-ages along some of $D_i$ and have large or small ages (including age $0$) along other $D_i$. Having a pair of markings with mid-ages along a divisor $D_i$ means the following. Let  $r_i\gg |D_i\cdot\beta|$ and $ k_{ai}, k_{bi}\gg |D_i\cdot\beta|$  such that
\[
k_{ai}+k_{bi}=r_i+c_i
\]
for some constant $c_i\in \mathbb Z$. The ages $k_{ai}/r_i$ and $k_{bi}/r_i$ are mid-ages.

By \cite{You24}*{Theorem 5.9}, the genus zero cycle class for the orbifold Gromov--Witten theory of $(X,D)$ with a pair of mid-ages, after multiplying suitable powers of $r_i$, does not depend on the value of $r_i$, $k_{ai}$ and $k_{bi}$ when they are sufficiently large.
   Then the genus zero cycle class for the orbifold Gromov--Witten theory of $(X,D)$ with a pair of mid-ages can be denoted by
    \[
\left[\overline{M}_{0,\{\vec k^j\}_{j=1}^l, \vec{\mathbf{b}},-\vec{\mathbf{b}}+\vec c,\beta}(X,D)\right]^{\on{vir}}.
\]
Mid-age invariants are defined using this cycle. Again, we will use the notation $\langle \cdots\rangle^{(X,D)}$ or $\langle \cdots \rangle^{X_{D,\infty}}$ to denote the orbifold invariants of snc pairs with mid-ages.

\begin{remark}
    We briefly recall how mid-age invariants fit into the relative-orbifold correspondence. In genus zero, the relative-orbifold correspondence \cite{ACW}, \cite{TY18}, \cite{FWY} states that relative invariants of a smooth pair $(X,D)$ are equal to orbifold invariants of the root stack $X_{D,r}$ for $r$ sufficiently large. A relative marking with positive contact order $k\in \mathbb Z_{\geq 0}$, corresponds to an orbifold marking with age $k/r$. A relative marking with negative contact order $k\in \mathbb Z_{<0}$ corresponds to an orbifold marking with age $(r+k)/r$. When $r$ is sufficiently large, $k/r$ is sufficiently close to 0 and is called a small age; $(r+k)/r$ is sufficiently close to $1$ and is called a large age. However, there are more orbifold invariants of $X_{D,r}$ that are not covered in this correspondence. In particular, we can have a pair of orbifold markings with ages $k_1/r$ and $k_2/r$ such that $k_1+k_2=r+a$ for a fixed integer $a$ and $k_1\geq k_2 \gg |D\cdot \beta|$. In particular, we can have $k_1/r$ and $k_2/r$ sufficiently close to $1/2$ when $r$ is sufficiently large. These invariants are called mid-age invariants and are studied in \cite{You21} and \cite{You24}.
\end{remark}

\subsection{Root stacks as hypersurfaces}

Following \cite{FTY} and \cite{TY20b}, (multi)-root stacks are complete intersections in toric stack bundles. We briefly review the construction. Let $X$ be a smooth projective variety and $D$ be a smooth divisor of $X$. Let 
\[
P:=\mathbb P(\mathcal O_X(-D)\oplus \mathcal O)
\]
 and $X_\infty\subset P$ be the infinity divisor, that is, the divisor we add at infinity to compactify $\mathcal O_X(-D)$. Let $P_{X_\infty,r}$ be the $r$-th root stack of $P$ along $X_\infty$ and 
\[
p:P_{X_\infty,r}\rightarrow P
\]
 be the projection. Then $X_{D,r}$ is the zero locus of a section of $p^*\mathcal O_P(1)$. 

We consider the $\mathbb C^\times$-action on $P$ such that $\mathbb C^\times$ acts on $\mathcal O_X(-D)$ with weight $-1$ and acts trivially on $\mathcal O$. It induces an action on $P_{X_{\infty},r}$. The $\mathbb C^\times$-action naturally lifts to a $\mathbb C^\times$-action on the line bundle $p^*\mathcal O_P(1)$. In \cite{FTY}, we applied the orbifold quantum Lefschetz theorem to show that the non-equivariant limits of the genus zero Gromov--Witten invariants twisted by the equivariant line bundle $p^*\mathcal O_P(1)$ coincide with the genus zero orbifold Gromov--Witten invariants of $X_{D,r}$ assuming that $D$ is nef: 

\[
\left\langle \prod_{i=1}^l \tau_{a_i}(\iota^*\gamma_i)\right\rangle_{0,l,\beta}^{X_{D,r}}=\left[\left\langle \prod_{i=1}^l \tau_{a_i}(\gamma_i^{\mathbb C^\times})\right\rangle_{0,l,\beta+kf}^{P_{X_\infty,r},p^*\mathcal O_P(1)}\right]_{\lambda=0},
\]
where $\gamma_i\in H^*_{\on{CR}}(P_{X_\infty,r})$; $a_i\in \mathbb Z_{\geq 0}$, for $1\leq i\leq l$; $k=D\cdot\beta$, $\iota:X_{D,r}\hookrightarrow P_{X_\infty,r}$ is the inclusion, and $\lambda$ is the equivariant parameter. This gives the $I$-function for root stacks as the non-equivariant limit of twisted $I$-function of $P_{X_\infty,r}$.  The $I$-function for the pair $(X,D)$ is taken as a limit of the $I$-function for the root stack $X_{D,r}$ as $r\rightarrow \infty$. 

In Section \ref{sec:rel-mirror-deg}, we will generalize this relation to the case where $D$ is not necessarily nef. Therefore, we can consider the genus zero Gromov--Witten theory of $P_{X_{\infty},r}$ twisted by the equivariant line bundle $p^*\mathcal O_P(1)$ as the genus zero equivariant Gromov--Witten theory of the pair $(X,D)$ (with ambient insertions). 

Using the iterative construction in \cite{TY20b}, we generalized this construction to snc pairs. Now, let $X$ be a smooth projective variety, and let $D\subset X$ be an snc divisor such that
\[
D=D_1+D_2+\ldots+D_m,
\]
where $D_i\subset X$ is smooth and irreducible. We consider the following tower:
\[
P_m\rightarrow \ldots \rightarrow P_2\rightarrow P_1\rightarrow X,
\]
where $\pi_{i+1}: P_{i+1}\rightarrow P_i$ is a $\mathbb P^1$-bundle, with 
\[
P_{i+1}:=\mathbb P(\mathcal O_X(-D_i)\oplus \mathcal O_{P_i})
\]
 and we have omitted various pull-backs from the notation. Let 
\[
h_i=c_1(\mathcal O_{P_i}(1)).
\]
Let
\[
X_{i\infty}:=\mathbb P(\mathcal O_{X}(-D_i))\subset P_{i+1},
\]
and write 
\[
X_\infty=X_{1\infty}+X_{2\infty}+\cdots+X_{m\infty}.
\] 
Let
\[
p_m: P_{m,X_\infty,\vec r}:=P_{m,(X_{1\infty},r_1),(X_{2\infty},r_2),\ldots,(X_{m\infty},r_m)}\rightarrow P_m.
\]
The multi-root stack $X_{D,\vec r}$ is the zero locus of a section of $\oplus_{i=1}^m p_m^*\mathcal O_{P_i}(1)$. If $D_i$ are nef, we can also apply the orbifold quantum Lefschetz theorem to obtain the following:
\begin{align}
\left\langle \prod_{i=1}^l \tau_{a_i}(\iota^*\gamma_i)\right\rangle_{0,l,\beta}^{X_{D,\vec r}}=\left[\left\langle \prod_{i=1}^l \tau_{a_i}(\gamma_i)\right\rangle_{0,l,\beta+\sum_{i=1}^m d_if_i}^{P_{m,X_\infty,\vec r},\oplus_{i=1}^mp_m^*\mathcal O_{P_i}(1)}\right]_{\vec\lambda=\vec 0}.
\end{align}

\begin{remark}
Instead of considering root stacks as hypersurfaces, we can also consider genus zero orbifold Gromov--Witten invariants of root stacks $X_{D,r}$ (with ambient insertions) as (doubly-)twisted invariants of the gerbe $\mathcal X_r$ of $r$-th roots of $\mathcal O(D)$ over $X$ twisted by $\frac{e_{\mathbb C^\times}(\mathcal O(D))_{0,k+1,\beta}}{e_{\mathbb C^\times}(\mathcal O(D/r))_{0,k+1,\beta}}$. This was proved in \cite{FTY}*{Theorem 1.3} for smooth pairs when $D$ is nef. The proof is through the $\mathbb C^\times$-virtual localization computation on the bundle $P_{X_\infty,r}$. A similar computation also works for snc pairs when $D_i$'s are nef.
\end{remark}

\subsection{A mirror theorem for nef snc pairs}\label{sec:mirror-thm-nef}
Recall that the $J$-function of $X$ is
\[
J_{X}(\tau,z)=z+\tau+\sum_{\substack{(\beta,l)\neq (0,0), (0,1)\\ \beta\in \on{NE(X)}}}\sum_{\alpha}\frac{Q^{\beta}}{l!}\left\langle \frac{\phi_\alpha}{z-\psi},\tau,\ldots, \tau\right\rangle_{0,1+l, \beta}^{X}\phi^{\alpha},
\]
where $\tau=\tau_{0,2}+\tau^\prime\in  H^*(X)$; $\tau_{0,2}\in H^2(X)$; $\tau^\prime\in H^*(X)\setminus H^2(X)$; $\on{NE(X)}$ is the cone of effective curve classes in $X$; $\{\phi_\alpha\}$ is a basis of $H^*(X)$; $\{\phi^\alpha\}$ is the dual basis under the Poincar\'e pairing.

The $J$-function of $(X,D)$ is defined in a similar way. Recall that the state space for the pair $(X,D)$ is 
\[
\mathfrak H:=\bigoplus_{\vec s\in \mathbb Z^m}\mathfrak H_{\vec s}:=\bigoplus_{\vec s\in \mathbb Z^m}H^*(D_{\vec s})
\]
and
\[
D_{\vec s}:=D_{I_{\vec s}}:=\cap_{i:s_i\neq 0}D_i \quad D_{\vec 0}:=X.
\]

\begin{defn}\label{def-relative-J-function}
Let $X$ be a smooth projective variety and $D$ be an snc divisor, the $J$-function for the pair $(X,D)$ is defined as
\[
J_{(X,D)}(\tau,z)=z+\tau+\sum_{\substack{(\beta,l)\neq (0,0), (0,1)\\ \beta\in \on{NE(X)}}}\sum_{k,\vec s}\frac{t^{\beta}}{l!}\left\langle \frac{T_{\vec s,k}}{z-\bar{\psi}},\tau,\ldots, \tau\right\rangle_{0,1+l, \beta}^{(X,D)}T_{-\vec s}^k,
\]
where $\tau=\tau_{0,2}+\tau^\prime\in  \mathfrak H$; $\tau_{0,2}\in H^2(X)\subset \mathfrak H_{\vec 0}$; $\tau^\prime\in \mathfrak H\setminus H^2(X)$; $\{T_{\vec s,k}\}$ is a basis of $\mathfrak H$; $\{T_{-\vec s}^k\}$ is the dual basis under the Poincar\'e pairing; $d_i:=D_i\cdot\beta$; the class $\bar{\psi}$ is the pullback of the corresponding descendant class $\psi$ from $\overline{M}_{0,l+1,\beta}(X)$.  
\end{defn}
 
\begin{remark}
In Definition \ref{def-relative-J-function}, we replace $Q^\beta$ by $t^{\beta}$. We will use $Q$ in the $J$-function of $X$ which will appear in the $I$-function of $(X,D)$. 
\end{remark}

We would also like to change $\tau_{0,2}$ to $\tau_{0,2}+D_i\log x_i^{-1}$ so that $\tau_{0,2}$ represents other divisor classes in $H^2(X)$.  Then we can write the $J$-function as
\begin{align*}
&J_{(X,D)}(\tau,z)\\
=&e^{-\sum_{i=1}^m D_i\log x_i/z}\left(z+\tau+\sum_{\substack{(\beta,l)\neq (0,0), (0,1)\\ \beta\in \on{NE(X)}}}\sum_{k,\vec s}\frac{t^{\beta}x^{-\vec d}}{l!}\left\langle \frac{T_{\vec s,k}}{z-\bar{\psi}},\tau,\ldots, \tau\right\rangle_{0,1+l, \beta}^{(X,D)}T_{-\vec s}^k\right),
\end{align*}
where $x^{-\vec d }:=\prod_{i=1}^m x_i^{-d_i}$.

Consider the space
\[
\mathcal H=\mathfrak H \otimes_{\mathbb C}\mathbb C[\![\on{NE}(X)]\!](\!(z^{-1})\!),
\]
where $(\!(z^{-1})\!)$ means the formal Laurent series in $z^{-1}$.
There is a $\mathbb C[\![\on{NE}(X)]\!]$-valued symplectic form
\[
\Omega(f,g)=\text{Res}_{z=0}(f(-z),g(z))dz, \text{ for } f,g\in \mathcal H,
\]
where the pairing $(f(-z),g(z))$ takes values in $ \mathbb C[\![\on{NE}(X)]\!](\!(z^{-1})\!)$ and is induced by the pairing on $\mathfrak H$.

Consider the following polarization:
\[
\mathcal H=\mathcal H_+\oplus\mathcal H_-,
\]
where
\[
\mathcal H_+=\mathfrak H \otimes_{\mathbb C} \mathbb C[\![\on{NE}(X)]\!][z], \quad \text{and} \quad \mathcal H_-=z^{-1}\mathfrak H \otimes_{\mathbb C} \mathbb C[\![\on{NE}(X)]\!][\![z^{-1}]\!].
\]
There is a natural symplectic identification between $\mathcal H_+\oplus \mathcal H_-$ and the cotangent bundle $T^*\mathcal H_+$.

For $i\geq 0$, we write $ \bt_i=\sum\limits_{\vec s,k} \bt_{i;\vec s,k} T_{\vec s,k}$ where $\bt_{i;\vec s,k}$ are formal variables. We also write
\[
\mathbf t(z)=\sum\limits_{i=0}^\infty \bt_i z^i.
\]
The genus $0$ descendant Gromov-Witten potential of $(X,D)$ is defined as
\[
\mathcal F^0_{X_{D,\infty}}(\bt(z))=\sum\limits_\beta \sum\limits_{l=0}^\infty \dfrac{t^\beta }{l!} \left\langle\mathbf t(\bar\psi),\ldots,\mathbf t(\bar\psi)\right\rangle_{0,l,\beta}^{(X,D)}.
\]
Givental's Lagrangian cone is the graph of the differential $d\mathcal F^0_{X_{D,\infty}}$.
A (formal) point in Givental's Lagrangian cone $\mathcal L_{(X,D)}\subset \mathcal H=T^*\mathcal H_+$ can be explicitly written as
\[
-z+\mathbf t(z)+\sum\limits_{\beta} \sum\limits_{l} \sum\limits_{\vec s,k} \dfrac{t^\beta}{l!} \left\langle\dfrac{ T_{\vec s,k}}{-z-\bar\psi},\mathbf t(\bar\psi),\ldots,\mathbf t(\bar\psi)\right\rangle_{0,l+1,\beta}^{X_{D,\infty}} T_{-\vec s}^k.
\]
In particular, the $J$-function is a slice in Givental's Lagrangian cone. 

The following $I$-function was defined in \cite{TY20b}.

\begin{defn}\label{def-I-snc-nef}
    The non-extended $I$-function for $(X,D)$, where $D=D_1+\cdots+D_m$ and $D_i$'s are nef for $i=1,\ldots,m$, is
\begin{align}\label{I-snc}
I_{(X,D)}(Q,q,\tau,z):=e^{\sum_{i=1}^m D_i\log q_i/z}\sum_{\beta\in \on{NE}(X)} J_{X, \beta}(\tau,z)Q^{\beta}q^{\vec d}
\prod_{i=1}^m\prod_{0< a< d_i}(D_i+az)[\textbf{1}]_{ -\vec d},
\end{align}
where $\vec d:=(d_1,\ldots,d_m):=(D_1\cdot \beta,\ldots,D_m\cdot\beta)$. 
\end{defn}

\begin{remark}
We recall that the variables $q_i$ in the $I$-function come from the fiber direction of the iterative $\mathbb P^1$-bundle over $X$ which we use to obtain the $I$-function. The variable $q_i$ in the $I$-function corresponds to the variable $x_i^{-1}$ in the $J$-function of $(X,D)$. Therefore, $D_i\log x_i^{-1}$ should be considered as the restriction of $h_i\log x_i^{-1}$ from the ambient toric stack bundle, where $h_i=c_1(\mathcal O_{P_i}(1))$. 
\end{remark}

One can also consider the $S$-extended $I$-function $I_{(X,D)}(Q,q,\tau,x_{\vec a},z)$. The precise form of the $S$-extended $I$-function depends on the extended data $S$. Note that the variables $x_{\vec a}$ come from the extended data, so they are not the same as the variables $x_i$ (although one may identify them later in some computations). 

A mirror theorem for $(X,D)$ was proved in  \cite{TY20b}*{Section 4} as follows.
\begin{thm}
    Let $X$ be a smooth projective variety and $D=D_1+\cdots+D_m$ be an snc divisor such that all $D_i$' s are nef. Then the $I$-function $I_{(X,D)}(Q,q,\tau,x_{\vec a},z)$ lies in Givental's Lagrangian cone of $(X,D)$ defined in \cite{TY20c}*{Section 5}. 
\end{thm}

\subsection{The ambient $I$-function}
Recall that the root stack $X_{D,r}$ is a hypersurface of the weighted $\mathbb P^1$-bundle $P_{X_{\infty},r}$, where $D$ is a smooth divisor of $X$. When $D$ is nef, we have the orbifold quantum Lefschetz principle that relates their $I$-functions as follows:
\[
\iota_* I_{X_{D,r}}=h\cup I_{P_{X_\infty},r}|_{k=D\cdot \beta}
\] 
which takes value in $H_{\on{CR}}^*(P_{X_{\infty},r})$, where $h:=c_1(p^*\mathcal O_P(1))$ and $k$ is the degree of the fiber direction. We denote this pushforward of the $I$-function of $X_{D,r}$ by $$I_{X_{D,r},\on{amb}}.$$ Taking the large $r$ limit, we obtain the ambient $I$-function of the smooth pair $(X,D)$ denoted by $$I_{(X,D),\on{amb}}.$$ 
Sometimes it is more convenient to consider the $I$-function which takes value in the cohomology of the ambient space $H_{\on{CR}}^*(P_{X_{\infty},r})$ instead of in $H_{\on{CR}}^*(X_{D,r})$. For example, we have 
\[
I_{(X,D),\on{amb}}(Q,q=1,z)=I_{\mathcal O_X(-D)}(Q,z).
\]
Note that the ambient $I$-function takes value in the twisted sectors $H_{\on{CR}}^*(P_{X_{\infty},r})$. Therefore, both sides take values in $H^*(X)$ when $D\cdot\beta>0$. This is the local-relative/orbifold correspondence \cite{vGGR}, \cite{TY20b} and \cite{BNTY} for the $I$-functions.

Similarly, we have ambient $I$-functions for multi-root stacks and their limits are denoted by 
\[
I_{X_{D,\vec r},\on{amb}} \text{ and } I_{(X,D),\on{amb}}
\]
 respectively for $D=D_1+\cdots+D_m$ with $D_i$ are nef for all $i\in \{1,\ldots,m\}$. 

\section{A relative mirror theorem via degeneration}\label{sec:rel-mirror-deg}
In this section, we will prove a relative mirror theorem of an snc pair $(X,D)$ via degeneration. We start with the case where $D$ has one irreducible component. That is, we consider a smooth pair $(X,D)$. The case for a general snc pair will follow by iteration. If we restrict to the nef case, we recover the relative mirror theorems of \cite{FTY} and \cite{TY20b}.

We recall \cite{FTY} that, when $D$ is nef, we have
\[
\left\langle \prod_{i=1}^l \tau_{a_i}(\iota^*\gamma_i)\right\rangle_{0,l,\beta}^{X_{D,r}}=\left[\left\langle \prod_{i=1}^l \tau_{a_i}(\gamma_i^{\mathbb C^\times})\right\rangle_{0,l,\beta+df}^{P_{X_\infty,r},p^*\mathcal O_P(1)}\right]_{\lambda=0},
\]
where $d=D\cdot\beta$. Therefore, the $I$-function for the root stack $X_{D,r}$ is
\begin{equation}\label{orb-I-function}
\begin{split}
&I_{X_{D,r}}(Q,q,\tau,z):=\\
&e^{D\log q/z}\sum_{\beta\in \on{NE}(X)} J_{X, \beta}(\tau,z)Q^{\beta}q^d\frac{\prod_{0< a\leq D\cdot\beta}(D+az)}{\prod_{\langle a \rangle =\langle D\cdot\beta/r\rangle, 0< a\leq D\cdot\beta/r}(D+az)}[1]_{ -\langle D\cdot\beta/r\rangle}.
\end{split}
\end{equation}

The relative $I$-function for $(X,D)$ is taken as a limit of the $I$-function for $X_{D,r}$:
\begin{align}\label{I-function}
I_{(X,D)}(Q,q,\tau,z):=e^{D\log q/z}\sum_{\beta\in \on{NE}(X)} J_{X, \beta}(\tau,z)Q^{\beta}q^d\prod_{0< a< D\cdot\beta}(D+az)[1]_{ -D\cdot\beta}.
\end{align}
This is a special case of Definition \ref{def-I-snc-nef} when $m=1$. Now we consider the case where $D$ is not necessarily nef.

\subsection{An example}\label{sec:ex}
As remarked in \cite{FTY}, in addition to the case where $D$ is nef, there are also some special cases where we know the $I$-functions of the root stacks. For example, when $D$ is a toric invariant divisor of a toric stack or a toric stack bundle $X$, the root stack $X_{D,r}$ is still a toric stack or a toric stack bundle whose $I$-functions are known from \cite{CCIT15} and \cite{JTY}. This follows from the virtual localization computation for the corresponding torus action.

We now explain an example of the relative $I$-function for a $\mathbb P^1$-bundle which will be our local model in the general setting. Let $D$ be a smooth projective variety equipped with a line bundle $L$ and let $Y$ be the total space of the $\mathbb P^1$-bundle
\[
\pi: \mathbb P(L\oplus \mathcal O_D)\rightarrow D.
\] 
Let $D_0$ and $D_\infty$ be the zero and infinity divisors of $Y$ respectively. 
The $I$-function $I_{Y}(Q,q,\tau,z)$ of $Y$ as defined in \cite{Brown} is a $H^*(Y)$-valued function. The relative $I$-function $I_{(Y,D_0)}(Q,q,\tau,z)$ of $(Y,D_0)$ in \cite{FTY} is obtained as a $(r_0\rightarrow \infty)$-limit of the orbifold $I$-function of $Y_{D_0,r_0}$:
\begin{equation}\label{I-func-Y}
\begin{split}
   e^{h\log q/z}\sum_{\beta\in \on{NE}(D),k\in \mathbb Z} & J_{D, \beta}(\tau,z)Q^{\beta}q^k
\cdot\left(\prod_{\text{if } k-c_1(L)\cdot\beta>0}\frac{1}{(h-c_1(L)+kz)}\right)\\
&\cdot\left(\frac{\prod_{a\leq 0}(h+az)}{\prod_{a\leq k}(h+az)}\right)[1]_{-k+c_1(L)\cdot\beta},
\end{split}
\end{equation}
where $h:=c_1(\mathcal O_Y(1))$. Note that, when $k<0$, we have a factor of $h$ in the numerator of the $I$-function (\ref{I-func-Y}). So, this part of the $I$-function (\ref{I-func-Y}) takes value in $h\cup \mathfrak H (Y,D_0)$, where $\mathfrak H(Y,D_0)$ is the state space for the relative Gromov--Witten theory of $(Y,D_0)$ defined in \cite{FWY} (see also Section \ref{sec:mirror-thm-nef}). So does the corresponding part of the $J$-function of $(Y,D_0)$ under the mirror map. We also see that the $I$-function takes value in $H^*(D_0)$ when $k-c_1(L)\cdot\beta\neq 0$. Therefore, when $k-c_1(L)\cdot\beta\neq 0$, we can not have $k<0$. Because $k<0$ implies there is a factor of $h$ in (\ref{I-func-Y}) and $h|_{D_0}=0$. The mirror map may appear only when $k=c_1(L)\cdot\beta<0$. Hence, if $c_1(L)\cdot\beta\geq 0$ for all $\beta$, then there is no mirror map.

Now we consider the $r_\infty$-th root stack $Y_{(D_0,r_0),(D_\infty,r_\infty)}$ of $Y_{D_0,r_0}$ along $D_\infty$. This is a $\mathbb P^1[r_0,r_\infty]$-bundle over $D$ and the $I$-function is defined in \cite{JTY}. Then, taking the limit $r_0\rightarrow \infty$, the $I$-function $I_{(Y_{D_\infty,r_\infty},D_0)}(Q,q,\tau,z)$ is 
\begin{equation}\label{I-func-Y-D-r}
\begin{split}
   e^{h\log q/z}\sum_{\beta\in \on{NE}(D),k\in \mathbb Z}& J_{D, \beta}(\tau,z)Q^{\beta}q^k
\cdot\left(\prod_{\text{if } k-c_1(L)\cdot\beta>0}\frac{1}{(h-c_1(L)+kz)}\right)\\
&\cdot\left(\frac{\prod_{\langle a\rangle=\langle k/r_\infty\rangle, a\leq 0}(h/r_\infty+az)}{\prod_{\langle a\rangle=\langle k/r_\infty\rangle, a\leq k}((h/r_\infty+az)}\right)[1]_{-\langle k-c_1(L)\cdot\beta, (k)/r_\infty\rangle}.
\end{split}
\end{equation}

In the mirror theorem, we would like to write the $I$-function of the $r_\infty$-th root stack $(Y_{D_\infty,r},D_0)$ in terms of the $J$-function of $(Y,D_0)$, similar to the form in (\ref{orb-I-function}). If $D_\infty=h$ is nef, we see that the $I$-function (\ref{I-func-Y}) of $(Y,D_0)$ is of the form $z+\tau+h\log q+O(z^{-1})$, so the mirror map is trivial. Then we can write (\ref{I-func-Y-D-r}) as
\begin{align*}
 e^{h\log q/z}\sum_{\beta\in \on{NE}(D),k\in \mathbb Z} & I_{(Y,D_0), \beta,k}(\tau,z)Q^{\beta}q^k
\cdot\left(\frac{\prod_{a\leq k}(h+az)}{\prod_{a\leq 0}(h+az)}\right)\\
&\cdot\left(\frac{\prod_{\langle a\rangle=\langle k/r_\infty\rangle, a\leq 0}(h/r_\infty+az)}{\prod_{\langle a\rangle=\langle k/r_\infty\rangle, a\leq k}(h/r_\infty+az)}\right)\cdot[1]_{-\langle k/r_\infty\rangle}\\
=e^{h\log q/z}\sum_{\beta\in \on{NE}(D),k\in \mathbb Z} & J_{(Y,D_0), \beta,k}(\tau,z)Q^{\beta}q^k
\cdot\left(\frac{\prod_{a\leq k}(h+az)}{\prod_{a\leq 0}(h+az)}\right)\\
&\cdot\left(\frac{\prod_{\langle a\rangle=\langle k/r_\infty\rangle, a\leq 0}(h/r_\infty+az)}{\prod_{\langle a\rangle=\langle k/r_\infty\rangle, a\leq k}(h/r_\infty+az)}\right)\cdot [1]_{-\langle k/r_\infty\rangle}.
\end{align*}  
This is what we already knew from \cite{FTY}. 

However, when $D_\infty$ is not nef, we may not simply replace the $I$-function $I_{(Y,D_0)}$ by the $J$-function $J_{(Y,D_0)}$ in the formula since there can be non-trivial mirror maps. Let $\tau(z)$ be the mirror map between $I_{(Y,D_0)}$ and $J_{(Y,D_0)}$. In other words, we have
\[
J_{(Y,D_0)}(\tau(z),z)=I_{(Y,D_0)}(\tau,z)=z+\tau(z)+h\log q+O(z^{-1}).
\]

Then we can write the $I$-function of $(Y_{D_\infty,r_\infty},D_0)$ as
\begin{align*}
 e^{h\log q/z}\sum_{\beta\in \on{NE}(D),k\in \mathbb Z} & I_{(Y,D_0), \beta,k}(\tau,z)Q^{\beta}q^k
\cdot\left(\frac{\prod_{a\leq k}(h+az)}{\prod_{a\leq 0}(h+az)}\right)\\
&\cdot\left(\frac{\prod_{\langle a\rangle=\langle (k)/r_\infty\rangle, a\leq 0}(h/r_\infty+az)}{\prod_{\langle a\rangle=\langle k/r_\infty\rangle, a\leq k}(h/r_\infty+az)}\right) [1]_{-\langle k/r_\infty\rangle}\\
=e^{h\log q/z}\sum_{\beta\in \on{NE}(D),k\in \mathbb Z} & J_{(Y,D_0), \beta,k}(\tau(z),z)Q^{\beta}q^k
\cdot\left(\frac{\prod_{a\leq k}(h+az)}{\prod_{a\leq 0}(h+az)}\right)\\
&\cdot\left(\frac{\prod_{\langle a\rangle=\langle k/r_\infty\rangle, a\leq 0}(h/r_\infty+az)}{\prod_{\langle a\rangle=\langle k/r_\infty\rangle, a\leq k}(h/r_\infty+az)}\right) \cdot [1]_{-\langle k/r_\infty\rangle},
\end{align*}  
where $ J_{(Y,D_0), \beta,k}(\tau(z),z)$ is the $Q^\beta q^k$-coefficient of $J_{(Y,D_0)}(\tau(z),z)$. In other words, we replace the $J$-function of $Y$ by $J_{(Y,D_0)}(\tau(z),z)$ which is another slice in Givental's Lagrangian cone $\mathcal L_{(Y,D_0)}$ for the Gromov--Witten theory of $(Y,D_0)$.

\begin{remark}\label{rm-tw}
We may consider the $I$-function of $(Y_{D_\infty,r_\infty},D_0)$ as a non-equivariant limit of the $I$-function of the Gromov--Witten theory of a gerbe $(\mathcal Y_{D_0,r_0})_{r_\infty}$ over $Y_{D_0,r_0}$  twisted by $(e_{\mathbb C^\times}, O(D))$ and $(e^{-1}_{\mathbb C^*}, O(D/r_\infty))$ as in \cite{FTY}*{Theorem 3.6} when $D_\infty$ is nef. When $D_\infty$ is not nef, if we simply use the $J$-function $J_{Y,D_0}$ without $\tau(z)$, then the non-equivariant limit may not be well-defined. For example, when $k<0$, we have a factor $\frac{1}{h+\lambda}$, where $\lambda$ is an equivariant parameter. The non-equivariant limit does not exist if there is no $(h+\lambda)$ in the numerator to cancel this out. By considering the $J$-function with the mirror map $\tau(z)$, we get the $I$-function of $I_{(Y,D_0)}$, which contains a factor of $(h+\lambda)$ to cancel $\frac{1}{h+\lambda}$. The non-equivariant limit of the equivariant $I$-function of $(Y_{D_\infty,r_\infty},D_0)$ becomes well-defined.
\end{remark}

We can also obtain $S$-extended $I$-functions in a similar way. 
This allows us to compute genus zero invariants with several orbifold markings.

Now, consider again the $I$-function (\ref{I-func-Y}). There are special cases where the mirror map $\tau(z)$ is trivial. For example, when $\tau\in H^2(D)$ and $(-K_D-c_1(L))\cdot\beta>0$, for all $\beta$, the $I$-function (\ref{I-func-Y}) still has the form $z+\tau+O(z^{-1})$. So, the mirror map between $J_{(Y,D_0)}$ and $I_{(Y,D_0)}$ is also trivial in this case. If we have $\tau\in H^2(D)$ and $(-K_D-c_1(L))\cdot\beta\geq 0$, there can be non-trivial mirror maps. But the mirror map $\tau(z)$ does not have positive powers in $z$, so it is also relatively simple. 

\subsection{The statement of the mirror theorem}
In this section, we compare orbifold Gromov--Witten invariants of $X_{D,r}$ and twisted Gromov--Witten invariants of $P_{X_{\infty},r}$ via degeneration and without assuming $D$ is nef. Let $$Y:=\mathbb P(L\oplus \mathcal O_D), \text{ with } L:=N_{D/X}$$ and $$D_0:=\mathbb P_D(N_{D/X})\subset Y$$ be the zero divisor of $Y$. Let $I_{(Y,D_0)}(Q,q,\tau,z)$ be the $I$-function of $(Y,D_0)$:

\begin{equation}\label{I-func-Y-D-0}
\begin{split}
   e^{h\log q/z}\sum_{\beta\in \on{NE}(D)}\sum_{k\in \mathbb Z)\cdot\beta} & J_{D, \beta}(\tau,z)Q^{\beta}q^k[1]_{-k+c_1(N_{D/X})\cdot\beta}
\cdot\left(\frac{\prod_{a\leq 0}(h+az)}{\prod_{a\leq k}(h+az)}\right)\\
&\cdot\prod_{k-c_1(N_{D/X})\cdot\beta>0}\frac{1}{h-c_1(N_{D/X})+(k-c_1(N_{D/x})\cdot\beta)z},
\end{split}
\end{equation}
where $h:=c_1(\mathcal O_Y(1))$, obtained as a limit of the $I$-function of orbifold invariants of $Y_{D_0,r}$ by the relative-orbifold correspondence of \cite{ACW}, \cite{TY18} and \cite{FWY}, where the $I$-function of $Y_{D_0,r}$ is defined in \cite{JTY} as a weighted-$\mathbb P^1$-bundle over $D$.  When $D$ is not nef, there can be a non-trivial mirror map between $J_{(Y,D_0)}$ and $I_{(Y,D_0)}$. As we explained in Section \ref{sec:ex}, the mirror map appears only when there are curve classes $\beta$ such that $k=c_1(N_{D/X} )\cdot\beta<0$. This mirror map is actually the same as the mirror map between the $J$-function $J_{N_{D/X}}$ of $N_{D/X}$ and the $I$-function $I_{N_{D/X}}$ of $N_{D/X}$. The mirror map takes value in $h\cup H^*(D)$. For an element $(h\cup \gamma )\in h\cup H^*(D)$, we can think about it coming from a class $\iota_{D,*}\gamma\in H^*(X)$, where $\iota_D: D\hookrightarrow X$ is the inclusion map, under the degeneration of $X$ to the normal cone of $D$.

We claim the following. 
\begin{theorem}\label{thm-quan-lef}

The $I$-function $I_{X_{D,r}}(Q,q,\tau_D,z)$ for $X_{D,r}$ is a $H_{\on{CR}}^*(X_{D,r})$-valued function:
\begin{equation}\label{I-func-X-D-r}
\begin{split}
= e^{D\log q/z}\sum_{\beta\in \on{NE}(X)} J_{X, \beta}(\tau_D,z)Q^{\beta}q^{D\cdot\beta}\frac{\prod_{ a\leq D\cdot\beta}(D+az)\prod_{\langle a \rangle =\langle D\cdot\beta/r\rangle, a\leq 0}(D/r+az)}{\prod_{ a\leq 0}(D+az)\prod_{\langle a \rangle =\langle D\cdot\beta/r\rangle, a\leq D\cdot\beta/r}(D/r+az)}[1]_{ -D\cdot\beta/r},
\end{split}
\end{equation}
where $\tau_D$ is the mirror map between $J_{N_{D/X}}$ and $I_{N_{D/X}}$. The $I$-function $I_{X_{D,r}}(Q,q,\tau_D,z)$  lies in  Givental's Lagrangian cone  $\mathcal L_{X_{D,r}}$of $X_{D,r}$.

\end{theorem}

\begin{remark}
We will also use the ambient $I$-function $I_{X_{D,r},\on{amb}}(Q,q,\tau,z)$ for $X_{D,r}\subset P_{X_{\infty},r}$ which is a $H_{\on{CR}}^*(P_{X_{\infty},r})$-valued function:
\begin{equation}\label{I-func-X-D-r-amb}
\begin{split}
 & e^{h\log q/z}\sum_{\beta\in \on{NE}(X),d=D\cdot\beta} J_{X, \beta}(\tau_D,z)Q^{\beta}q^{d}\frac{\prod_{ a\leq d}(h+az)\prod_{\langle a \rangle =\langle d/r\rangle, a\leq 0}(h/r+az)}{\prod_{ a\leq 0}(h+az)\prod_{\langle a \rangle =\langle d/r\rangle, a\leq d/r}(h/r+az)}[1]_{ -d/r}\cup h\\
=& e^{D\log q/z}\sum_{\beta\in \on{NE}(X)} J_{X, \beta}(\tau_D,z)Q^{\beta}q^{D\cdot\beta}\frac{\prod_{ a\leq D\cdot\beta}(D+az)\prod_{\langle a \rangle =\langle D\cdot\beta/r\rangle, a\leq 0}(D/r+az)}{\prod_{ a\leq 0}(D+az)\prod_{\langle a \rangle =\langle D\cdot\beta/r\rangle, a\leq D\cdot\beta/r}(D/r+az)}[1]_{ -D\cdot\beta/r}\cup h.
\end{split}
\end{equation}
\end{remark}

\begin{remark}
When $D$ is nef, the $I$-function (\ref{I-func-Y-D-0}) is of the form $z+\tau+h\log q+O(z^{-1})$. The mirror map between $J_{(Y,D_0)}$ and $I_{(Y,D_0)}$ is trivial, so we do not see it in \cite{FTY}. If $X$ is Fano and $\tau\in H^{\leq 2}(D)$, the mirror map is also trivial. In the general setting, the mirror map is not trivial and may require a Birkhoff factorization procedure. This mirror map comes from the contribution of counting curves that map into $D$ in the degeneration formula. In the remainder of the paper, we will refer to the mirror map between $J_{(Y,D_0)}$ and $I_{(Y,D_0)}$ as the mirror map in the divisor $D$.
\end{remark}

\begin{remark}
We note that Theorem \ref{thm-quan-lef} holds without requiring $r$ to be sufficiently large. Although we are mostly interested in the large $r$ limit in this paper, we state results for general $r$ because they will be of independent interest in orbifold Gromov--Witten theory. 
\end{remark}

\subsection{Proof of Theorem \ref{thm-quan-lef}}

We would like to compare the invariants of $X_{D,r}$ and the twisted invariants of $P_{X_\infty,r}$. We consider the equivariant invariant
\[
\left\langle \prod_{i=1}^l \tau_{a_i}(\gamma_i^{\mathbb C^\times})\right\rangle_{0,l,\beta+df}^{P_{X_\infty,r},p^*\mathcal O_P(1)},
\]
where $\gamma_i\in H^*_{\on{CR}}(P_{X_\infty,r})$ and $a_i\in \mathbb Z_{\geq 0}$, for $1\leq i\leq l$; the class $\gamma_i^{\mathbb C^\times}$ is a $\mathbb C^\times$-equivariant lift of $\gamma_i$, the class $f$ is the class of a fiber; and $d=D\cdot\beta$. We hope to prove that
\begin{align}\label{iden-quan-lef}
\left\langle \prod_{i=1}^l \tau_{a_i}(\iota^*\gamma_i)\right\rangle_{0,l,\beta}^{X_{D,r}}=\left[\left\langle \prod_{i=1}^l \tau_{a_i}(\gamma_i^{\mathbb C^\times})\right\rangle_{0,l,\beta+df}^{P_{X_\infty,r},p^*\mathcal O_P(1)}\right]_{\lambda=0},
\end{align}
where $\iota: X_{D,r}\hookrightarrow P_{X_\infty,r}$ is the inclusion as a zero locus of a section of $p^*\mathcal O_P(1)$ and $\lambda$ is the equivariant parameter. However, this is not true in general when $D$ is not nef. But we can still obtain an $I$-function after applying a possibly non-trivial mirror map from $(Y,D_0)$. We will prove this as follows.

\begin{remark}
As the root construction is ``local over the divisor $D$", a general observation is that one may use the degeneration to the normal cone to study questions about the properties of the Gromov--Witten theory of $X_{D,r}$ via the corresponding properties of the relative local model $(Y_{D_\infty,r},D_0)$. As a mirror theorem is known for the relative local model, we can expect that a mirror theorem for general root stacks can also be obtained using the degeneration argument.  
\end{remark}

\subsubsection{Degenerations}\label{sec:degeneration}
The degeneration for $P_{X_\infty,r}$ that we will consider here is based on the one considered in \cite{Wang} and \cite{WY24}, but we study different types of invariants here. We first recall the degeneration in \cite{Wang} and \cite{WY24}. 

We start with the degeneration to the normal cone of $D$ in $X$:
\[
\mathfrak X=\on{Bl}_{D\times \{0\}}(X\times \mathbb A^1)\rightarrow \mathbb A^1.
\]
Let $\mathfrak D$ be the strict transform of $D\times \mathbb A^1$ in $\mathfrak X$. Now we consider
\[
\pi: \mathfrak P=\mathbb P(\mathcal O_{\mathfrak X}(-\mathfrak D)\oplus \mathcal O)\rightarrow \mathbb A^1.
\]
This gives a degeneration, such that the generic fiber $\pi^{-1}(t)$ is $P$ and the central fiber $\pi^{-1}(0)$ is
\[
(X\times \mathbb P^1)\cup _{D\times \mathbb P^1}P_Y,
\]
where $$Y:=\mathbb P(N_{D/X}\oplus \mathcal O_D),$$ $D_\infty$ is the infinity section of $Y$ and $$P_Y:=\mathbb P(\mathcal O_Y(-D_\infty)\oplus \mathcal O).$$ We then apply the $r$-th root construction to the infinity divisor $\mathfrak X_\infty$ of $\mathfrak P$, and obtain the following
\[
\pi: \mathfrak P_{\mathfrak X_\infty,r}\rightarrow \mathbb A^1,
\]
where, by abuse of notation, we still use $\pi$ here. Then the generic fiber $\pi^{-1}(t)$ is $P_{X_\infty,r}$ and the central fiber $\pi^{-1}(0)$ is
\[
\left((X\times \mathbb P^1)_{X\times\{\infty\},r}\right)\bigcup_{(D\times \mathbb P^1)_{D\times\{\infty\},r}}\left((P_Y)_{Y_\infty,r}\right).
\]

We then apply the degeneration formula to the genus zero Gromov--Witten invariants of $P_{X_\infty,r}$ twisted by the equivariant line bundle $p^*\mathcal O_P(1)$. 

The degeneration graphs are bipartite graphs. For the vertices over $(X\times \mathbb P^1)_{X\times\{\infty\},r}$, we will have Gromov--Witten invariants of $$((X\times \mathbb P^1)_{X\times\{\infty\},r}, (D\times \mathbb P^1)_{D\times\{\infty\},r})$$ 
twisted by  $p^*\mathcal O_{X\times \mathbb P^1}(0,1)$. For the vertices over $(P_Y)_{Y_\infty,r}$, we will have Gromov--Witten invariants of 
$$\left((P_Y)_{Y_{\infty,r}}, (D\times \mathbb P^1)_{D\times\{\infty\},r}\right)$$
twisted by $p^*\mathcal O_{P_Y}(1)$. 
 
Following \cite{TY20c}*{Theorem 9}, the relative divisor $D\times \mathbb P^1$ can also be considered as an orbifold divisor with a sufficiently large root. To simplify the notation, we write Gromov--Witten invariants of $((X\times \mathbb P^1)_{X\times\{\infty\},r}, (D\times \mathbb P^1)_{D\times\{\infty\},r})$ as Gromov--Witten invariants of 
$$\left(X\times \mathbb P^1, X\times\{\infty\}+D\times \mathbb P^1\right)_{\vec r}$$
 and write Gromov--Witten invariants of $((P_Y)_{Y_{\infty,r}}, (D\times \mathbb P^1)_{D\times\{\infty\},r})$ as Gromov--Witten invariants of $$\left(P_Y,Y_\infty+D\times\mathbb P^1\right)_{\vec r}.$$

The degeneration formula is
\begin{equation}\label{deg-formula}
\begin{split}
&\left\langle \prod_{i=1}^l \tau_{a_i}(\gamma_i^{\mathbb C^\times})\right\rangle_{0,l,\beta+df}^{P_{X_\infty,r},p^*\mathcal O_P(1)}\\
=& \sum \frac{\prod \eta^i}{|\Aut(\eta)|}\left\langle\left. \prod_{i\in S_1}\tau_{a_i}(\gamma_i^{\mathbb C^\times})\right| \eta\right\rangle_{0,|S_1|+|\eta|,\beta_1}^{\bullet,P_{1,\vec r}}\left\langle\eta^\vee\left| \prod_{i\in S_2}\tau_{a_i}(\gamma_i^{\mathbb C^\times})\right.\right\rangle_{0,|S_2|+|\eta|,\beta_2}^{\bullet,P_{2,\vec r}},
\end{split}
\end{equation}
where
\begin{itemize}
\item the dual $\eta^\vee$ is defined by taking the Poincar\'e duals of the cohomology weights of the cohomology weighted partition $\eta$; $|\Aut(\eta)|$ is the order of the automorphism group $\Aut(\eta)$ that preserves equal parts of the cohomology weighted partition $\eta$. The sum is over all possible splittings of $\beta+df=\beta_1+\beta_2$, $\{1,2,\ldots,l\}=S_1\sqcup S_2$ and the cohomology weighted partition $\eta$.
\item The RHS of (\ref{deg-formula}) is the sum of product of invariants of $$P_{1,\vec r}=\left((X\times \mathbb P^1, X\times\{\infty\}+D\times \mathbb P^1)_{\vec r}, p^*\mathcal O_{X\times \mathbb P^1}(0,1)\right)$$ and $$P_{2,\vec r}=\left((P_Y,Y_\infty+D\times\mathbb P^1)_{\vec r},p^*\mathcal O_{P_Y}(1)\right).$$
\end{itemize}

On the other hand, we can also apply the degeneration to the normal cone of $D$ in $X_{D,r}$. More precisely, we first consider $\mathfrak X$ which is a degeneration of $X$ to the normal cone of $D$ then take the $r$-th root stack along the strict transform $\mathfrak D$ of $D\times \mathbb A^1$. The degeneration formula for the Gromov--Witten invariants of $X_{D,r}$ is
\begin{equation}\label{iden-deg-root}
\begin{split}
&\left\langle \prod_{i=1}^l \tau_{a_i}(\iota^*\gamma_i)\right\rangle_{0,l,\beta}^{X_{D,r}}\\
=& \sum \frac{\prod \eta^i}{|\Aut(\eta)|}\left\langle\left. \prod_{i\in S_1}\tau_{a_i}(\iota^*\gamma_i)\right| \eta\right\rangle_{0,|S_1|+|\eta|,\beta_1}^{\bullet,(X,D)}\left\langle\eta^\vee\left| \prod_{i\in S_2}\tau_{a_i}(\iota^*\gamma_i)\right.\right\rangle_{0,|S_2|+|\eta|,\beta_2}^{\bullet,(Y_{D_\infty,r},D_0)}.
\end{split}
\end{equation}

\subsubsection{Reduction to the relative-local model}

Now we study the degeneration formula (\ref{deg-formula}) and compare it with the degeneration formula (\ref{iden-deg-root}). Because $p^*\mathcal O_{X\times \mathbb P^1}(0,1)$ is a convex line bundle, we can apply the orbifold quantum Lefschetz principle of \cite{Tseng} to state that the non-equivariant limits of the orbifold invariants of $(X\times \mathbb P^1, X\times\{\infty\}+D\times \mathbb P^1)_{\vec r}$ twisted by  $p^*\mathcal O_{X\times \mathbb P^1}(0,1)$ are the same as the corresponding invariants of $(X\times [\on{pt}], D\times [\on{pt}])_{r_2}=(X,D)$. Note that the hypersurface $X\times [\on{pt}]\subset \left(X\times \mathbb P^1\right)_{X\times \{\infty\}, r}$ is disjoint from the orbifold hypersurface $X\times\{\infty\}\subset \left(X\times \mathbb P^1\right)_{X\times \{\infty\}, r}$. Therefore, orbifold markings are distributed to the $P_Y$-side. For the degeneration formula (\ref{iden-deg-root}) of $X_{D,r}$, orbifold markings are also distributed to the $Y$-side.

Then, the comparison between the invariants in Identity (\ref{deg-formula}) reduces to the comparison between (disconnected) invariants in the relative-local models, which are 
\[
\left(P_Y,Y_\infty+D\times\mathbb P^1\right)_{\vec r}, \text{ twisted by } p^*\mathcal O_{P_Y}(1), \quad \text{and } Y_{(D_0,r_0), (D_\infty,r_\infty)},
\]
where $D\times \mathbb P^1\subset P_Y$ and $D_0\subset Y$ are relative divisors, but we can consider them as orbifold divisors with large roots.

\subsubsection{The relative-local model}

Now we consider invariants on the $P_Y$-side. We have Gromov--Witten invariants of $\left(P_Y,Y_\infty+D\times\mathbb P^1\right)_{\vec r}$ twisted by $p^*\mathcal O_{P_Y}(1)$. These are twisted invariants of the rank $2$ iterative weighted $\mathbb P^1$-bundle $P_Y$ over $D$. Furthermore, $Y_\infty$ and $D\times\mathbb P^1$ are divisors that are invariant under the fiberwise torus action.

Here we compare genus zero Gromov--Witten invariants of $\left(P_Y,Y_\infty+D\times\mathbb P^1\right)_{\vec r}$ twisted by $p^*\mathcal O_{P_Y}(1)$ and genus zero Gromov--Witten invariants of $Y_{(D_0,r_0),(D_\infty,r_\infty)}$, where $r_0$ is sufficiently large and $r_\infty:=r$ is not necessarily large. Note that $Y_{D_0,r_0}$ is a $\mathbb P^1[r_0]$-bundle over $D$ and $D_0$ and $D_\infty$ are $\mathbb C^\times$-invariant divisors. A mirror theorem for toric stack bundles was proved in \cite{JTY}. Therefore, we can write down the $I$-function of $Y_{(D_0,r_0),(D_\infty,r_\infty)}$ and the $I$-function of $(P_Y,Y_\infty+D\times\mathbb P^1)_{\vec r}$ twisted by $p^*\mathcal O_{P_Y}(1)$. 

We recall that the $I$-function $ I_{Y_{(D_0,r_0), (D_\infty,r_\infty)}}$ for $Y_{(D_0,r_0), (D_\infty,r_\infty)}$ is 
\begin{equation}\label{I-func-Y-D}
\begin{split}
 &  e^{h\log q/z}\sum_{\beta\in \on{NE}(D)}\sum_{k\geq c_1(N_{D/X})\cdot\beta}  J_{D, \beta}(\tau,z)Q^{\beta}q^k[1]_{\langle -k/r_\infty,(-k+c_1(N_{D/X})\cdot\beta)/r_0\rangle}\\
&\cdot\left(\frac{\prod_{\langle a\rangle=\langle k/r_\infty\rangle, a\leq 0}(h/r_\infty+az)}{\prod_{\langle a\rangle=\langle k/r_\infty\rangle, a\leq k/r_\infty}(h/r_\infty+az)}\right)\\
&\cdot\frac{\prod_{\langle a\rangle=\langle (k-c_1(N_{D/X})\cdot\beta)/r_0\rangle, a\leq 0}((h-c_1(N_{D/X}))/r_0+az)}{\prod_{\langle a\rangle=\langle (k-c_1(N_{D/X})\cdot\beta)/r_0\rangle, a\leq (k-c_1(N_{D/X})\cdot\beta)/r_0}((h-c_1(N_{D/X}))/r_0+az)}.
\end{split}
\end{equation}

On the other hand, the ``$I$-function" $\tilde I_{P_{Y, (D\times \mathbb P^1,r_0),(Y_\infty,r_\infty)}}^{\on{tw}}$ for $P_{Y, (D\times \mathbb P^1,r_0),(Y_\infty,r_\infty)}$ twisted by the equivariant line bundle $p^*\mathcal O_{P_Y}(1)$ is 

\begin{equation}\label{I-func-P-Y-tw}
   \begin{split}
&e^{(h_1\log q_1+(h_2+\lambda)\log q_2)/z}\sum_{\beta\in \on{NE}(D)}\sum_{k_2\geq k_1\geq c_1(N_{D/X})\cdot\beta} J_{D, \beta}(\tau,z)Q^{\beta}q_1^{k_1}q_2^{k_2}[1]_{\langle (-k_1+c_1(N_{D/X})\cdot\beta)/r_0, -k_2/r_\infty\rangle}\\
&\cdot\left(\frac{\prod_{ a\leq 0}(h_1+az)\prod_{\langle a\rangle=\langle (k_1-c_1(N_{D/X})\cdot\beta)/r_0\rangle, a\leq 0}((h_1-c_1(N_{D/X}))/r_0+az)}{\prod_{ a\leq k_1}(h_1+az)\prod_{\langle a\rangle=\langle (k_1-c_1(N_{D/X})\cdot\beta)/r_0\rangle, a\leq (k_1-c_1(N_{D/X})\cdot\beta)/r_0}((h_1-c_1(N_{D/X}))/r_0+az)}\right)\\
&\cdot\left(\frac{\prod_{\langle a\rangle=\langle k_2/r_\infty\rangle, a\leq 0}((h_2+\lambda)/r_\infty+az)\prod_{a\leq 0}((h_2+\lambda-h_1+az)}{\prod_{\langle a\rangle=\langle k_2/r_\infty\rangle, a\leq k_2/r_\infty}(h_2+\lambda)/r_\infty+az)\prod_{ a\leq k_2-k_1}(h_2+\lambda-h_1+az)}\right)\\
&\cdot\left(\frac{\prod_{ a\leq k_2}(h_2+\lambda+az)}{\prod_{ a\leq 0}(h_2+\lambda+az)}\right),
\end{split}
\end{equation}
where $h_1:=c_1(\mathcal O_Y(1))$ and $h_2:=c_1(\mathcal O_{P_Y}(1))$.

As in Section \ref{sec:ex}, there is a subtlety here. The twisted $I$-function (\ref{I-func-P-Y-tw}) should be written in terms of the $J$-function of $\left(P_Y,Y_\infty+D\times\mathbb P^1\right)_{\vec r}$, we replace the $J$-function of $\left(P_Y,Y_\infty+D\times\mathbb P^1\right)_{\vec r}$ by the $I$-function of $\left(P_Y,Y_\infty+D\times\mathbb P^1\right)_{\vec r}$ from \cite{JTY}. 
Compared with the $I$-function $I_{Y_{(D_0,r_0),(D_\infty,r_\infty)}}$, we will consider 
\[
\left[(h_2+\lambda)\cup \tilde I_{P_{Y, (D\times \mathbb P^1,r_0),(Y_\infty,r_\infty)}}^{\on{tw}}\right]_{\lambda=0}.
\]
We also notice the following. If $k_2-k_1<0$, then there will be a factor $h_2-h_1$ in the numerator of (\ref{I-func-P-Y-tw}) and $(h_2-h_1)h_2=0$. So we cannot have $k_2-k_1<0$. As in \cite{FTY}, we restrict the Novikov variables to $\NE(Y_{D_0,r_0})$. Hence, we actually restrict to the invariants of the infinity section $(Y_{D_0,r_0})_{r_\infty}$, which is a $\mu_{r_\infty}$-gerbe over $Y_{D_0,r_0}$, twisted by $(e_{\mathbb C^\times}, O(D))$ and $(e^{-1}_{\mathbb C^*}, O(D/r_\infty))$. This is similar to Remark \ref{rm-tw} and \cite{FTY}*{Theorem 3.6}. The $I$-function for this twisted theory is a hypergeometric modification of the $J$-function of $Y_{D_0,r_0}$.
 The mirror map may appear when we replace the $J$-function of $Y_{D_0,r_0}$ by its $I$-function,
which is in terms of the $J$-function of $D$:
\begin{align*}
I_{Y_{D_0,r_0}}=&e^{(h_1\log q_1)/z}\sum_{\beta\in \on{NE}(D)}\sum_{k_1\geq c_1(N_{D/X})\cdot\beta} J_{D, \beta}(\tau,z)Q^{\beta}q_1^{k_1}\frac{\prod_{ a\leq 0}(h_1+az)}{\prod_{ a\leq k_1}(h_1+az)}\\
&\cdot\frac{\prod_{\langle a\rangle=\langle (k_1-c_1(N_{D/X})\cdot\beta)/r_0\rangle, a\leq 0}((h_1-c_1(N_{D/X}))/r_0+az)}{\prod_{\langle a\rangle=\langle (k_1-c_1(N_{D/X})\cdot\beta)/r_0\rangle, a\leq (k_1-c_1(N_{D/X})\cdot\beta)/r_0}((h_1-c_1(N_{D/X}))/r_0+az)}.
\end{align*}
When $D$ is nef, the mirror map between $J_{(Y,D_0)}$ and $I_{(Y,D_0)}$ is trivial. In general, the mirror map for the $I$-function of $(Y,D_0)$ may not be trivial. When $D$ is not necessarily nef, we may have $c_1(N_{D/X})\cdot\beta< 0$. There can be a non-trivial mirror map between $J_{(Y,D_0)}$ and $I_{(Y,D_0)}$. But we do not see this mirror map if we only look at (\ref{I-func-P-Y-tw}). Therefore, the function (\ref{I-func-P-Y-tw}) is denoted as $\tilde I_{P_{Y, (D\times \mathbb P^1,r_0),(Y_\infty,r_\infty)}}^{\on{tw}}$ instead of $I_{P_{Y, (D\times \mathbb P^1,r_0),(Y_\infty,r_\infty)}}^{\on{tw}}$. 
If we assume that $X$ is Fano and $\tau\in H^2(D)$, then the mirror map between $J_{(Y,D_0)}$ and $I_{(Y,D_0)}$ is trivial.

Recall that we still have a hypersurface $$\iota:Y_{(D_0,r_0), (D_\infty,r_\infty)}\hookrightarrow P_{Y, (D\times \mathbb P^1,r_0),(Y_\infty,r_\infty)}$$ as the zero locus of a section of $p^*\mathcal O_P(1)$. We consider the restriction of the $I$-function (\ref{I-func-P-Y-tw}) to $Y_{(D_0,r_0), (D_\infty,r_\infty)}$ and recall that we must have $h=h_1=h_2$ and $k=k_1=k_2$. We also set $q=q_1q_2$. Taking the non-equivariant limit, we have
\[
 \iota_*I_{Y_{(D_0,r_0), (D_\infty,r_\infty)}}=\left[(h+\lambda)\cup \tilde I_{P_{Y, (D\times \mathbb P^1,r_0),(Y_\infty,r_\infty)}}^{\on{tw}}\right]_{\lambda=0}.
\]

The identity also holds for extended $I$-functions when the extended data are chosen to be the same on both sides. Indeed, the genus zero non-equivariant limit of the twisted invariants of $P_{Y, (D\times \mathbb P^1,r_0),(Y_\infty,r_\infty)}$ with $k=D\cdot\beta$ coincides with the genus zero orbifold invariants of $Y_{(D_0,r_0), (D_\infty,r_\infty)}$ if the mirror map for $(Y,D_0)$ is trivial. 

For the twisted invariants of $P_{Y, (D\times \mathbb P^1,r_0),(Y_\infty,r_\infty)}$ that appear in the degeneration formula, one can consider an appropriate extended $I$-function of $P_{Y, (D\times \mathbb P^1,r_0),(Y_\infty,r_\infty)}$ and compare it with the corresponding extended $I$-function of $Y_{(D_0,r_0), (D_\infty,r_\infty)}$. The invariants of $P_{Y, (D\times \mathbb P^1,r_0),(Y_\infty,r_\infty)}$ and the invariants of $Y_{(D_0,r_0), (D_\infty,r_\infty)}$ can be computed from the $I$-function via the same Birkhoff factorization procedure.
We can also compute the invariants of the relative-local model via the $\mathbb C^\times$-virtual localization. The difference between the orbifold invariants and the twisted invariants of the bundle comes from the vertex contributions along $D_\infty$ which are invariants of $N_D$. And the difference now comes from the mirror map for $N_D$ which is the same as the mirror map for $(Y,D_0)$.

 When the mirror map for $(Y,D_0)$ is trivial, we take the non-equivariant limit, and the degeneration formula (\ref{deg-formula}) becomes
\begin{equation}\label{iden-deg-P}
\begin{split}
&\left[\left\langle \prod_{i=1}^l \tau_{a_i}(\gamma_i^{\mathbb C^\times})\right\rangle_{0,l,\beta+df}^{P_{X_\infty,r},p^*\mathcal O_P(1)}\right]_{\lambda=0}\\
=& \sum \frac{\prod \eta^i}{|\Aut(\eta)|}\left\langle \prod_{i\in S_1}\tau_{a_i}(\iota^*\gamma_i)| \eta\right\rangle_{0,|S_1|+|\eta|,\beta_1}^{\bullet,(X,D)}\left\langle\eta^\vee| \prod_{i\in S_2}\tau_{a_i}(\iota^*\gamma_i)\right\rangle_{0,|S_2|+|\eta|,\beta_2}^{\bullet,(Y_{D_\infty,r},D_0)}.
\end{split}
\end{equation}
Two degeneration formulas (\ref{iden-deg-root}) and (\ref{iden-deg-P}) coincide. This gives the $I$-function of $X_{D,r}$. Then we have
\[
\left\langle \prod_{i=1}^l \tau_{a_i}(\iota^*\gamma_i)\right\rangle_{0,l,\beta}^{X_{D,r}}=\left[\left\langle \prod_{i=1}^l \tau_{a_i}(\gamma_i^{\mathbb C^\times})\right\rangle_{0,l,\beta+df}^{P_{X_\infty,r},p^*\mathcal O_P(1)}\right]_{\lambda=0}.
\]
When the mirror map is not trivial, the invariants are not exactly the same and the difference is governed by the mirror map of $(Y,D_0)$. In nice cases (e.g., $X$ is weak Fano), applying the mirror map in $(Y,D_0)$ just means a change of variables for $Q$. In general, it requires a mirror map $\tau\mapsto \tau(z)$ of $(Y,D_0)$.

\begin{remark}
When $X$ is not Fano and $D$ is not nef, the mirror theorem for the pair $(X,D)$ will be more complicated, as there can be mirror maps in $D$. We will study this in more detail in a different paper. 
\end{remark}

\subsubsection{The large $r$ limit}
Following \cite{FTY}, we can take the limit of $r\rightarrow \infty$ to obtain the $I$-function for $(X,D)$.  

\begin{corollary}\label{cor-rel-I}
The $I$-function $I_{(X,D)}(Q,q,\tau_D,z)$ for the smooth pair $(X,D)$:
\begin{equation*}
\begin{split}
I_{(X,D)}(Q,q,\tau_D,z)=I_++I_-,
\end{split}
\end{equation*}
where
\[
I_+:=e^{D\log q/z}\sum_{\beta\in \NE(X), D\cdot\beta>0} J_{X,\beta}(\tau_D,z)Q^\beta q^{D\cdot\beta}\frac{\prod_{ a\leq D\cdot\beta}(D+az)}{\prod_{ a\leq 0}(D+az)} \frac{1}{D+D\cdot\beta z}[1]_{-D\cdot\beta}
\]
and 
\[
I_-:=e^{D\log q/z}\sum_{\beta\in \NE(X), D\cdot\beta\leq 0} J_{X,\beta}(\tau_D,z)Q^\beta q^{D\cdot\beta}\frac{\prod_{ a\leq D\cdot\beta}(D+az)}{\prod_{ a\leq 0}(D+az)} [1]_{-D\cdot\beta},
\]
 lies in Givental's Lagrangian cone $\mathcal L_{(X,D)}$ of $(X,D)$.
\end{corollary}


\begin{remark}
Recall that relative Gromov--Witten invariants with negative contact orders in \cite{FWY} and \cite{FWY19} are defined as limits of orbifold Gromov--Witten invariants of root stacks. Relative Gromov--Witten invariants in \cite{FWY} and \cite{FWY19} are also defined as graph sums which come from computing the orbifold invariants of root stacks via the degeneration formula and the virtual localization formula. In the virtual localization formula for the relative-local model, one has the vertex contributions at the infinity divisor, which give twisted invariants of the root gerbe $\mathcal D_r$ over $D$. The twisted invariants of $P$ can be computed in a similar way through degeneration and virtual localization. We may think of mirror maps for invariants of $N_D$ as possible corrections required when comparing the vertex contributions at the infinity divisor. We may also think about it as modifying the definition of relative Gromov--Witten invariants in \cite{FWY}*{Section 5} by correcting the invariants of $N_D$ in the vertex contribution using the mirror map.
\end{remark}

\subsection{Special cases}\label{sec:special}
Theorem \ref{thm-quan-lef} and Corollary \ref{cor-rel-I}  require mirror maps in $D$ along with the degeneration formula.  Mirror maps can be quite complicated, in general. We study some special cases where the $I$-functions are more explicit.

We would like to find the relation between the small $J$-function and the $I$-function, so we will focus on the genus zero invariants with one marking $p_1$ (we denote the age by $\on{age}_{p_1}\in \mathbb Q$):
\[
\langle \gamma \bar{\psi}^{a}\rangle_{0,1,\beta}^{X_{D,r}},
\]
where the invariants vanish unless it satisfies the virtual dimension constraint:
\[
\on{vdim}:=\dim_{\mathbb C}X-3+(-K_X-D+D/r)\cdot\beta+1-\on{age}_{p_1}=\frac{1}{2}\deg(\gamma)+a.
\]
Note that we must have
\[
\on{age}_{p_1}-D\cdot\beta/r\in \mathbb Z.
\]
The marking $p_1$ will be distributed to $X$ or $Y$ under the degeneration. For $p_1$ to be distributed to $X$, we must have $\on{age}_{p_1}=0$. If we do not require $r$ to be large, we do not need to have $D\cdot\beta=0$. Instead, we have $D\cdot\beta/r\in \mathbb Z$. If $D\cdot\beta/r\not\in \mathbb Z$, the marking $p_1$ must be mapped to $D_\infty\subset Y$. The proof will be similar when we consider invariants with some extra insertions of the form $[1]_{k/r}$.

We consider the case where $-K_X-D$ is nef and $r$ is sufficiently large. 
Now we consider the degenerations that we studied in Section \ref{sec:degeneration}. We study the degeneration formula for the root stack $X_{D,r}$, for $r$ sufficiently large. The analysis of the degeneration formula for the twisted theory of $P_{X_\infty,r}$ works similarly. The degeneration formula is given in Section \ref{sec:degeneration}. The degeneration graph is described by the admissible triple \cite{Li01}*{Definition 4.11}. In genus zero, the graph is a tree. Let $V$ be the set of vertices and $E$ be the set of edges of the tree. Then
\[
|V|-|E|=1.
\]

Now we consider one-point invariants (the argument for invariants with extra insertions from $H^2(X)$ works the same). Let $v_0$ be the vertex that contains the distinguished marking. The distinguished marking carries all the virtual dimension. For each vertex $v$ that is not $v_0$, if it has only one edge, then the virtual dimension of the moduli space $\overline{M}_{v}$ is either $\dim_{\mathbb C}X-1$ or $\dim_{\mathbb C}X-2$ by the virtual dimension constraint and the assumption that $-K_X-D$ is nef. The insertion at the unique special point (the cohomology weight) is of degree $2(\dim_{\mathbb C}X-1)$ or $2(\dim_{\mathbb C}X-2)$. Vertices of the graph are connected by chains, and at the end of the chain is a vertex with one edge. We can then examine the cohomology weights on other vertices on the chain in a similar way.
Let $\eta$ be the cohomology weighted partition in the degeneration formula. By the virtual dimension constraint and the assumption that $-K_X-D$ is nef, we can see that all the cohomology classes in $\eta$ are of degree two or zero.

Therefore, all the invariants of $(Y_{D_\infty,r},D_0)$ that appear in the degeneration formula can be computed from certain extended $I$-functions of $(Y_{D_\infty,r},D_0)$. We can also compute these invariants of $(Y_{D_\infty,r},D_0)$ via the $\mathbb C^\times$-action. Then we will consider equivariant invariants of $N_{\mathcal D_r/X_{D,r}}$ with $(k+1)$ markings where there is one distinguished marking and $k$ markings each with insertion from $H^{\leq 2}(D)$. We have a similar argument for the twisted invariants of $P_{X_\infty,r}$ and we will have the same type of invariants of $N_{D/X}$ twisted by $\frac{e_{\mathbb C^\times}(\mathcal O(D))_{0,k+1,\beta}}{e_{\mathbb C^\times}(\mathcal O(D/r))_{0,k+1,\beta}}$. When we write down the $I$-function, the difference is simply the mirror map between the small equivariant $I$-function  and the $J$-function of $N_{D/X}$. In this case, the $I$-function and the $J$-function are related by the mirror map:
\[
J_{(X,D)}(\tau(Q,q),z)=I_{(X,D)}(\tau_D(z),z),
\]
where the $I$-function $I_{(X,D)}(\tau,z)$ of $(X,D)$ is a hypergeometric modification of the $J$-function of $X$ and $\tau_D$ is the mirror map in $D$ and $\tau(Q,q)$ is the usual mirror map computed from the $z^0$-coefficient of the $I$-function of $(X,D)$. The mirror map $\tau_D(z)$ in $D$ can be computed from the $z^{\geq 0}$-coefficient of the $I$-function (\ref{I-func-Y-D-0}). If $-K_D\cdot\beta>-c_1(N_{D/X})\cdot\beta$ for all $\beta\in \NE(D)$, then the mirror map is trivial. If $-K_D\cdot\beta\geq-c_1(N_{D/X})\cdot\beta$ for all $\beta\in \NE(D)$, then the mirror map is the $z^0$-coefficient of the $I$-function:
\[
\tau_D(z)=\tau_D:=\sum_{\substack{\beta\in \NE(D):d:=-K_D\cdot\beta\geq 2\\-K_D\cdot\beta=-c_1(N_{D/X})\cdot\beta}}\left(\langle [\on{pt}]\psi^{d-2}\rangle_{0,\beta,1}^D(-1)^{d-1}(d-1)!\right)D,
\]
where, under the degeneration, a marking whose insertion includes a class $D$ maps into the divisor and the class $D$ becomes $h\in H^2(Y)$. Otherwise, the mirror map $\tau_D(z)$ is a polynomial in $z$. And there will be classes $\delta\cup h\in H^*(Y)$ which, under the degeneration, is obtained from the class $\iota_!\delta\in H^*(X)$, where $\iota: D\hookrightarrow X$ is the inclusion map.

\subsection{The I-functions}

We can also obtain the extended $I$-function for $X_{D,r}$. We consider the extended data:
\[
S:=\{a_1,a_2,\ldots,a_k\}\subset\{0,1,\ldots,r-1\}.
\]
Then, instead of studying one-point invariants, we study invariants with several additional markings with degree zero insertions $[1]_{a_i}$. A similar degeneration argument also implies the following.
\begin{corollary}
The $S$-extended $I$-function $I_{X_{D,r}}^S(Q,q,x_{a},\tau_D,z)$ for $X_{D,r}$ is
\begin{equation}
\begin{split}
&I_{X_{D,r}}^S(Q,q,x_{a},\tau_D,z):=e^{D\log q/z}\sum_{\beta\in \on{NE}(X)} \sum_{(l_1,\ldots,l_k)\in (\mathbb Z_{\geq 0})^k}J_{X, \beta}(\tau_D,z)Q^{\beta}q^{D\cdot\beta}\frac{\prod_{i=1}^kx_{a_i}^{l_i}}{z^{\sum_{i=1}^k l_i}\prod_{i=1}^k (l_i!)}\\
&\cdot\frac{\prod_{ a\leq D\cdot\beta}(D+az)\prod_{\langle a \rangle =\langle (D\cdot\beta-\sum_{i=1}^k l_ia_i)/r\rangle, a\leq 0}(D/r+az)}{\prod_{ a\leq 0}(D+az)\prod_{\langle a \rangle =\langle (D\cdot\beta-\sum_{i=1}^k l_ia_i)/r\rangle, a\leq (D\cdot\beta-\sum_{i=1}^k l_ia_i)/r}(D/r+az)}[1]_{ (-D\cdot\beta+\sum_{i=1}^k l_ia_i)/r}.
\end{split}
\end{equation}
The $S$-extended $I$-function $I_{X_{D,r}}^S(Q,q,x_a,\tau_D,z)$ lies in  Givental's Lagrangian cone $\mathcal L_{X_{D,r}}$ of $X_{D,r}$.
\end{corollary}

Again, we can take the limit of $r\rightarrow \infty$ to get the $S$-extended $I$-function of $(X,D)$.

\begin{corollary}
The $S$-extended $I$-function $I^S_{(X,D)}(Q,q,x_a,\tau_D,z)$ for the smooth pair $(X,D)$:
\begin{equation*}
\begin{split}
I^S_{(X,D)}(Q,q,x_a,\tau_D,z)=I_++I_-,
\end{split}
\end{equation*}
where
\begin{align*}
I_+:=  e^{D\log q/z}\sum_{\substack{\beta\in \NE(X),\\ D\cdot\beta>\sum_{i=1}^k l_ia_i}} J_{X,\beta}(\tau_D,z) & Q^\beta q^{D\cdot \beta}  \frac{\prod_{i=1}^k x_{a_i}^{l_i}}{z^{\sum_{i=1}^k l_i}\prod_{i=1}^k (l_i!)}\\
& \cdot\frac{\prod_{ a\leq D\cdot\beta}(D+az)}{\prod_{ a\leq 0}(D+az)}\frac{ [1]_{-D\cdot\beta+\sum_{i=1}^k l_ia_i}}{D+(D\cdot\beta-\sum_{i=1}^k l_ia_i)z}
\end{align*}
and 

\begin{align*}
I_-:= e^{D\log q/z}\sum_{\substack{\beta\in \NE(X), \\ D\cdot\beta\leq\sum_{i=1}^k l_ia_i}} J_{X,\beta}(\tau_D,z) & Q^\beta q^{D\cdot\beta} \frac{\prod_{i=1}^k x_{a_i}^{l_i}}{z^{\sum_{i=1}^k l_i}\cdot\prod_{i=1}^k (l_i!)} \\
& \cdot\frac{\prod_{ a\leq D\cdot\beta}(D+az)}{\prod_{ a\leq 0}(D+az)} [1]_{-D\cdot\beta+\sum_{i=1}^k l_ia_i},
\end{align*}
 lies in Givental's Lagrangian cone   $\mathcal L_{(X,D)}$ of $(X,D)$.
\end{corollary} 

The insertions $[1]_{a_i}$ are of degree zero. The virtual dimension argument in Section \ref{sec:special} also works for extended $I$-functions.

Using the iterative construction \cite{TY20b} or the fiber product formula \cite{BNR24}, \cite{BNR22} and \cite{BNTY} for multi-root stacks, we can obtain the following results.
Given cohomological classes $\gamma_i\in H^*_{\on{CR}}(P_{m,X_\infty,\vec r})$ and nonnegative integers $a_i$, for $1\leq i\leq l$, we consider the equivariant invariant
\[
\left\langle \prod_{i=1}^l \tau_{a_i}(\gamma_i)\right\rangle_{0,l,\beta+\sum_{i=1}^m d_if_i}^{P_{m,X_\infty,\vec r},\oplus_{i=1}^m p_m^*\mathcal O_{P_i}(1)},
\]
where $f_i$ is the class of a fiber of $P_i$ and $d_i=D_i\cdot\beta$. The following result holds when the mirror maps in $D_i$ are trivial:
\begin{align}
\left\langle \prod_{i=1}^l \tau_{a_i}(\iota^*\gamma_i)\right\rangle_{0,l,\beta}^{X_{D,\vec r}}=\left[\left\langle \prod_{i=1}^l \tau_{a_i}(\gamma_i)\right\rangle_{0,l,\beta+\sum_{i=1}^m d_if_i}^{P_{m,X_\infty,\vec r},\oplus_{i=1}^mp_m^*\mathcal O_{P_i}(1)}\right]_{\vec\lambda=\vec 0},
\end{align}
where $\iota: X_{D,\vec r}\hookrightarrow P_{m,X_\infty,\vec r}$ is the inclusion as a zero locus of a section of $\oplus_{i=1}^mp_m^*\mathcal O_{P_i}(1)$ and $\vec \lambda=(\lambda_1,\ldots,\lambda_m)$ is the vector of equivariant parameters. Similarly to the case of smooth divisors, when $D_i$'s are nef, the mirror maps in $D_i$ are trivial.

We obtain the $I$-function $I_{X_{D,\vec r}}(Q,q,\tau_D,z)$ for the multi-root stack $X_{D,\vec r}$:
\begin{equation}
\begin{split}
e^{\sum_{i=1}^m D_i\log q_i/z}\sum_{\beta\in \on{NE}(X)}&  J_{X, \beta}(\tau_D,z)Q^{\beta}q^{D\cdot\beta}[1]_{ (-D_1\cdot\beta/r_1, \ldots, -D_m\cdot\beta/r_m)}\\
&\cdot\prod_{j=1}^m\frac{\prod_{ a\leq D_j\cdot\beta}(D_j+az)\prod_{\langle a \rangle =\langle D_j\cdot\beta/r\rangle, a\leq 0}(D_j/r_j+az)}{\prod_{ a\leq 0}(D_j+az)\prod_{\langle a \rangle =\langle D_j\cdot\beta/r\rangle, a\leq D_j\cdot\beta/r_j}(D_j/r_j+az)}.
\end{split}
\end{equation}
The $I$-function $I_{X_{D,\vec r}}(Q,q,\tau_D,z)$ lies in Givental's Lagrangian cone $\mathcal L_{X_{D,\vec r}}$ of $X_{D,\vec r}$.

Following \cite{TY20b} and \cite{TY20c}, the $I$-function for the corresponding snc pair $(X,D)$ is a large $\vec r$ limit of the $I$-functions for multi-root stacks $X_{D,\vec r}$.
The $I$-function $I_{(X,D)}(Q,q,\tau,z)$ for $(X,D)$ is
\begin{equation}\label{I-func-X-D}
\begin{split}
e^{\sum_{i=1}^m D_i\log q_i/z}\sum_{\beta\in \on{NE}(X)} J_{X, \beta}(\tau_D,z) & Q^{\beta}q^{D\cdot\beta}\prod_{j=1}^m\frac{\prod_{ a\leq D_j\cdot\beta}(D_j+az)}{\prod_{ a\leq 0}(D_j+az)}\\
&\cdot\prod_{j:D_j\cdot\beta>0}\frac{1}{D_j+D_j\cdot\beta z}[1]_{ -D\cdot\beta}.
\end{split}
\end{equation}
The $I$-function $I_{(X,D)}(Q,q,\tau_D,z)$, after applying the mirror map in $D_i$, lies in Givental's Lagrangian cone  $\mathcal L_{(X,D)}$ is of $(X,D)$.

For multi-root stacks, we can also consider the extended $I$-function with the extended data:
\[
S=\{\vec a_1,\ldots, \vec a_k\},
\]
where $\vec a_i\in (\mathbb Z_{\geq 0})^m$ and $a_{ij}<r_j$ for $j\in \{1,\ldots, m\}$. Then
the $S$-extended $I$-function $I_{X_{D,\vec r}}^S(Q,q,x_{\vec a},\tau_D,z)$ for $X_{D,\vec r}$ is
\begin{equation}
\begin{split}
&e^{\sum_{i=1}^m D_i\log q_i/z}\sum_{\beta\in \on{NE}(X)} \sum_{(l_1,\ldots,l_k)\in (\mathbb Z_{\geq 0})^k}\frac{J_{X, \beta}(\tau_D,z)Q^{\beta}q^{D\cdot\beta}\prod_{i=1}^kx_{\vec a_i}^{l_i}}{z^{\sum_{i=1}^k l_i}\prod_{i=1}^k (l_i!)}\\
&\cdot\prod_{j=1}^m\frac{\prod_{ a\leq D_j\cdot\beta}(D_j+az)\prod_{\langle a \rangle =\langle (D_j\cdot\beta-\sum_{i=1}^k l_ia_{ij})/r_j\rangle, a\leq 0}(D_j/r_j+az)}{\prod_{ a\leq 0}(D_j+az)\prod_{\langle a \rangle =\langle (D_j\cdot\beta-\sum_{i=1}^k l_ia_{ij})/r_j\rangle, a\leq (D_j\cdot\beta-\sum_{i=1}^k l_ia_{ij})/r_j}(D_j/r_j+az)}\\
& \cdot [1]_{( (-D_1\cdot\beta+\sum_{i=1}^k l_ia_{i1})/r_1,\ldots, (-D_m\cdot\beta+\sum_{i=1}^k l_ia_{im})/r_m)}.
\end{split}
\end{equation}
The $S$-extended $I$-function $I_{X_{D,r}}^S(Q,q,x_{\vec a},\tau_D,z)$ lies in   Givental's Lagrangian cone $\mathcal L_{X_{D,r}}$  of $X_{D,r}$.

Taking the large $\vec r$ limit, we have the $S$-extended $I$-function $I_{(X,D)}^S(Q,q,x_{\vec a},\tau_D,z)$ for the snc pair $(X,D)$:
\begin{equation}\label{extended-rel-I-func}
\begin{split}
&e^{\sum_{i=1}^m D_i\log q_i/z}\sum_{\beta\in \on{NE}(X)} \sum_{(l_1,\ldots,l_k)\in (\mathbb Z_{\geq 0})^k}\frac{J_{X, \beta}(\tau_D,z)Q^{\beta}q^{D\cdot\beta}\prod_{i=1}^kx_{\vec a_i}^{l_i}}{z^{\sum_{i=1}^k l_i}\prod_{i=1}^k (l_i!)}\\
&\cdot\prod_{j=1}^m\frac{\prod_{ a\leq D_j\cdot\beta}(D_j+az)}{\prod_{ a\leq 0}(D_j+az)}\prod_{j: D_j\cdot\beta-\sum_{i=1}^k l_ia_{ij}>0}\frac{1}{(D_j+(D_j\cdot\beta-\sum_{i=1}^k l_ia_{ij})z)}\\
& \cdot [1]_{ (-D_1\cdot\beta+\sum_{i=1}^k l_ia_{i1},\ldots, -D_m\cdot\beta+\sum_{i=1}^k l_ia_{im})}.
\end{split}
\end{equation}
The $S$-extended $I$-function $I_{(X,D)}^S(Q,q,x_{\vec a},\tau_D,z)$ lies in  Givental's Lagrangian cone  $\mathcal L_{(X,D)}$  of $(X,D)$.

\begin{remark}
As in \cite{FTY} and \cite{TY20b}, we usually set $Q^\beta q^{D\cdot\beta}=Q^\beta$, that is, $q=1$, as $q$ is an additional parameter coming from the ambient toric stack bundle. 
\end{remark}

We also consider the ambient  $I$-function $I_{(X,D),\on{amb}}(Q,q,\tau_D,z)$ for $(X,D)\subset (P_m,X_{\infty})$:
\begin{equation}\label{I-func-X-D-amb}
\begin{split}
e^{\sum_{i=1}^m D_i\log q_i/z}\sum_{\beta\in \on{NE}(X)} J_{X, \beta}(\tau_D,z) & Q^{\beta}q^{D\cdot\beta}\prod_{j=1}^m\frac{\prod_{ a\leq D_j\cdot\beta}(D_j+az)}{\prod_{ a\leq 0}(D_j+az)}\\
&\cdot\prod_{j:D_j\cdot\beta>0}\frac{1}{D_j+D_j\cdot\beta z}[1]_{ -D\cdot\beta}\cup_{i=1}^m h_i.
\end{split}
\end{equation}
And the $S$-extended $I$-function $I_{(X,D),\on{amb}}^S(Q,q,x_{\vec a},\tau_D,z)$ for $(X,D)\subset (P_m,X_{\infty})$:
\begin{equation}\label{extended-rel-I-func-amb}
\begin{split}
&e^{\sum_{i=1}^m D_i\log q_i/z}\sum_{\beta\in \on{NE}(X)} \sum_{(l_1,\ldots,l_k)\in (\mathbb Z_{\geq 0})^k}\frac{J_{X, \beta}(\tau_D,z)Q^{\beta}q^{D\cdot\beta}\prod_{i=1}^kx_{\vec a_i}^{l_i}}{z^{\sum_{i=1}^k l_i}\prod_{i=1}^k (l_i!)}\\
&\cdot\prod_{j=1}^m\frac{\prod_{ a\leq D_j\cdot\beta}(D_j+az)}{\prod_{ a\leq 0}(D_j+az)}\prod_{j: D_j\cdot\beta-\sum_{i=1}^k l_ia_{ij}>0}\frac{1}{(D_j+(D_j\cdot\beta-\sum_{i=1}^k l_ia_{ij})z)}\\
& \cdot [1]_{ (-D_1\cdot\beta+\sum_{i=1}^k l_ia_{i1},\ldots, -D_m\cdot\beta+\sum_{i=1}^k l_ia_{im})}\cup_{i=1}^m h_i.
\end{split}
\end{equation}

The $I$-functions in this section involve mirror maps in $D_i$ that we do not describe explicitly. In general, mirror maps can be quite complicated. Here we will consider a special case that we will use in the computation and the $I$-function is explicit. There are other interesting cases where the mirror maps in $D_i$ are not trivial, we plan to explore them elsewhere.

For the purpose of this paper, we focus on the case where $X$ is Fano and $\tau\in H^2(X)$. Then we have the following.
\begin{corollary}\label{cor-extended-rel-I-func}
For a Fano variety $X$ with an snc divisor $D$ and the parameter $\tau \in H^2(X)$, the $S$-extended $I$-function $I_{(X,D)}^S(Q,q,x_{\vec a},\tau,z)$ for $(X,D)$ is
\begin{equation}\label{extended-rel-I-func-Fano}
\begin{split}
&e^{\sum_{i=1}^m D_i\log q_i/z}\sum_{\beta\in \on{NE}(X)} \sum_{(l_1,\ldots,l_k)\in (\mathbb Z_{\geq 0})^k}\frac{J_{X, \beta}(\tau,z)Q^{\beta}q^{D\cdot\beta}\prod_{i=1}^kx_{\vec a_i}^{l_i}}{z^{\sum_{i=1}^k l_i}\prod_{i=1}^k (l_i!)}\\
&\cdot\prod_{j=1}^m\frac{\prod_{ a\leq D_j\cdot\beta}(D_j+az)}{\prod_{ a\leq 0}(D_j+az)}\prod_{j: D_j\cdot\beta-\sum_{i=1}^k l_ia_{ij}>0}\frac{1}{(D_j+(D_j\cdot\beta-\sum_{i=1}^k l_ia_{ij})z)}\\
& \cdot [1]_{ (-D_1\cdot\beta+\sum_{i=1}^k l_ia_{i1},\ldots, -D_m\cdot\beta+\sum_{i=1}^k l_ia_{im})}.
\end{split}
\end{equation}
The $S$-extended $I$-function $I_{(X,D)}^S(Q,q,x_{\vec a},\tau,z)$ lies in  Givental's Lagrangian cone  $\mathcal L_{(X,D)}$  of $(X,D)$.
\end{corollary}

\begin{remark}
We may allow for an additional class $\varphi \in H^*(X)$ in $\tau$. Then it comes with an additional mirror map associated with this class $\varphi$. 
\end{remark}

As we have already observed in Section \ref{sec:ex} for the relative-local model,  when $D_i\cdot\beta<0$, the $Q^\beta$-coefficient of the $J$-function of $X$ takes value in $D_i \cup H^*(X)$. Now we explain this for Fano varieties. 

\begin{lemma}
We consider an one-point invariant of the Fano variety $X$:
\begin{align}\label{1-pt-inv-X}
\langle [\gamma]\psi^{k}\rangle_{0,1,\beta}^X.
\end{align}
Let $D_{i_0}\subset X$ be a smooth divisor of $X$ such that $D_{i_0}\cdot\beta<0$. Let $\iota: D_{i_0}\hookrightarrow X$ be the inclusion map. Suppose $\gamma^\vee\not\in D_{i_0}\cup H^*(X)$, then the one-point invariant (\ref{1-pt-inv-X}) is zero. 
\end{lemma}
\begin{proof}
By assumption, the only marking carries all the virtual dimensions and is not mapped into the divisor $D_{i_0}$. Intuitively, $D_{i_0}\cdot\beta<0$ implies that there will be an irreducible component of the curve that is mapped into the divisor $D_{i_0}$. But there is not enough dimension to allow such curves. So, it is a contradiction. We explain it more precisely as follows.

The virtual dimension constraint is
\[
\dim_{\mathbb C}X-3+(-K_X\cdot\beta)+1=\frac 12 \deg[\gamma]+k.
\]
 We can consider the degeneration of $X$ to the normal cone of $D_{i_0}$. The degeneration graphs are bipartite graphs. The marking is away from $D_{i_0}$ and is attached to a vertex $v_0$. The virtual dimension constraint at $v_0$ is
\[
\dim_{\mathbb C}X-3+(-K_X\cdot\beta_{v_0})-D_{i_0}\cdot\beta+1+|\eta|=\frac 12 \deg[\gamma]+\sum \frac 12 \deg[\eta_i]+k,
\]
where $\eta=(\eta_1,\ldots, \eta_{|\eta|})$ is the set of cohomology weights of the relative markings and $\eta_i\in H^*(D)$. Each relative marking has contact order at least one with the divisor $D_{i_0}$. Therefore, $|\eta|\leq D_{i_0}\cdot\beta$. Suppose $\beta_{v_0}\neq \beta$.
 Then the virtual dimension of the moduli space associated to this vertex $v_0$ will be less than the degree of insertions. So, we must have $\beta_{v_0}=\beta$. If there are edges, then $D_{i_0}\cdot\beta>0$; If there are no edges, then $D_{i_0}\cdot\beta=0$. Both contradict the assumption that $D_{i_0}\cdot\beta<0$. 
\end{proof}

\subsection{The log Calabi--Yau case}\label{sec-log-cy}
Now, we examine the mirror map in the log Calabi--Yau case. Let $X$ be a Fano variety and $D=D_1+D_2\ldots+D_m$ be an snc anticanonical divisor of $X$. The mirror maps in $D$ are trivial, but we may still have relative mirror maps in the usual sense by studying the coefficients of the relative $I$-function. We consider the non-extended $I$-function with $\tau=\tau_{0,2}\in H^2(X)$:
\begin{equation*}
\begin{split}
e^{\sum_{i=1}^m D_i\log q_i/z}\sum_{\beta\in \on{NE}(X)} J_{X, \beta}(\tau_{0,2},z) & Q^{\beta}q^{D\cdot\beta}\prod_{j=1}^m\frac{\prod_{ a\leq D_j\cdot\beta}(D_j+az)}{\prod_{ a\leq 0}(D_j+az)}\\
&\cdot\prod_{j:D_j\cdot\beta>0}\frac{1}{D_j+D_j\cdot\beta z}[1]_{ -D\cdot\beta}.
\end{split}
\end{equation*}

Recall that, when $\beta=0$,
\[
J_{X, \beta}(\tau_{0,2},z)=z+\tau_{0,2}.
\]
 When $\beta\neq 0$,
\[
J_{X,\beta}(\tau_{0,2},z)=\sum_{k\geq 2, \alpha}\frac{1}{l!}\langle \tau_{0,2},\ldots,\tau_{0,2}, \psi^{k-2}\phi_{\alpha}\rangle_{0,l+1,\beta}^X\phi^\alpha (1/z)^{k-1},
\]
where 
\[
k=\dim_{\mathbb C}X-K_X\cdot\beta-\deg(\phi_\alpha)\geq -K_X\cdot\beta.
\]

Given a curve class $\beta$, we denote the intersection
\[
D_{\beta,-}:=\cap_{j: D_j\cdot\beta<0}D_j.
\]
In particular, if $D_j\cdot\beta\geq 0$ for all $j\in \{1,\ldots,m\}$, we have $D_{\beta,-}=X$. Then we indeed have the following bound for $k$: 
\[
k\geq -K_X\cdot\beta+\dim_{\mathbb C}X-\dim_{\mathbb C} D_{\beta,-}.
\]

If $D_j\cdot\beta< 0$, we have
\begin{align*}
\frac{\prod_{ a\leq D_j\cdot\beta}(D_j+az)}{\prod_{ a\leq 0}(D_j+az)}\cup D_j&=\frac{1}{\prod_{ D_j\cdot\beta<a< 0}(D_j+az)}\\
&=\frac{1}{(-D_j\cdot\beta-1)!(-z)^{-D_j\cdot\beta-1}\prod_{ D_j\cdot\beta<a< 0}\left(\frac{D_j}{az}+1\right)}.
\end{align*}
If $D_j\cdot\beta> 0$, we have
\begin{align*}
\frac{1}{D_j+D_j\cdot\beta z}\frac{\prod_{ a\leq D_j\cdot\beta}(D_j+az)}{\prod_{ a\leq 0}(D_j+az)}&=\prod_{0<a \leq  D_j\cdot\beta-1}(D_j+az)\\
&=(D_j\cdot\beta-1)!z^{D_j\cdot\beta-1}\prod_{0<a \leq  D_j\cdot\beta-1}\left(\frac{D_j}{az}+1\right).
\end{align*}

Since $X$ is Fano, for each $\beta\in \NE(X)$, there is always at least one $j$ such that $D_j\cdot\beta>0$. Combining the above calculations, the power of $(1/z)$ is bounded below by
\begin{align*}
& -K_X\cdot\beta+\dim_{\mathbb C}X-\dim_{\mathbb C} D_{\beta,-}-1+ \sum_{j: D_j\cdot\beta< 0}(-D_j\cdot\beta-1)+\sum_{j:D_j\cdot\beta> 0}(-D_j\cdot\beta+1)\\
=&\#\{j:D_j\cdot\beta> 0\}-1.
\end{align*}
 Hence, we can write the $I$-function $I_{(X,D)}(Q,q,\tau_{0,2},z)$ as
\begin{align*}
&z+\tau_{0,2}+\sum_{i=1}^m D_i\log q_i\\
&+\sum_{i=1}^m\sum_{\substack{\beta\in \NE(X),k\geq 2 \\ d_i>0,d_j\leq 0, \text{ for } j \neq i}}\langle [\on{pt}]_{D_{\beta,-}}\psi^{k-2}\rangle_{0,1,\beta}^X e^{\tau_{0,2}\cdot\beta} Q^\beta q^{\vec d}\frac{(d_i-1)!}{\prod_{j: d_j<0}(-d_j-1)!}[1]_{-\vec d}\\
&+\sum_{k\geq 1} I_{-k}z^{-k},
\end{align*}
where 
\[
D\cdot\beta:=\vec d:=(d_1,\ldots,d_m):=(D_1\cdot\beta,D_2\cdot\beta,\ldots,D_m\cdot\beta)
\]
and $[\on{pt}]_{D_{\beta,-}}$ denotes the point class of $D_{\beta,-}\subset X$. More precisely, $[\on{pt}]_{D_{\beta,-}}\in H^*(X)$ is Poincar\'e dual to the class $\cup_{j: D_j\cdot\beta<0}[D_j]\in H^*(X)$.

The $z^0$-coefficient, denoted by $\tau(Q,q)$, is the (non-extended) relative mirror map:
\begin{equation}\label{mirror-map}
\begin{split}
&\tau_{0,2}+\sum_{i=1}^m D_i\log q_i\\
+&\sum_{i=1}^m\sum_{\substack{\beta\in \NE(X),k\geq 2 \\ d_i>0,d_j\leq 0, \text{ for } j \neq i}}\langle [\on{pt}]_{D_{\beta,-}}\psi^{k-2}\rangle_{0,1,\beta}^XQ^\beta q^{\vec d}\frac{(d_i-1)!}{\prod_{j: d_j<0}(-d_j-1)!}[1]_{-\vec d}.
\end{split}
\end{equation}
The relative mirror theorem, Corollary \ref{cor-extended-rel-I-func}, gives the following:
\[
\iota_*J_{(X,D)}(\tau(Q,q),z)=I_{(X,D),\on{amb}}(Q,q,\tau_{0,2},z).
\]

The following condition ensures that the relative mirror map (\ref{mirror-map}) is trivial:
\[
\#\{j: D_j\cdot\beta> 0\}\geq 2, \text{ for all } \beta\in \NE(X).
\]
We would like to point out that the main calculations in this paper do not rely on this condition.

We can also consider the extended $I$-function (\ref{cor-extended-rel-I-func}). For the purpose of the proof of the mirror symmetric Gamma conjecture, we consider the extended $I$-function with the extended data:
\[
S=\{\vec e_i\}_{i=1}^m\cup \{\vec{\mathbf b}\}.
\]
We require $\vec{\mathbf b}$ to be of mid-age such that $D_{\vec{\mathbf b}}$ is a zero-dimensional stratum corresponding to a maximal cone $\sigma$. We write the index set $I_{\sigma}:=I_{\vec{\mathbf b}}$. For $i\in I_\sigma$, $\mathbf b_i$ is a mid-age; otherwise, $\mathbf b_i=0$.

In this case, the extended mirror map $\tau(Q,q,x_{\vec a})$ is
\begin{equation}\label{extended-mirror-map}
\begin{split}
&\tau_{0,2}+\sum_{i=1}^m D_i\log q_i+\sum_{i=1}^mx_{\vec e_i}[1]_{\vec e_i}\\
& + \left([1]_{\vec{\mathbf b}}+\sum_{\substack{\beta:d_i=0, \text{ for } i\not\in I_\sigma,\\ k\geq 2}} \langle [\on{pt}]_{D_{\beta,-}}\psi^{k-2}\rangle_{0,1,\beta}^X \frac{\prod_{i: d_i>0} d_i!}{\prod_{j: d_j<0}(-d_j-1)!}Q^\beta q^{\vec d}[1]_{\vec{\mathbf b}-\vec d}\right)x_{\vec{\mathbf b}}\\
&+\sum_{i=1}^m\sum_{\substack{\beta\in \NE(X),k\geq 2 \\ d_i>0,d_j\leq 0, \text{ for } j \neq i}}\langle [\on{pt}]_{D_{\beta,-}}\psi^{k-2}\rangle_{0,1,\beta}^XQ^\beta q^{\vec d}\frac{(d_i-1)!}{\prod_{j: d_j<0}(-d_j-1)!}[1]_{-\vec d}.
\end{split}
\end{equation}
The variables 
\[
x_{\vec a}=(x_{\vec e_1},x_{\vec e_2},\ldots,x_{\vec e_m},x_{\vec{\mathbf b}})
\]
 are the variables corresponding to the extended data. We have the following relation:
\begin{align}\label{rel-J-I}
\iota_*J_{(X,D)}(\tau(Q,q),x_{\vec a},z)=I_{(X,D),\on{amb}}(Q,q,x_{\vec a},\tau_{0,2},z).
\end{align}

If $\{D_i^*\}_{i=1}^m$ span the cone of effective curve classes, then all these mirror maps will be trivial.

Given a monoid ideal $I\subset \NE(X)$, we can restrict the relation (\ref{rel-J-I}) between the $I$-function and the $J$-function to $\beta\in \NE(X)\setminus I$.

\begin{remark}
Recall that a pair of mid-age is a pair of ages $k_{ai}/r_i, k_{bi}/r_i$ along $D_i$ such that $k_{ai},k_{bi}$ are sufficiently bigger than $D_i\cdot\beta$ and $k_{ai}+k_{bi}=r_i+c_i$ for certain constant $c_i\in \mathbb Z$. We may need to choose a different $\vec{\mathbf b}$ for a different $\beta$. If there is a finite collection of $\beta$, for example, when $\beta \in \NE(X)\setminus I$ and $\NE(X)\setminus I$ is a finite set, then we can choose a mid-age $\vec{\mathbf b}$ once for all. In (\ref{extended-mirror-map}), we may consider the sum of $\beta$ over $\NE(X)\setminus I$ instead of $\NE(X)$. Alternatively, we may consider $\vec{\mathbf b}$ as a formal symbol for `$\infty$'. What matters is that the mid-ages always come in pairs and the sum of two mid-ages is a constant (a small or a large age).  
\end{remark}

If there is a class $\varphi\in H^*(X)$ with $\deg(\varphi)>0$, the mirror map associated with $\varphi$ will be non-trivial in general, and the computation of the invariants involves the Birkhoff factorization procedure. This is because there will be non-zero coefficients for positive powers of $z$. However, we do not need to compute the invariants of $(X,D)$ here. Instead, we are interested in obtaining a formula for the $J$-function of $X$ which is used to define the $I$-function of $(X,D)$. Write
\begin{align}\label{mirror-map-varphi}
\varphi(z)
\end{align}
for the mirror map associated with $\varphi$. We do not give an explicit formula for $\varphi(z)$ as it can be very complicated, in general. 

In the $J$-function of $(X,D)$, we take the parameter $\tau=\tau^\prime+\tau_{0,2}$ and
\[
\tau^\prime=[\varphi(z)]x_{\varphi}+\sum_{i=1}^m [1]_{\vec e_i}x_{\vec e_i}+[1]_{\vec{\mathbf b}}x_{\vec{\mathbf b}}.
\]
We can take the $x_{\varphi}^1$-coefficient of the $J$-function and the $I$-function. In the $I$-function of $(X,D)$, we will have $\mathbb J_X(\tau_{0,2},z)\varphi$. The mirror map $\varphi(z)$ associated with the class $\varphi$ is a polynomial in $z$ of degree $\frac{1}{2}\deg(\varphi)-1$. 

\begin{remark}
In Section \ref{sec:func-mirror}, we will consider the mirror function $\check{\varphi}$ associated with $\varphi$. The function $\check{\varphi}(z)$ that appears in the mirror symmetric Gamma conjecture is the mirror function associated with $\varphi(z)$.  Therefore, the mirror map is encoded in $\check{\varphi}(z)$.
\end{remark}

\subsection{Absolute $J$-function in terms of relative}

Now we consider the relative mirror theorem for $(X,D)$. When the mirror map is trivial, we can try to solve the small $J$-function of $X$. When the mirror map is non-trivial, we need to include the inverse of the mirror map. We ignore the mirror map for now for simplicity. The following result does not require that the divisor $D$ has zero dimensional strata.

Consider an effective curve class $\beta\in \NE(X)$. After rearranging the order of $D_i$, we can assume that 
\[
d_i=D_i\cdot\beta>0 \text{ for }i=1,\ldots, m_+, \quad d_i=D_i\cdot\beta=0\text{ for } i=m_+ +1,\ldots, m_++m_0,
\]
and 
\[
d_i=D_i\cdot\beta<0\text{ for }i=m_++m_0+1,\ldots, m.
\]
We consider the extended $I$-function of $(X,D)$ such that
\[
S=\{\vec e_i\}_{i=1}^{m_+}.
\]  
Now we consider the $\prod_{i=1}^{m_+}x_{\vec e_i}^{d_i}$-coefficient of the $S$-extended $I$-function of $(X,D)$ that takes value in $\mathfrak H_{\vec d-\sum_{i=1}^{m_+}d_i\vec e_i}=H^*(D_{\vec d-\sum_{i=1}^{m_+}d_i\vec e_i})$, we have
\[
e^{\sum_{i=1}^m D_i\log q_i/z} J_{X,\beta}(\tau_{0,2},z)Q^\beta q^{\vec d}\frac{1}{\prod_{i=1}^{m_+}d_i!z^{d_i}}\prod_{j=1}^m \frac{\prod_{a\leq d_j}(D_j+az)}{\prod_{a\leq 0}(D_j+az)}[1]_{-\sum_{i=m_++m_0+1}^m  d_i\vec e_i}.
\] 
The corresponding coefficient of the $J$-function of $(X,D)$ is
\begin{align*}
e^{-\sum_{i=1}^m D_i\log x_i/z}\sum_{[\gamma]} & \langle \prod_{i=1}^{m_+}[1]_{\vec e_i}^{d_i}, \tau_{0,2},\ldots,\tau_{0,2},[\gamma]_{\sum_{i=m_++m_0+1}^m  d_i\vec e_i}\psi^{k} \rangle_{0,1+l+\sum_{i=1}^{m_+}d_i,\beta}^{(X,D)}\\
& \cdot[\gamma^\vee]_{-\sum_{i=m_++m_0+1}^m  d_i\vec e_i}\frac{t^\beta x^{-\vec d}}{l!\prod_{i=1}^{m_+}d_i!},
\end{align*}
where $\gamma$ is pulled back from the cohomology of the ambient space $X$ and the class $[\gamma^\vee\cup_{i=m_+=m_0+1}^m D_i] \in H^*(X)$ is the Poincar\'e dual of $\gamma$ (we omit the pullback sign in the notation); and $k$ is determined by the discrete data, the insertion $[\gamma]$ and the virtual dimension constraint. 
Using the relative mirror theorem, we have $Q=t$ and $q_i=x_i^{-1}$. Then,
\begin{align*}
& J_{X,\beta}(\tau_{0,2},z)\frac{1}{\prod_{i=1}^{m_+}z^{d_i}}\prod_{j=1}^m \frac{\prod_{a\leq d_j}(D_j+az)}{\prod_{a\leq 0}(D_j+az)}[1]_{-\sum_{i=m_++m_0+1}^m  d_i\vec e_i}\\
=& \sum_{[\gamma]}\langle \prod_{i=1}^{m_+}[1]_{\vec e_i}^{d_i}, \tau_{0,2},\ldots,\tau_{0,2},[\gamma]_{\sum_{i=m_++m_0+1}^m  d_i\vec e_i}\psi^{k} \rangle_{0,1+l+\sum_{i=1}^{m_+}d_i,\beta}^{(X,D)}[\gamma^\vee]_{-\sum_{i=m_++m_0+1}^m  d_i\vec e_i}\frac{1}{l!}.
\end{align*}
We can solve the $J$-function of $X$:
\begin{align*}
J_{X,\beta}(\tau_{0,2},z)= &\sum_{[\gamma]}\langle \prod_{i=1}^{m_+}[1]_{\vec e_i}^{d_i},\tau_{0,2}, \ldots,\tau_{0,2}, [\gamma]_{\sum_{i=m_++m_0+1}^m  d_i\vec e_i}\psi^{k} \rangle_{0,1+l+\sum_{i=1}^{m_+}d_i,\beta}^{(X,D)}\\
&\cdot [\gamma^\vee]\frac{z^{\sum_{i=1}^{m_+} d_i}}{l!}\prod_{j=1}^m \frac{\prod_{a\leq 0}(D_j+az)}{\prod_{a\leq d_j}(D_j+az)},
\end{align*}
in $H^*(X)$. Note that this part of the $J$-function of $X$ takes values in $\cup_{i: d_i< 0} D_i \cup H^*(X)$. We then have the following result, which may be considered as a mirror theorem for Fano varieties. 
\begin{theorem}\label{thm-mirror-X}
The $J$-function of a Fano variety $X$ can be written in terms of invariants of the log Calabi--Yau pair $(X,D)$:
\begin{equation}\label{iden-mirror-theorem-X}
\begin{split}
& J_{X}(\tau_{0,2},Q,x,z)\\
= & e^{-\sum_{i=1}^m D_i\log x_i/z}\sum_{\beta\in \NE(X),[\gamma]}\langle \prod_{i: d_i>0}[1]_{\vec e_i}^{d_i}, \tau_{0,2},\ldots,\tau_{0,2},[\gamma]_{\sum_{i: d_i\leq 0}  d_i\vec e_i}\psi^{k} \rangle_{0,1+l+\sum_{i:d_i>0}d_i,\beta}^{(X,D)}\\
& \cdot [\gamma^\vee]\frac{z^{\sum_{i: d_i>0} d_i}}{l!}\prod_{j=1}^m  x_i^{-d_i}Q^\beta\frac{\prod_{a\leq 0}(D_j+az)}{\prod_{a\leq d_j}(D_j+az)}.
\end{split}
\end{equation}
\end{theorem}

\begin{remark}
We can also state the mirror theorem of $X$ when the relative mirror map of $(X,D)$ is not trivial. Then the LHS and the RHS of Equation (\ref{iden-mirror-theorem-X}) are equal after imposing the relative mirror map. For example, when $D$ is a smooth anticanonical divisor of $X$, we can use the relative mirror theorem of \cite{FTY} and the relative mirror map computed in \cite{You22} to write the $J$-function of $X$ in terms of relative invariants of $(X,D)$ and a hypergeometric modification.
\end{remark}

We consider the following operator of the pair $(X,D)$:
\begin{definition}
The $\mathbb J$-operator of $(X,D)$ is defined as
\begin{align*}
\mathbb J_{(X,D)}(\tau_{0,2},z)(\phi):=&e^{\left(-\sum_{i=1}^m D_i\log x_i+\tau_{0,2}\right)/z}\phi\\
&+e^{-\sum_{i=1}^m D_i\log x_i/z}\sum_{\beta\in \NE(X)\setminus \{0\}}\sum_{k,\vec s}\langle \phi,\tau_{0,2},\ldots,\tau_{0,2}, \frac{T_{\vec s,k}}{z-\bar{\psi}}\rangle_{0,2+l,\beta}^{(X,D)} T_{-\vec s}^k \frac{ x^{-\vec d}}{l!}.
\end{align*}
\end{definition}

Then the TRR in \cite{TY20c}*{Proposition 26} implies that the $J$-function of $X$ can be written as
\begin{align}\label{mirror-thm-X}
J_{X}(\tau_{0,2},z)=\left[ \mathbb J_{(X,D)}(\tau_{0,2},z)\left(\prod_{i=1}^{m_+} \vartheta_{\vec e_i}^{d_i}\right)\right]_{[\gamma^\vee]_{-\sum_{i=m_++m_0+1}^m  d_i\vec e_i}}[\gamma^\vee]z^{\sum_{i=1}^{m_+} d_i}\prod_{j=1}^m \frac{\prod_{a\leq 0}(D_j+az)}{\prod_{a\leq d_j}(D_j+az)}.
\end{align}
Furthermore, we have
\[
\mathbb J_{(X,D)}(\tau_{0,2},z)\left(\prod_{i=1}^{m_+} \vartheta_{\vec e_i}^{d_i}\right)=\sum_{\vec p\in B(\mathbb Z)}\left[ \prod_{i=1}^{m_+} \vartheta_{\vec e_i}^{d_i}\right]_{\vartheta_{\vec p}}\mathbb J_{(X,D)}(\tau_{0,2},z)([1]_{\vec p}).
\]

\section{The mirror construction}\label{sec:mirror-constr}

Let $(X,D)$ be a log Calabi--Yau pair and $B$ be pure-dimensional with $\dim_{\mathbb R}B=\dim_{\mathbb C}X$. By log Calabi--Yau, we mean $D\in |-K_X|$. We also assume that the strata of $D$ are connected. This is the absolute case in \cite{GS21}*{Assumption 1.1} with all the divisors being good. Therefore, every stratum of $D$ contains a zero-dimensional stratum. For each one-dimensional stratum $D_I$, we have $D_I\cong \mathbb P^1$ by \cite{GS21}*{Proposition 1.3}.

The intrinsic mirror construction of Gross--Siebert for the pair $(X,D)$ was given in \cite{GS19}. We recall a variant of this construction using orbifold invariants \cite{TY20c}. Since we do not aim to explore geometric duality or ``intrinsicity" of this variation of the mirror construction, it might be better not to refer to it as a variant of the intrinsic mirror symmetry. We will refer to it as a version of relative mirror symmetry for the pair $(X,D)$.

Let $\NE(X)\subset H_2(X)$ be a finitely generated submonoid, containing all effective curve classes, and let the group of invertible elements of $\NE(X)$ coincide with the torsion part of $H_2(X)$. Let $I\subset \NE(X)$ be a monoid ideal such that $\NE(X)\setminus I$ is finite. That is,
\[
S_I:=\mathbb C[\NE(X)]/I
\]
is Artinian. Then we consider a free $S_I$-module
\[
R_I:=\bigoplus_{\vec p \in B(\mathbb Z)} S_I \vartheta_{\vec p},
\]
where the theta functions satisfy the product rule 
\begin{align}\label{prod-rule}
\vartheta_{\vec p_1}\star\vartheta_{\vec p_2}=\sum_{\vec r\in B(\mathbb Z),\beta}N^\beta_{\vec p_1,\vec p_2,-\vec r}t^\beta \vartheta_{\vec r},   
\end{align}
where $\beta\in \NE(X)\setminus I$. 
In \cite{GS19}, the structure constants $N^\beta_{\vec p_1,\vec p_2,-\vec r}$ are punctured logarithmic invariants of $(X,D)$ defined in \cite{ACGS}. Here, we use orbifold invariants of root stacks defined in \cite{TY20c}. By \cite{TY20c}*{Theorem 37}, the orbifold structure constants define a commutative, associative $S_I$-algebra structure on $R_I$ with unit given by $\vartheta_0$.  

\begin{remark}
The vector $\vec r$ in (\ref{prod-rule}) represents the contact order. In previous sections, we used $\vec r$ as the rooting parameters of the multi-root stacks when we explained the definition of orbifold invariants. Since we have already taken the limit $\vec r\rightarrow \infty$, we will no longer need to specify the rooting parameters. Therefore, from now on, $\vec r$ always means contact orders, similar to $\vec p, \vec q$. Since $\vec r$ is standard notation for the product rule in the literature, we decide to use this notation here as well. We hope that this does not cause any confusion.
\end{remark}

Recall that in \cite{GRZ} and \cite{You22}, one sets 
\[
Q^\beta=t^\beta x^{-D\cdot\beta},
\]   
where $x$ corresponds to the anticanonical direction there.
Here, we set
\begin{align}\label{Q=x}
(t/Q)^\beta=x^{D\cdot\beta}=x_1^{D_1\cdot\beta}\cdots x_m^{D_m\cdot\beta}.
\end{align}
Note that $x_i$ are local coordinates that correspond to the divisor $D_i$. For toric varieties with their toric invariant divisors, the relation (\ref{Q=x}) is the usual relation that we have in the Hori--Vafa mirror.

For the mirror construction, one defines families of schemes
\[
\on{Spec} R_I\rightarrow \on{Spec} S_I.
\]
One obtains a formal flat family of affine schemes by taking the direct limit of the above families of schemes
\begin{align}\label{mirror-constr}
\check{\mathfrak X}\rightarrow \on{Spf}\widehat{\mathbb C[\NE(X)]},
\end{align}
where $\widehat{\mathbb C[\NE(X)]}$ is the completion of $\mathbb C[\NE(X)]$ with respect to the maximal ideal $\NE(X)\setminus \NE(X)^\times$. This is considered as the mirror family to $X\setminus D$. For the mirror to the pair $(X,D)$, we also need to turn on the potential. Following Section 0.4 of the first arXiv version of \cite{GHK}, the Landau--Ginzburg potential is the following
\[
W=\sum_{i=1}^m \vartheta_{[D_i]},
\]
where $\vartheta_{[D_i]}:=\vartheta_{\vec e_i}$ and $\{\vec e_i\}_{i=1}^m$ is the set of standard unit vectors. However, since the orbifold invariants are not the same as the logarithmic invariants, the potential $W$ here may not be the same as the potential defined using logarithmic invariants. In Section \ref{sec:mirror-map}, we define theta functions by composing the mirror map in the definition. Then all the computation still works. It aligns with the expectation in \cite{You22a} that the difference between orbifold and logarithmic invariants can be described in terms of relative mirror maps.

For a log Calabi--Yau pair $(X,D)$, where $D$ does not have a zero-dimensional stratum, we need to consider a maximally unipotent monodromy degeneration $g: \mathcal X\rightarrow S$ of $(X,D)$. Then the intrinsic mirror \cite{GS19} of $X\setminus D$ is constructed as the projective spectrum of the degree zero part of the relative quantum cohomology of $(\mathcal X,\mathcal D)$, where $\mathcal D$ is a certain divisor of $\mathcal X$ that includes $g^{-1}(0)$. This is  the relative case in \cite{GS19} and \cite{GS21}.

\section{Theta functions}\label{sec:theta-func}

We will use the Landau--Ginzburg potential defined using theta functions to study the mirror symmetric Gamma conjecture. For the LHS of (\ref{iden-gamma-conj}), we need to consider the Taylor expansion of $e^{-W/z}$, where $W=\sum_{i=1}^m\vartheta_{[D_i]}$. Therefore, we need to consider the product of theta functions. Recall that theta functions satisfy the following product rule. For $\vec p_1,\vec p_2\in B(\mathbb Z)$,
\begin{align}\label{theta-func-prod}
\vartheta_{\vec p_1}\star\vartheta_{\vec p_2}=\sum_{\vec r\in B(\mathbb Z),\beta}N^\beta_{\vec p_1,\vec p_2,-\vec r}t^\beta \vartheta_{\vec r},    
\end{align}
where the structure constants $N^\beta_{\vec p_1,\vec p_2,-\vec r}$ are three-point (punctured log or orbifold) invariants of $(X,D)$ with contact orders $\vec p_1,\vec p_2$ and $-\vec r$.

 We will consider more general products of theta functions, but we would first like to recall the Frobenius structure conjecture as a statement of a special coefficient of the product. The Frobenius structure conjecture can be partially rephrased to state that the coefficient of $\vartheta_0$ in the product $\vartheta_{\vec p_1}\star\cdots\star \vartheta_{p_l}$ is
\[
\sum_{\beta}N^\beta_{\vec p_1,\ldots,\vec p_l,0}t^\beta.
\]
This was proved in \cite{TY20c}*{Theorem 38} when the structure constants are orbifold invariants and when $(K_X+D)$ is nef or anti-nef. This coefficient is related to the classical period of $W$. Hence it is related to the mirror symmetric Gamma conjecture for $\mathcal O_{\on{pt}}$. We will compute them in Section \ref{sec:gamma-pt-sheaf}.

We also recall the definition of the orbifold theta function $\vartheta_{\vec r}$ from \cite{You24}:

\begin{definition}[=\cite{You24}, Definition 1.6]\label{def-theta}
For an snc log Calabi--Yau pair $(X,D)$,    the theta functions are defined as follows: Fix $\vec r \in B(\mathbb Z)\setminus \{0\}$. Let $\sigma_{\on{max}}\in \Sigma(X)$ be a maximal cone of $\Sigma(X)$ such that $p\in \sigma_{\on{max}}$, then
    \begin{align*}
   \vartheta_{\vec r}(p)&:=
   \sum_{\vec k\in {\mathbb Z}^m} \sum_{\beta: D_i\cdot \beta=k_i+r_i}N_{\vec r, \vec{\mathbf b}, -\vec{\mathbf{b}}+\vec k}^\beta t^{\beta} x^{-\vec k},
    \end{align*}
where the invariants are three-point genus zero orbifold invariants of $(X,D)$ with contact orders (or ages) $\vec r, \vec{\mathbf b}$ and, $-\vec{\mathbf{b}}+\vec k$ and, the last two markings are mid-age markings along $D_i$ for $i\in I_\sigma$ for $\sigma=\sigma_{\on{max}}$. 
\end{definition}

\begin{remark}\label{rmk-theta}
    In Definition \ref{def-theta}, if we take $\beta \in \NE(X)$, then we consider $\vartheta_{\vec r}(p)$ as a formal power series. To avoid convergence issues, we take $\beta \in \NE(X)\setminus I$, for $I\subset \NE(X)$ a monoid ideal such that $\NE(X)\setminus I$ is finite. Then there are only finitely many terms in $\vartheta_{\vec r}(p)$. In this case, we can choose one mid-age $\vec{\mathbf b}$ for all $\beta\in \NE(X)\setminus I$. This is not true when there are infinitely many terms and we may need to choose different $\vec{\mathbf b}$ for different $\beta$. Most of the computations in this paper work for both $\beta \in \NE(X)$ and $\beta \in \NE(X)\setminus I$. We will specify it when necessary.
\end{remark}

\begin{remark} Compared with \cite{You24}*{Definition 1.6}, we made a change of variable $x\rightarrow x^{-1}$, so we have $x^{-\vec k}$ instead of $x^{\vec k}$. The reason for the change is that it seems more natural in the computation in later sections. \end{remark}

\begin{remark}
Note that in the definition of theta functions, we have $x^{-\vec k}=x^{-\vec d}x^{\vec r}$, where $\vec d=(D_1\cdot\beta,\ldots,D_m\cdot\beta)$. We can consider the factor $x^{-\vec d}$ as coming from $-D_i\log x_i$ in the parameter $\tau_{0,2}$ of the relative quantum product.
\end{remark}

\begin{remark}
    We recall that we may think of a pair of mid-age markings as a pair of relative markings with one sufficiently large positive contact order $\mathbf b_i$ and one sufficiently large negative contact order $-\mathbf b_i+k_i$ along $D_i$. The invariants may depend on $k_i$ but do not depend on the value of $\mathbf b_i$ for $\mathbf b_i$ sufficiently large. In \cite{GS21}, logarithmic theta functions are in terms of punctured invariants of the broken line type. The orbifold invariants here may be considered as an orbifold analogous to the broken line type invariants. It is worth pointing out that although mid-age invariants are defined using root stacks, these invariants are indeed intrinsic to the geometry of the pair $(X,D)$ and can be defined without mentioning orbifold Gromov--Witten theory. But in many cases, considering them as orbifold invariants of root stacks makes it easier to study them.
\end{remark}

Now we consider a generalization of Definition \ref{def-theta} by defining theta functions that depend on a parameter $\tau$ with $\deg^0(\tau)\leq 1$. 
\begin{definition}\label{def-theta-tau}
For an snc log Calabi--Yau pair $(X,D)$, let $\tau\in \mathfrak H$ be a cohomology class with $\deg^0(\tau)\leq 1$.  The theta functions are defined as follows: Fix $\vec r \in B(\mathbb Z)\setminus \{0\}$. Let $\sigma_{\on{max}}\in \Sigma(X)$ be a maximal cone of $\Sigma(X)$ such that $p\in \sigma_{\on{max}}$, then
    \begin{align*}
   \vartheta_{\tau,\vec r}(p)&:=
   \sum_{\beta}\sum_{\vec k\in {\mathbb Z}^m}N_{\tau,\vec r, \vec{\mathbf b}, -\vec{\mathbf{b}}+\vec k}^\beta t^{\beta} x^{-\vec k},
    \end{align*}
with
\[
N_{\tau,\vec r, \vec{\mathbf b}, -\vec{\mathbf{b}}+\vec k}^\beta=\sum_{l\geq 0}\frac{1}{l!}\langle [1]_{\vec r},\tau,\ldots,\tau, [1]_{\vec{\mathbf b}}, [\on{pt}]_{-\vec{\mathbf{b}}+\vec k}\rangle_{0,l+3,\beta}^{(X,D)},
\]
where the last two markings are mid-age markings along $D_i$ for $i\in I_\sigma$ for $\sigma=\sigma_{\on{max}}$; the degree $\deg^0(\tau)$ is defined in (\ref{deg-0}).  
\end{definition}

Note that, for mid-age orbifold invariants of root stacks for the pair $(X,D)$, we need to choose the roots to be sufficiently large with respect to $l$ and $\beta$. In other words, we first fix all the topological data and then choose large roots. Then we choose mid-ages with respect to the roots.

By the virtual dimension constraint, we can simply restrict the parameter $\tau$ of the theta function to its degree one part. However, we write $\deg^0(\tau)\leq 1$ instead of simplifying it to $\deg^0(\tau)=1$ to make it more compatible with the discussion in later sections.  

Note that for a class $\tau \in \mathfrak H_{\vec k}$ of real degree $d$ in $H^*(D_{\vec k})$,  we have $\deg^0(\tau)=d/2+\#\{i:k_i<0\}$. Hence, the condition $\deg^0(\tau)=1$ does not mean that the class $\tau$ is not just the divisor classes in $H^2(X) \subset \mathfrak H_{\vec 0}$. For example, we can have the class $[1]_{-\vec k}$ such that $\#\{i:k_i<0\}=1$. The situation will be more complicated if we do not assume $\deg^0(\tau)\leq 1$ and we will need to take into account of more general functions. We will study this situation in the next section.

Considering the restriction of the big quantum product of $[1]_{\vec p}$, for $\vec p\in B(\mathbb Z)$, with the parameter $\tau$, we can see that theta functions $\vartheta_{\tau,\vec p}$ satisfy the following product rule:
\[
\vartheta_{\tau,\vec p_1}\star\vartheta_{\tau,\vec p_2}=\sum_{\vec r\in B(\mathbb Z),\beta}N^\beta_{\tau,\vec p_1,\vec p_2,-\vec r}t^\beta \vartheta_{\vec r},   
\]
where
\[
    N^\beta_{\tau,\vec p_1,\vec p_2,-\vec r}:=\sum_{l\geq 0}\frac{1}{l!}\langle [1]_{\vec p_1},[1]_{\vec p_2},\tau,\ldots,\tau,[\on{pt}]_{-\vec r}\rangle_{0,l+3,\beta}^{(X,D)}. 
    \]
    
We can compute the product of theta functions with a parameter $\tau\in \mathfrak H$.  We have
\begin{prop}\label{prop-prod-theta}
When $-K_X-D$ is nef or anti-nef. Let $\deg^0(\tau)\leq 1$,
\[
    \vartheta_{\tau,\vec p_1}\star\cdots \star \vartheta_{\tau,\vec p_{l^\prime}}=\sum_{\vec r\in B(\mathbb Z)}\sum_{\beta} N^\beta_{\tau,\vec p_1,\ldots,\vec p_{l^\prime},-\vec r}t^\beta\vartheta_{\vec r},\]
    where 
    \[
    N^\beta_{\tau,\vec p_1,\ldots,\vec p_{l^\prime},-\vec r}:=\sum_{l\geq 0}\frac{1}{l!}\langle [1]_{\vec p_1},\ldots,[1]_{\vec p_{l^\prime}},\tau,\ldots,\tau,[\on{pt}]_{-\vec r}\psi^{l^\prime-2}\rangle_{0,l+l^\prime+1,\beta}^{(X,D)}. 
    \]
\end{prop}
\begin{proof}
Similarly to the proof of \cite{TY20c}*{Theorem 38}, we can apply the TRR for orbifold invariants of the pair $(X,D)$ \cite{TY20c}*{Proposition 26} to remove the descendant classes in $N_{\tau,\vec p_1,\ldots,\vec p_{l^\prime},-\vec r}$ and show that it coincides with the product $\vartheta_{\tau,\vec p_1}\star\cdots \star \vartheta_{\tau,\vec p_{l^\prime}}$. 

By TRR, we have 
\begin{align*}
 &   N^\beta_{\tau,\vec p_1,\ldots,\vec p_{l^\prime},-\vec r}\\
=& \sum \frac{1}{l_1!l_2!}\langle [1]_{\vec p_1},[1]_{\vec p_2},\prod\tau, \prod_{j \in S_1}[1]_{p_j}, \tilde{T}_{-\vec s,\alpha}\rangle \langle \tilde{T}_{\vec s}^\alpha,\prod\tau,\prod_{j\in S_2}[1]_{\vec p_j},[\on{pt}]_{-\vec r}\psi^{l^\prime-3}\rangle,
    \end{align*}
    where the number of markings with the insertion $\tau$ are $l_1$ and $l_2$ respectively; the sum is over all splittings of $l_1+l_2=l$, $\beta_1+\beta_2=\beta$, all indices $\vec s, \alpha$ of the basis of the state space $\mathfrak H$ for $(X,D)$, and all splittings of disjoint sets $S_1, S_2$ with $S_1\cup S_2=\{3,\ldots, l^\prime\}$. 

    When $-K_X-D$ is nef, we consider the invariant 
    \[
    \langle [1]_{\vec p_1},[1]_{\vec p_2},\prod\tau, \prod_{j \in S_1}[1]_{p_j}, \tilde{T}_{-\vec s,\alpha}\rangle.
    \] 
    The virtual dimension constraint gives
    \begin{align*}
&    \dim_{\mathbb C}X+|S_1|+(-K_X-D)\cdot\beta_1+l_1=\deg(\tilde{T}_{\vec s, \alpha})+\#\{i:-s_i<0\}+\frac{l_1}{2}\deg(\tau)\\
\leq & \dim_{\mathbb C}X+l_1.
    \end{align*}
    The nefness of $-K_X-D$ implies that we must have
    \[
    |S_1|=0, \quad (-K_X-D)\cdot\beta_1=0, \quad \deg(\tilde{T}_{\vec s, \alpha})+\#\{i:-s_i<0\}= \dim_{\mathbb C}X.
    \]
    This implies that
    \[
    \tilde{T}_{-\vec s,\alpha}=[\on{pt}]_{-\vec s_1}, \text{ for some } \vec s_1\in B(\mathbb Z).
    \]
    Therefore, 
    \begin{align*}
  &  N_{\tau,\vec p_1,\ldots,\vec p_{l^\prime},-\vec r}\\
= & \sum\frac{1}{l_1!l_2!}\langle [1]_{\vec p_1},[1]_{\vec p_2},\prod \tau, \prod_{j \in S_1}[1]_{p_j}, \tilde{T}_{-\vec s,\alpha}\rangle \langle \tilde{T}_{\vec s}^\alpha,\prod\tau,\prod_{j\in S_2}[1]_{\vec p_j},[\on{pt}]_{-\vec r}\psi^{l^\prime-3}\rangle\\
    = & \sum\frac{1}{l_1!l_2!}\langle [1]_{\vec p_1},[1]_{\vec p_2}, \prod\tau, [\on{pt}]_{-\vec s_1}\rangle \langle [1]_{\vec s_1},[1]_{\vec p_3},\ldots, [1]_{\vec p_{l^\prime}},\prod\tau, [\on{pt}]_{-\vec r}\psi^{l^\prime-3}\rangle.
    \end{align*}
    We repeat this process $(l^\prime-3)$-times to get the following
    \begin{align*}
    N_{\tau,\vec p_1,\ldots,\vec p_{l^\prime},-\vec r}=\sum\frac{1}{\prod_{i=1}^{l^\prime-1} k_i}&\langle [1]_{\vec p_1},[1]_{\vec p_2}, \prod\tau, [\on{pt}]_{-\vec s_1}\rangle \langle [1]_{\vec s_1},[1]_{\vec p_3}, \prod\tau, [\on{pt}]_{-\vec s_2}\rangle\cdot \cdots \\
    &\cdot\langle [1]_{\vec s_{l^\prime-2}},[1]_{\vec p_{l^\prime}}, \prod\tau, [\on{pt}]_{-\vec r}\rangle
    \end{align*}
    By the product rule of the theta functions, this is precisely the coefficient of $\vartheta_{\tau,\vec r}$ for $\vartheta_{\tau,\vec p_1}\star \vartheta_{\tau,\vec p_2}\star \cdots \star \vartheta_{\tau,\vec p_{l^\prime}}$.
    The argument when $-K_X-D$ is anti-nef is similar.
\end{proof}

Together with Definition \ref{def-theta-tau}, we can consider the $x^{-\vec k}$-coefficient of the product of theta functions. 

\begin{prop}\label{prop-coeff-k}
When $-K_X-D$ is nef or anti-nef, $\deg^0(\tau)\leq 1$, we have
    \begin{equation}\label{iden-coeff-k}
        \begin{split}
               & \sum_{\substack{\beta_1+\beta_2=\beta,\\ \vec r\in B(\mathbb Z),l_1,l_2\geq 0,l_1+l_2=l}}\frac{1}{l_1!l_2!} \langle [1]_{\vec p_1},\ldots,[1]_{\vec p_{l^\prime}},\tau,\ldots,\tau,[\on{pt}]_{-\vec r}\psi^{l^\prime-2}\rangle_{0,l^\prime+l_1+1,\beta_1}^{(X,D)}\\
& \quad \quad \quad \quad \cdot\langle [1]_{\vec r},\tau,\ldots,\tau,[1]_{\vec {\mathbf b}},[\on{pt}]_{-\vec{\mathbf b}+\vec k}\rangle_{0,3+l_2,\beta_2}^{(X,D)}\\
                = & \sum_{k\geq 0}\frac{1}{k!}\langle [1]_{\vec p_1},\ldots,[1]_{\vec p_{l^\prime}},\tau,\ldots,\tau,[1]_{\vec {\mathbf b}},[\on{pt}]_{-\vec{\mathbf b}+\vec k}\psi^{l^\prime-1}\rangle_{0,l^\prime+l+2,\beta}^{(X,D)}.
        \end{split}
    \end{equation}   
\end{prop}
\begin{proof}
Recall that, we must have $\deg^0(\tau)=1$. We apply TRR to the RHS of (\ref{iden-coeff-k}):
\begin{align*}
&\frac{1}{k!}\langle [1]_{\vec p_1},\ldots,[1]_{\vec p_{l^\prime}},\tau,\ldots,\tau,[1]_{\vec {\mathbf b}},[\on{pt}]_{-\vec{\mathbf b}+\vec k}\psi^{l^\prime-1}\rangle\\
=&\sum\frac{1}{l_1!l_2!}\langle [1]_{\vec p_1},[1]_{\vec p_2},\prod\tau, \prod_{j \in S_1}[1]_{p_j}, \tilde{T}_{-\vec s,\alpha}\rangle \langle \tilde{T}_{\vec s}^\alpha,\prod\tau,\prod_{j\in S_2}[1]_{\vec p_j},[\on{pt}]_{-\vec {\mathbf b}+\vec k}\psi^{l^\prime-2}\rangle,
\end{align*}
  where the sum is over all splittings of $l_1+l_2=l$, $\beta_1+\beta_2=\beta$, all indices $\vec s, \alpha$ of basis of the state space $\mathfrak H$ for $(X,D)$, and all splittings of disjoint sets $S_1, S_2$ with $S_1\cup S_2=\{3,\ldots, l^\prime,l^\prime+1\}$ and $\vec p_{l^\prime+1}=\vec {\mathbf b}$ is a mid-age. 

We claim that $l^\prime+1\not\in S_1$, that is, the mid-age marking with insertion $[1]_{\vec {\mathbf b}}$ cannot be distributed to the first factor. Suppose $l^\prime+1\in S_1$, then the last marking of the invariant $\langle [1]_{\vec p_1},[1]_{\vec p_2},\prod\tau, \prod_{j \in S_1}[1]_{p_j}, \tilde{T}_{-\vec s,\alpha}\rangle$ must also be a mid-age marking of the form:
\[
\tilde{T}_{-\vec s,\alpha}=[\gamma]_{-\vec {\mathbf b}+\vec a},
\]
for some $\vec a\in \mathbb Z^m$ and $\gamma \in H^*(D_{-\vec {\mathbf b}})$. Then it is straightforward to check that the virtual dimension constraint will not be satisfied for the invariant $\langle [1]_{\vec p_1},[1]_{\vec p_2},\prod\tau, \prod_{j \in S_1}[1]_{p_j}, \tilde{T}_{-\vec s,\alpha}\rangle$. Therefore, we must have $l^\prime+1\in S_2$.
  
  The rest of the proof is similar to the proof of Proposition \ref{prop-prod-theta}, we must have $|S_1|=0$ and
  \[
    \tilde{T}_{-\vec s,\alpha}=[\on{pt}]_{-\vec s_1}, \text{ for some } \vec s_1\in B(\mathbb Z).
    \]  
    Apply TRR $(l^\prime-1)$-times, and then we conclude the proof.
\end{proof}

In Definition \ref{def-theta-tau}, we assume that $\deg^0(\tau)\leq 1$. It is also possible to consider a more general $\tau\in \mathfrak H$. Then the product rule (\ref{theta-func-prod}) will involve functions other than theta functions. Recall that the theta function $\vartheta_{\vec p}$ for $\vec p\in B(\mathbb Z)$ corresponds to the identity class $[1]_{\vec p}\in H^*(D_{\vec p})\subset \mathfrak H$. In the next section, we will consider functions associated with other classes $\varphi\in \mathfrak H$.

\section{More general functions on the mirror}\label{sec:func-mirror}

In this section, we would like to generalize the definition of theta functions in \cite{You24} and in Section \ref{sec:theta-func}. We will define the function $\check{\varphi}$ for a class $\varphi \in \mathfrak H$. As we are interested in functions on the mirror, we assume that $D$ has zero-dimensional strata. More precisely, we consider a pair $(X,D)$ such that $|B|$ is pure-dimensional with $\dim_{\mathbb R} B=\dim_{\mathbb C} X$.

\begin{definition}\label{def-hat-varphi}
For an snc log Calabi--Yau pair $(X,D)$, fix $[\varphi]_{\vec s}\in \mathfrak H_{\vec s}$. Let $\sigma_{\on{max}}\in \Sigma(X)$ be a maximal cone of $\Sigma(X)$ such that $p\in \sigma_{\on{max}}$, then
    \begin{align}\label{iden-def-hat-varphi-mid-age}
   \check{\varphi}_{\tau,\vec s}(p)&:=
  \sum_{\beta} \sum_{\vec k\in {\mathbb Z}^m, l\geq 0}  \frac{1}{l!}\langle [\varphi]_{\vec s},\tau,\cdots,\tau,[1]_{\vec{\mathbf b}}, [\on{pt}]_{-\vec{\mathbf{b}}+\vec k}\rangle_{0,3+l,\beta}^{(X,D)} t^{\beta}x^{-\vec k},
    \end{align}
where 
\[
\tau=\tau_{0,2}+\tau^\prime\in \mathfrak H \text{ and } \tau_{0,2}\in H^2(X)\subset \mathfrak H_{\vec 0}=H^*(X);
\]
\[
\vec d:=(D_1\cdot\beta,\cdots,D_m\cdot\beta);
\]
the invariants are orbifold invariants with mid-ages along the divisors $D_{\mathbf b_i}$ such that $D_{\vec {\mathbf b}}=D_{\sigma_{\on{max}}}$ is a lowest dimensional stratum in $D_{\vec k}$; $x^{-\vec k}:=x_1^{-k_1}\cdots x_m^{-k_m}$.
\end{definition}

Note that, for general $\tau$, we need to choose $\vec{\mathbf b}$ to be large compared to $\beta$ and $l$. Therefore, $\vec{\mathbf b}$ in (\ref{iden-def-hat-varphi-mid-age}) can be different for different invariants. However, for each invariant, we can always find such $\vec{\mathbf b}$ and the value of the mid-age invariant does not depend on $\vec{\mathbf b}$. If we only have a finite collection of invariants, we can also find one single $\vec{\mathbf b}$ for all invariants.

Applying the divisor equation, we can rewrite
 \begin{align*}
   \check{\varphi}_{\tau,\vec s}(p)&:=
   \sum_{\beta}\sum_{\vec k\in {\mathbb Z}^m, l\geq 0}  \frac{1}{l!}\langle [\varphi]_{\vec s},\tau^\prime,\cdots,\tau^\prime,[1]_{\vec{\mathbf b}}, [\on{pt}]_{-\vec{\mathbf{b}}+\vec k}\rangle_{0,3+l,\beta}^{(X,D)} t^{\beta}x^{-\vec k}e^{ \tau_{0,2}\cdot\beta},
    \end{align*}
  It is straightforward to see that
  \[
  \check{[1]}_{\tau,\vec p}=\vartheta_{\tau,\vec p},
  \]
for $\vec p\in B(\mathbb Z)$ and $\deg^0(\tau)\leq 1$, where $[1]_{\vec p}$ is the identity class in $H^*(D_{\vec p})$.

The well-definedness of the function $\check{\varphi}_{\tau,\vec s}(p)$ follows from a similar proof for the well-definedness of theta functions in \cite{You24}*{Theorem 1.5}.

We can prove that these functions satisfy the product rule with the parameter $\tau\in \mathfrak H$.
\begin{prop}\label{prop-phi-prod}
 The functions $ \check{\varphi}_{\tau,\vec s}$ in Definition (\ref{def-hat-varphi}) satisfy the following product rule
\[
(\check{\varphi_{1}})_{\tau,\vec s_1}\star(\check{\varphi_{2}})_{\tau,\vec s_2}=\sum_{\vec s_3\in \mathbb Z^m,l\geq 0,\beta,\varphi_3} \frac{1}{l!}\langle [\varphi_1]_{\vec s_1}, [\varphi_2]_{\vec s_2},\tau,\cdots,\tau,[\varphi_3^\vee]_{-\vec s_3}\rangle_{0,l+3,\beta}^{(X,D)} t^{\beta}(\check{\varphi_3})_{\tau,\vec s_3}.
\]

\end{prop}

\begin{proof}
    We can compare the $x^{\vec k}$-coefficient of the LHS and the RHS for $\vec k\in \mathbb Z^m$.  The comparison follows from the WDVV equation with the first four insertions
\[
[\varphi_1]_{\vec s_1}, \quad [\varphi_2]_{\vec s_2}, \quad [1]_{\vec{\mathbf b}}, \text{ and}\quad [\on{pt}]_{-\vec{\mathbf b}+\vec k}
\]
and a suitable number of markings with the insertion $\tau$.
\end{proof}

Recall that when $\vec s=\vec 0$, $[\varphi]_{\vec s}\in \mathfrak H_{\vec 0}=H^*(X)$. When $\vec s=\vec 0$, we simply write
\[
\check{\varphi}_{\tau}=\check{\varphi}_{\tau,\vec 0}.
\]
Sometimes, we simply write it as $\check{\varphi}$.

We can generalize Definition \ref{def-hat-varphi}, to define the mirror function of a polynomial
\[
\mathbf t(z):=\sum_{i=0}^{d_{\mathbf t}} \mathbf t_i z^i \in \mathcal H_+,
\]
where $\mathbf t_i=\sum_{\alpha}\mathbf t_{i,\alpha} \phi_\alpha \in \mathfrak H$. 

\begin{definition}\label{def-check-t-z}
For an snc log Calabi--Yau pair $(X,D)$. Let $\sigma_{\on{max}}\in \Sigma(X)$ be a maximal cone of $\Sigma(X)$ such that $p\in \sigma_{\on{max}}$, then
    \begin{equation}
\begin{split}
   \check{\mathbf t}_{\tau}(z)(p):=&\sum_{i=0}^{d_{\mathbf t}}  \left(\sum_{\vec k\in {\mathbb Z}^m, l\geq 0}  \frac{1}{l!}\langle \mathbf t_i,\tau,\cdots,\tau,[1]_{\vec{\mathbf b}}, [\on{pt}]_{-\vec{\mathbf{b}}+\vec k}\rangle_{0,3+l,0}^{(X,D)}\right)z^i\\  
&+\sum_{\beta\neq 0} \sum_{\vec k\in {\mathbb Z}^m, l\geq 0}  \frac{1}{l!}\langle \mathbf t(\bar{\psi}),\tau,\cdots,\tau,[1]_{\vec{\mathbf b}}, [\on{pt}]_{-\vec{\mathbf{b}}+\vec k}\rangle_{0,3+l,\beta}^{(X,D)} t^{\beta}x^{-\vec k},
   \end{split} 
\end{equation}
where the notation here follows from the notation in Definition \ref{def-hat-varphi}.
\end{definition}

We would like to unify the expression in Definition \ref{def-check-t-z}, so that the computation we present later will be cleaner. When $\beta=0$, we may write
\[
\langle \mathbf t(\bar{\psi}),\tau,\cdots,\tau,[1]_{\vec{\mathbf b}}, [\on{pt}]_{-\vec{\mathbf{b}}+\vec k}\rangle_{0,3+l,0}^{(X,D)}:=\langle \mathbf t(z),\tau,\cdots,\tau,[1]_{\vec{\mathbf b}}, [\on{pt}]_{-\vec{\mathbf{b}}+\vec k}\rangle_{0,3+l,0}^{(X,D)}.
\]
Then, we can write
\[
 \check{\mathbf t}_{\tau}(z)(p):=\sum_{\beta} \sum_{\vec k\in {\mathbb Z}^m, l\geq 0}  \frac{1}{l!}\langle \mathbf t(\bar{\psi}),\tau,\cdots,\tau,[1]_{\vec{\mathbf b}}, [\on{pt}]_{-\vec{\mathbf{b}}+\vec k}\rangle_{0,3+l,\beta}^{(X,D)} t^{\beta}x^{-\vec k}.
\]

Recall that the mirror map for $\varphi\in \mathfrak H_{\vec 0}=H^*(X)$ is $\varphi(z)$ in (\ref{mirror-map-varphi}). We define the function $\check{\varphi}_{\tau}(z)$ to be the mirror function associated with $\varphi(z)$ as defined in Definition \ref{def-check-t-z}.

\section{Products of functions on the mirror}

In this section, we will consider the big relative quantum product of $\varphi\in \mathfrak H$ with $[1]_{\vec e_i}$ and the parameter $\tau=\tau_{0,2}+\tau^\prime$, $\tau_{0,2}\in H^2(X)$ and
\[
\tau^\prime=\sum_{i=1}^m [1]_{\vec e_i}.
\]
Then we will take the degree zero part of the big relative quantum product. We will see how Definition \ref{def-hat-varphi} fits into the context. In this section, we only compute products of the functions that we need for the mirror symmetric Gamma conjecture. In general, one can also compute products of more general functions using the product rule in Proposition \ref{prop-phi-prod}.

We first recall the relative quantum product for classes $[1]_{\vec p_1},[1]_{\vec p_2}$, where $\vec p_1,\vec p_2\in B(\mathbb Z)$:
\begin{align}\label{quan-prod-tau}
[1]_{\vec p_1}\star_{\tau}[1]_{\vec p_2}=\sum_{\beta,l\geq 0, \alpha,\vec s} \frac{1}{l!}\langle [1]_{\vec p_1}, [1]_{\vec p_2},\tau,\cdots,\tau,T_{-\vec s,\alpha}\rangle_{0,l+3,\beta}^{(X,D)}t^{\beta} T_{\vec s}^{\alpha},
\end{align}
where we sum over effective curve classes $\beta$, non-negative integer $l$, and dual bases $\{T_{-\vec s,\alpha}\}$, $\{T_{\vec s}^{\alpha}\}$ of $\mathfrak H$. 

Recall that we choose $\tau^\prime=\sum_{i=1}^m \tau_i[1]_{\vec e_i}$ and $\deg^0([1]_{\vec e_i})=0$. Therefore, by the virtual dimension constraint for the structural constants , we must have $\tau_i=0$ for all $i\in\{1,\ldots,m\}$ in (\ref{quan-prod-tau}). Hence, when $\tau_{0,2}=0$, this is the same as the small quantum product
\[
[1]_{\vec p_1}\star [1]_{\vec p_2}=\sum_{\beta,\vec r\in B(\mathbb Z)} \langle [1]_{\vec p_1}, [1]_{\vec p_2},[\on{pt}]_{-\vec r}\rangle_{0,3,\beta}^{(X,D)}t^{\beta} [1]_{\vec r}.
\]
In other words, this is the simplest version of the product rule of the theta functions that we considered earlier:
\begin{align*}
\vartheta_{\vec p_1}\star\vartheta_{\vec p_2}=\sum_{\vec r\in B(\mathbb Z),\beta}N^\beta_{\vec p_1,\vec p_2,-\vec r}t^\beta\vartheta_{\vec r}.
\end{align*}

The situation is different if we consider the relative quantum product with the class $\varphi \in \mathfrak H_{\vec s}$. We have
\begin{align}\label{quan-prod-phi}
[\varphi]_{\vec s}\star_\tau [1]_{\vec p}=\sum_{\beta,l\geq 0, \varphi_\alpha,\vec r} \frac{1}{k!}\langle [\varphi]_{\vec s}, [1]_{\vec p},\tau,\cdots,\tau,[\varphi_\alpha^\vee]_{-\vec r}\rangle_{0,l+3,\beta}^{(X,D)}t^{\beta} [\varphi_\alpha]_{\vec r}.
\end{align}

We restrict (\ref{quan-prod-phi}) to the degree zero part and write the restricted quantum product.  
We have
\begin{equation}\label{prod-phi-theta}
\begin{split}
\check{\varphi}_{\tau,\vec s}\star \vartheta_{\tau,\vec p}:=&\sum_{\vec r,\varphi_{\alpha}}\left[[\varphi]_{\vec s}\star_\tau [1]_{\vec p}\right]_{\check{\varphi}_{\alpha,\tau,\vec r}}\check{\varphi}_{\alpha,\tau,\vec r}\\
=&\sum_{\beta,\varphi_\alpha}\sum_{\vec r} \sum_{l\geq 0}\sum_{\sum l_i=l}\frac{1}{\prod l_i!}\langle [\varphi]_{\vec s}, [1]_{\vec p},\prod [1]_{\vec e_i}^{l_i}, [\varphi_\alpha^\vee]_{-\vec r}\rangle_{0,l+3,\beta}^{(X,D)}e^{\tau_{0,2}\cdot\beta }t^{\beta} \check{\varphi}_{\alpha,\tau,\vec r},
\end{split}
\end{equation}
where  $[\cdots]_{\check{\varphi}_{\alpha,\tau,\vec r}}$ means the coefficient of $\check{\varphi}_{\alpha,\tau,\vec r}$, the function $\check{\varphi}_{\tau,\vec s}$ is defined in Definition \ref{def-hat-varphi} and $\vartheta_{\tau,\vec p}$ is defined in Definition \ref{def-theta-tau}.

In general, we can consider the restricted quantum product
\[
\check{\varphi}_{\tau,\vec s}\star \vartheta_{\tau,\vec p_1}\star\cdots \star \vartheta_{\tau,\vec p_{l^\prime}}:=\sum_{\vec r,\varphi_\alpha}\left[[\varphi]_{\vec s}\star_\tau [1]_{\vec p_1}\cdots\star_\tau [1]_{\vec p_{l^\prime}}\right]_{\check{\varphi}_{\alpha,\tau,\vec r}}\check{\varphi}_{\alpha,\tau,\vec r}.
\]
We have 
\begin{prop}\label{prop-prod-phi-theta}
    For $\varphi\in \mathfrak H_{\vec s}$ and $d_\varphi=\deg^0(\varphi)$, the product of $\check{\varphi}_{\tau,\vec s}$ with $\frac{1}{l!}(W_{\tau})^l$ for $\tau=\tau^\prime+\tau_{0,2}$, $\tau^\prime=\sum_{i=1}^m[1]_{\vec e_i}$ and $\tau_{0,2}\in H^2(X)$ satisfies
    \begin{equation}\label{quan-prod-phi-l}
    \begin{split}
    &\frac{1}{l!}\check{\varphi}_{\tau,\vec s}\star \underbrace{W_\tau\star\cdots\star W_\tau}_{l-\text{times}}(p)\\
    =&\sum_{\beta} \sum_{\vec k\in \mathbb Z^m} \frac{1}{(d_{\varphi}+l)!}\langle [\varphi]_{\vec s}, \tau^\prime,\cdots,\tau^\prime, [1]_{\vec{\mathbf b}},[\on{pt}]_{-\vec {\mathbf b}+\vec k}\bar{\psi}^{l}\rangle_{0,d_{\varphi}+l+3,\beta}^{(X,D)}t^\beta e^{\tau_{0,2}\cdot\beta}x^{-\vec k},
    \end{split}
    \end{equation}
    where $W_\tau=\sum_{i=1}^m\vartheta_{\tau,[D_i]}$.
\end{prop}
\begin{proof}
    The $x^{-\vec k}$-coefficient of the product of the theta functions is given in Proposition \ref{prop-coeff-k}. Therefore, we have
    \begin{align*}
    &\frac{1}{l!}\underbrace{W_\tau\star\cdots\star W_\tau}_{l-\text{times}}(p)\\
    =&\sum_{\beta}\sum_{\vec k\in \mathbb Z^m} \sum_{l,j\geq 0} \frac{1}{l!j!}\langle \tau^\prime,\cdots,\tau^\prime,\tau_{0,2},\cdots,\tau_{0,2}, [1]_{\vec{\mathbf b}},[\on{pt}]_{-\vec {\mathbf b}+\vec k}\bar{\psi}^{l-1}\rangle_{0,l+j+2,\beta}^{(X,D)}t^\beta x^{-\vec k}\\
    =&\sum_{\beta}\sum_{\vec k\in \mathbb Z^m} \sum_{l\geq 0}\frac{1}{l!}\langle \tau^\prime,\cdots,\tau^\prime, [1]_{\vec{\mathbf b}},[\on{pt}]_{-\vec {\mathbf b}+\vec k}\bar{\psi}^{l-1}\rangle_{0,l+2,\beta}^{(X,D)}t^\beta e^{\tau_{0,2}\cdot\beta} x^{-\vec k},
    \end{align*}
    where $\tau^\prime=\sum_{i=1}^m [1]_{\vec e_i}$ comes from $W$ and $\tau_{0,2}$ comes from the degree two part of the quantum parameter $\tau$. The third line follows from the divisor equation.

    On the other hand, By the definition of $\check{\varphi}_{\tau,\vec s}$ in Definition \ref{def-hat-varphi} when $\tau^\prime=\sum_{i=1}^m [1]_{\vec e_i}$, we have
    \begin{align*}
   \check{\varphi}_{\tau,\vec s}(p)&:=
   \sum_{\vec k\in {\mathbb Z}^m, d_{\varphi}\geq 0} \sum_{\beta} \frac{1}{d_{\varphi}!}\langle [\varphi]_{\vec s},\tau^\prime,\cdots,\tau^\prime,[1]_{\vec{\mathbf b}}, [\on{pt}]_{-\vec{\mathbf{b}}+\vec k}\rangle_{0,3+d_{\varphi},\beta}^{(X,D)} e^{\tau_{0,2}\cdot\beta}t^{\beta}x^{-\vec k},
    \end{align*}
where the number of markings is $3+d_\varphi$ by the virtual dimension constraint.

Therefore, the LHS of (\ref{quan-prod-phi-l}) is
\begin{align*}
    \sum_{\beta_1,\beta_2}\sum_{\vec k_1,\vec k_2\in \mathbb Z^m}\sum_{l,d_{\varphi}\geq 0} &\frac{1}{l!d_{\varphi}!}\langle \tau^\prime,\cdots,\tau^\prime, [1]_{\vec{\mathbf b}},[\on{pt}]_{-\vec {\mathbf b}+\vec k_1}\bar{\psi}^{l-1}\rangle_{0,l+2,\beta_1}^{(X,D)}t^{\beta_1}e^{ \tau_{0,2}\cdot\beta_1} x^{-\vec k_1}\\
    &\cdot\langle [\varphi]_{\vec s},\tau^\prime,\cdots,\tau^\prime,[1]_{\vec{\mathbf b}}, [\on{pt}]_{-\vec{\mathbf{b}}+\vec k_2}\rangle_{0,3+d_{\varphi},\beta_2}^{(X,D)}  t^{\beta_2}e^{\tau_{0,2}\cdot\beta_2}x^{-\vec k_2}.
\end{align*}

Now for the RHS of (\ref{quan-prod-phi-l}), we apply the TRR to the invariants with the first three markings having insertions
\[
[\on{pt}]_{-\vec {\mathbf b}+\vec k}\bar{\psi}^{l}, \quad [\varphi]_{\vec s}, \quad [1]_{\vec{\mathbf b}},
\]
then the RHS of (\ref{quan-prod-phi-l}) becomes
\begin{align*}
 \sum_{\beta_1,\beta_2}\sum_{\vec k,\vec r\in \mathbb Z^m,\alpha}\sum_{l,d_{\varphi}\geq 0}\sum_{0\leq a \leq l+d_{\varphi}} &\binom{d_{\varphi}+l}{a}\langle [\on{pt}]_{-\vec {\mathbf b}+\vec k}\bar{\psi}^{l-1},\tau^\prime,\cdots,\tau^\prime, T_{-\vec r,\alpha}\rangle_{0,a+2,\beta_1}^{(X,D)} \\
 &\cdot\langle T_{\vec r}^\alpha,[\varphi]_{\vec s}, [1]_{\vec{\mathbf b}},\tau^\prime,\cdots,\tau^\prime\rangle_{0,3+d_{\varphi}+l-a,\beta_2}^{(X,D)}  t^\beta e^{ \tau_{0,2}\cdot\beta}x^{-\vec k}.
\end{align*}

There are only finitely many ways of splitting $\beta=\beta_1+\beta_2$ such that $\beta_1$ and $\beta_2$ are effective. As $[1]_{\vec{\mathbf b}}$ is a mid-age marking, the marking with insertion $T_{\vec r}^\alpha$ must also be a mid-age marking with mid-ages along $D_{\mathbf b_i}$. Since we assume that $D_{\vec {\mathbf b}}$ is a $0$-dimensional stratum, we must have 
\[
T_{\vec r}^\alpha=[1]_{-\vec{\mathbf b}+\vec c}=[\on{pt}]_{-\vec{\mathbf b}+\vec c},
\]
for some $\vec c\in \mathbb Z^m$. Then we have
\[
T_{-\vec r,\alpha}=[1]_{\vec {\mathbf b}-\vec c}.
\]
By the virtual dimension constraint and $\deg^0(\tau^\prime)=0$, we need to have $a=l$.

Hence, the RHS of (\ref{quan-prod-phi-l}) is
\begin{align*}
 \sum_{\beta_1,\beta_2}\sum_{\vec k,\vec c\in \mathbb Z^m}\sum_{l,d_{\varphi}\geq 0} &\frac{1}{(d_{\varphi}+l)!}\binom{d_{\varphi}+l}{d_{\varphi}}\langle [\on{pt}]_{-\vec {\mathbf b}+\vec k}\bar{\psi}^{l-1},\tau^\prime,\cdots,\tau^\prime,[1]_{\vec {\mathbf b}-\vec c}\rangle_{0,l+2,\beta_1}^{(X,D)} \\
 &\cdot\langle [\on{pt}]_{-\vec{\mathbf b}+\vec c},[\varphi]_{\vec s}, [1]_{\vec{\mathbf b}},\tau^\prime,\cdots,\tau^\prime\rangle_{0,d_{\varphi}+3,\beta_2}^{(X,D)}  t^\beta e^{ \tau_{0,2}\cdot\beta}x^{-\vec k}\\
 = \sum_{\beta_1,\beta_2}\sum_{\vec k,\vec c\in \mathbb Z^m}\sum_{l,d_{\varphi}\geq 0} &\frac{1}{(d_{\varphi})!(l)!}\langle [\on{pt}]_{-\vec {\mathbf b}+\vec k}\bar{\psi}^{l-1},\tau^\prime,\cdots,\tau^\prime,[1]_{\vec {\mathbf b}-\vec c}\rangle_{0,l+2,\beta_1}^{(X,D)} \\
 &\cdot\langle [\on{pt}]_{-\vec{\mathbf b}+\vec c},[\varphi]_{\vec s}, [1]_{\vec{\mathbf b}},\tau^\prime,\cdots,\tau^\prime\rangle_{0,d_{\varphi}+3,\beta_2}^{(X,D)}  t^\beta e^{\tau_{0,2}\cdot\beta}x^{-\vec k}\\
=\sum_{\beta_1,\beta_2}\sum_{\vec k,\vec c\in \mathbb Z^m}\sum_{l,d_{\varphi}\geq 0} &\frac{1}{(d_{\varphi})!(l)!}\langle [\on{pt}]_{-\vec {\mathbf b}+\vec k-\vec c}\bar{\psi}^{l-1},\tau^\prime,\cdots,\tau^\prime,[1]_{\vec {\mathbf b}}\rangle_{0,l+2,\beta_1}^{(X,D)}  t^{\beta_1} e^{\tau_{0,2}\cdot\beta_1}x^{-\vec k+\vec c}\\
 &\cdot\langle [\on{pt}]_{-\vec{\mathbf b}+\vec c},[\varphi]_{\vec s}, [1]_{\vec{\mathbf b}},\tau^\prime,\cdots,\tau^\prime\rangle_{0,d_{\varphi}+3,\beta_2}^{(X,D)}  t^{\beta_2} e^{\tau_{0,2}\cdot\beta_2}x^{-\vec c}.
\end{align*}

This agrees with our computation for the LHS of (\ref{quan-prod-phi-l}). Hence, we conclude the proof.
\end{proof}

Now we consider $\varphi\in \mathfrak H_{\vec 0}=H^*(X)$ and $d_\varphi:=\deg^0(\varphi)=\frac{1}{2}\deg(\varphi)$. With Proposition \ref{prop-prod-phi-theta}, we can consider the Taylor series expansion
\[
\exp\left(-\sum_{i=1}^m\vartheta_{\tau,[D_i]}(p)/z\right)=\sum_{l_i\geq 0}\frac{\prod_{i=1}^m(-1)^{l_i}\vartheta_{\tau,[D_i]}^{l_i}(p)}{\prod_{i=1}^ml_i!z^{l_i}}.
\]
and compute
\begin{align*}
    &\check{\varphi}_\tau (p)\exp\left(-\sum_{i=1}^m\vartheta_{\tau,[D_i]}(p)/z\right)\\
    =&\sum_{l_i\geq 0}\frac{\prod_{i=1}^m(-1)^{l_i}\check{\varphi}_{\tau}(p)\vartheta_{\tau,[D_i]}^{l_i}(p)}{\prod_{i=1}^ml_i!z^{l_i}}\\
    =& \sum_{l_i\geq 0,\sum l_i\geq d_{\varphi},\beta}\frac{(-1)^{-d_{\varphi}+\sum_{i=1}^m l_i}\langle [\varphi]_0,\prod[1]_{\vec e_i}^{l_i},[1]_{\vec {\mathbf b}},[\on{pt}]_{-\vec {\mathbf b}+\vec k}\bar{\psi}^{-d_{\varphi}+\sum l_i}\rangle t^\beta e^{\tau_{0,2}\cdot\beta} x^{-\vec k}}{z^{-d_{\varphi}}\prod_{i=1}^ml_i!z^{l_i}},
\end{align*}
where the third line follows from Proposition \ref{prop-prod-phi-theta}. 

We can replace $\check{\varphi}$, by $\check{\varphi}(-z)$ which encodes the mirror map for $\varphi$, we have 
\begin{align*}
    &\check{\varphi}_\tau (-z) (p)\exp\left(-\sum_{i=1}^m\vartheta_{\tau,[D_i]}(p)/z\right)\\
    =& \sum_{l_i\geq 0,\sum l_i\geq d_{\varphi},\beta}\frac{(-1)^{-d_{\varphi}+\sum_{i=1}^m l_i}\langle \varphi(-\bar{\psi}),\prod[1]_{\vec e_i}^{l_i},[1]_{\vec {\mathbf b}},[\on{pt}]_{-\vec {\mathbf b}+\vec k}\bar{\psi}^{-d_{\varphi}+\sum l_i}\rangle t^\beta e^{\tau_{0,2}\cdot\beta} x^{-\vec k}}{z^{-d_{\varphi}}\prod_{i=1}^ml_i!z^{l_i}},
\end{align*}
where, to unify the notation, we write
\begin{align*}
&\langle \varphi(\bar{\psi}),\prod[1]_{\vec e_i}^{l_i},[1]_{\vec {\mathbf b}},[\on{pt}]_{-\vec {\mathbf b}+\vec k}\bar{\psi}^{-d_{\varphi}+\sum l_i}\rangle_{0,n+3,0}^{(X,D)}\\
:=&\langle \varphi(z),\prod[1]_{\vec e_i}^{l_i},[1]_{\vec {\mathbf b}},[\on{pt}]_{-\vec {\mathbf b}+\vec k}\bar{\psi}^{-d_{\varphi}+\sum l_i}\rangle_{0,n+3,0}^{(X,D)},
\end{align*}
when $\beta=0$.

\begin{remark}\label{rmk:quantum-connection}
If we allow ${\mathbf t}_\alpha\varphi_\alpha\in H^*(X)$ to be a term in $\tau$.
Then, multiplication of $\exp\left(-\sum_{i=1}^m\vartheta_{\tau,[D_i]}/z\right)$ by $\check{\varphi}_\alpha$ is the same as taking the derivative $z\frac{\partial}{\partial {\mathbf t}_\alpha}$. If we consider the multiplication by $\check{\varphi}(-z)$, then we turn the invariants of $(X,D)$ to absolute invariants of $X$ via the relative mirror theorem and this multiplication is the same as taking derivaive $z\frac{\partial}{\partial {\mathbf t}_\alpha}$ in the absolute theory.
\end{remark}

\section{Functions from polyvector fields}\label{sec:polyvector}
Before continuing, we take a brief digression in this section to give a different view of the functions in Definition \ref{def-hat-varphi} in terms of polyvector fields on the mirror. A class $\varphi\in \mathfrak H$ can induce a polyvector field on the mirror. When $\varphi$ is of degree zero, it is the theta function. When $\varphi$ is of degree two (e.g. a divisor class in $H^2(X)$), it induces a vector field. We may consider the product (\ref{quan-prod-phi-l}) as an action of a polyvector field on functions.

We start with a special case where $\deg^0(\varphi)=1$. In this case $\check{\varphi}$ induces a vector field on the mirror. This includes some interesting cases such as $\varphi=[D_0]_{\vec 0}\in H^2(X)$ for a divisor class $[D_0]$ and $\varphi=[1]_{-\vec e_1}$. Recall that we defined theta functions with a parameter $\tau$ in Definition \ref{def-theta-tau}, where $\deg^0(\tau)\leq 1$. Let $\tau_2$ be the part of $\tau$ such that $\deg^0(\tau_2)=1$. We write
\[
\tau_2=\sum_{\alpha}{\mathbf t}_\alpha\varphi_\alpha.
\]
We can regard $\check{\varphi}_\alpha$ as a vector field
\[
\frac{\partial}{\partial {\mathbf t}_\alpha}
\]
on the mirror.

The action of a vector field $\check{\varphi}_\alpha$ on a theta function $\vartheta_{\vec r}$ is
\[
\check{\varphi}_\alpha \vartheta_{\tau,\vec r}(p)= \sum_{\vec k\in {\mathbb Z}^m} \sum_{\beta}N_{\tau,\vec r, \vec{\mathbf b}, -\vec{\mathbf{b}}+\vec k,\varphi_\alpha}^\beta t^{\beta} x^{-\vec k}
\] 
where
\[
N_{\tau,\vec r, \vec{\mathbf b}, -\vec{\mathbf{b}}+\vec k,\varphi_\alpha}^\beta=\sum_{l\geq 0}\frac{1}{l!}\langle [1]_{\vec r},[\varphi_\alpha],\tau,\ldots,\tau, [1]_{\vec{\mathbf b}}, [\on{pt}]_{-\vec{\mathbf{b}}+\vec k}\rangle_{0,l+4,\beta}^{(X,D)}.
\]

\begin{remark}
For mirror symmetry of a log Calabi--Yau manifold $X\setminus D$, the symplectic cohomology of the log Calabi--Yau manifold should be isomorphic to the ring of polyvector fields on the mirror. The relative quantum cohomology of $(X,D)$ is also expected to be isomorphic to a ring of polyvector fields on the mirror. This is related to an ongoing work of Gross--Pomerleano--Siebert. However, what we have here seems to be different. In that setting, the $H^1$-part of the cohomology should correspond to vector fields. In our setting, a class in $H^2(X)$ induces a vector field which is what we usually see in quantum cohomology. We then described how it acts on theta functions. 
\end{remark}

Now we can consider $\deg^0(\varphi)=d_{\varphi}>1$. This corresponds to a $d_{\varphi}$-polyvector field $\check{\varphi}$. Given functions $\vartheta_{\tau,\vec r_1},\ldots,\vartheta_{\tau,\vec r_{d_\varphi}}$, we can consider the action of a $d_{\varphi}$-polyvector field on these functions:
\begin{align}
\check{\varphi}(\vartheta_{\tau,\vec r_1},\ldots,\vartheta_{\tau,\vec r_{d_\varphi}})=\sum_{\vec k\in {\mathbb Z}^m} \sum_{\beta}N_{\tau,\vec r_1,\ldots,\vec r_{d_{\varphi}}, \vec{\mathbf b}, -\vec{\mathbf{b}}+\vec k,\varphi}^\beta t^{\beta} x^{-\vec k}
\end{align}
where
\[
N_{\tau,\vec r_1,\ldots,\vec r_{d_{\varphi}}, \vec{\mathbf b}, -\vec{\mathbf{b}}+\vec k,\varphi}^\beta=\sum_{l\geq 0}\frac{1}{l!}\langle [1]_{\vec r_1},\cdots,[1]_{\vec r_{d_{\varphi}}},[\varphi],\tau,\ldots,\tau, [1]_{\vec{\mathbf b}}, [\on{pt}]_{-\vec{\mathbf{b}}+\vec k}\rangle_{0,l+d_{\varphi}+3,\beta}^{(X,D)}.
\]

Therefore, for $\varphi\in \mathfrak H_{\vec s}$, the function $ \check{\varphi}_{\tau,\vec s}$ defined in Definition \ref{def-hat-varphi} can be considered as the action of a polyvector field on functions. For simplicity, we assume that $\deg^0(\tau^\prime)=0$. Then
\[
\check{\varphi}(\check{\tau}^\prime,\ldots,\check{\tau}^\prime)= \sum_{\beta}\sum_{\vec k\in {\mathbb Z}^m}  \frac{1}{d_{\varphi}!}\langle [\varphi]_{\vec s},\tau^\prime,\cdots,\tau^\prime,[1]_{\vec{\mathbf b}}, [\on{pt}]_{-\vec{\mathbf{b}}+\vec k}\rangle_{0,3+d_{\varphi},\beta}^{(X,D)} t^{\beta}x^{-\vec k}e^{ \tau_{0,2}\cdot\beta},
\]
where the quantum parameter is $\tau_{0,2}\in H^2(X)\subset \mathfrak H_{\vec 0}$.
So we have
\[
  \check{\varphi}_{\tau,\vec s}=\check{\varphi}(\check{\tau}^\prime,\ldots,\check{\tau}^\prime),
\]
for $\tau=\tau_{0,2}+\tau^\prime$.

\section{The equivariant potential}\label{sec:equiv-potential}

Following Gross--Siebert's intrinsic mirror construction \cite{GS19}, the mirror of the pair $(X,D)$ is a Landau--Ginzburg model $(\check{\mathfrak X},W)$, where $\check{\mathfrak X}$ is constructed as the Spec of the degree zero part of the relative quantum cohomology of $(X,D)$. Invariants here can be either punctured log invariants of \cite{ACGS} or orbifold invariants of \cite{TY20c}. The potential is a sum of theta functions
\[
W:=\vartheta_{[D_1]}+\cdots+\vartheta_{[D_m]},
\]
where $\vartheta_{[D_i]}:=\vartheta_{\vec e_i}$.

Recall that the mirror is
\[
\check{\mathfrak X}\subset \mathbb C^m\widehat{[\NE(X)]}=\on{Spec}(\mathbb C[x_1,\ldots,x_m])\widehat{[\NE(X)]}.
\] 
We consider the real Lefschetz thimble
\[
\Gamma_{\mathbb R}:=\check{\mathfrak X}\cap \mathbb R_{>0}^m
\]
and the compact cycle
\[
\Gamma_{c}:=\check{\mathfrak X}\cap (S^1)^m.
\]

Recall that we replace $Q^\beta$ by $t^\beta x^{-\vec d}$, we have
\begin{align}
(t/Q)^\beta=x^{D\cdot\beta}=x_1^{D_1\cdot\beta}\cdots x_m^{D_m\cdot\beta}.
\end{align}
We can consider the holomorphic volume form  $\omega$ on $\check{\mathfrak X}$ such that its restriction to each maximal cone $\sigma$ is of the form $\frac{\prod_{j\in I_\sigma}d x_j}{\prod_{j\in I_\sigma} x_j}$.

Similarly to the toric case considered in \cite{Iritani09}, we can consider the equivariant potential:
\[
W_{\tau}^{\vec\lambda}=\sum_{i=1}^m \vartheta_{\tau,[D_i]}^{\vec\lambda},
\]
where the invariants in $\vartheta_{\tau,[D_i]}$ are replaced by the corresponding equivariant twisted invariants of the toric stack bundle $P_{m,D,\vec r}$ over $X$; $\vec \lambda=(\lambda_1,\ldots,\lambda_m)$ is the vector of the equivariant parameters; $\tau=\sum_{i=1}^m [1]_{\vec e_i}+ \tau_{0,2}$ and $\tau_{0,2}\in H^2(X)$.

For a class $\varphi\in H^*(X)$, we defined the function $\check{\varphi}_\tau$ on the mirror in Definition \ref{def-hat-varphi} which is considered as the mirror of $\varphi \in H^*(X)$. Then we replace $\varphi$ by $\varphi(z)$ which is the mirror map associated with $\varphi$. Let $\check{\varphi}_\tau(z)$ be the function on the mirror associated with $\varphi(z)$. We consider the equivariant version of the theta functions and the function $\check{\varphi}_\tau(z)$ by replacing the relevant invariants with the equivariant twisted invariants of the bundle. We consider the integral
\begin{align}\label{integral-omega}
\int_{\Gamma_{\mathbb R}}\check{\varphi}^{\vec\lambda}_\tau(-z) e^{-W_{\tau}^{\vec \lambda}/z}\omega
\end{align}
and the integral 
\begin{align}\label{integral-pt}
\int_{\Gamma_{c}}\check{\varphi}^{\vec \lambda}_\tau(-z) e^{-W_{\tau}^{\vec\lambda}/z}\omega.
\end{align}

Consider the Taylor series expansion
\[
\exp\left(W_{\tau}^{\vec\lambda}\right)=\exp\left(-\sum_{i=1}^m\vartheta_{\tau,[D_i]}^{\vec\lambda}/z\right)=\sum_{l_i\geq 0}\frac{\prod_{i=1}^m(-1)^{l_i}\vartheta_{\tau,[D_i]}^{l_i,{\vec\lambda}}}{\prod_{i=1}^ml_i!z^{l_i}}.
\]
We have
\begin{align*}
    &\check{\varphi}_\tau(-z)(p)\exp\left(-\sum_{i=1}^m\vartheta_{\tau,[D_i]}^{\vec\lambda}(p)/z\right)\\
    =&\sum_{l_i\geq 0}\frac{\prod_{i=1}^m(-1)^{l_i}\check{\varphi}_{\tau}(-z)(p)\vartheta_{\tau,[D_i]}^{l_i,\vec{\lambda}}(p)}{\prod_{i=1}^ml_i!z^{l_i}}\\
    =& \sum_{l_i\geq 0,\sum l_i\geq d_\varphi,\beta}\frac{(-1)^{-d_\varphi+\sum_{i=1}^m l_i}\langle \varphi(-\bar{\psi}),\prod[1]_{\vec e_i}^{l_i},[1]_{\vec {\mathbf b}},[\on{pt}]_{-\vec {\mathbf b}+\vec k}\bar{\psi}^{-d_{\varphi}+\sum l_i}\rangle t^\beta e^{\tau_{0,2}\cdot\beta} x^{-\vec k}}{z^{-d_{\varphi}}\prod_{i=1}^ml_i!z^{l_i}}.
\end{align*}

Note that the equivariant twisted relative quantum cohomology for the bundle is well-defined. The properties of theta functions also follow in a similar way.

\section{Mirror symmetric Gamma conjecture for $\mathcal O_{\on{pt}}$}\label{sec:gamma-pt-sheaf}

In this section, we prove the mirror symmetric Gamma conjecture for $\mathcal O_{\on{pt}}$. When the relative mirror map is trivial, we simply have the equality. When the relative mirror map is non-trivial, the computation is the same and we have equality after a change of coordinates. We will explain how the relative mirror map can also be encoded in the definition of theta functions in Section \ref{sec:mirror-map}. 

Under homological mirror symmetry, $\mathcal O_{\on{pt}}$ is mirror to the compact cycle $\Gamma_c$. We have the following.

\begin{theorem}\label{thm-gamma-conj-pt}
Mirror symmetric Gamma conjecture for $\mathcal O_{\on{pt}}$ holds:

\begin{align}\label{gamma-conj-pt}
\int_{\Gamma_{c}}\check{\varphi}_\tau(-z) e^{-W_{\tau}/z}\omega=\int_X \left(z^{c_1}z^{\deg/2} \mathbb{J}_{X}(\tau_{0,2},-z)\varphi\right)\cup \hat{\Gamma}_X \on{Ch}(\mathcal O_{\on{pt}}),
\end{align}
where $\tau=\sum_{i=1}^m [1]_{\vec e_i}+ \tau_{0,2}$, $\tau_{0,2}\in H^2(X)$, and the equality holds under a possibly non-trivial mirror map. 

\end{theorem}

\begin{proof}
Recall that  for a maximal cone $\sigma$, we have
\begin{align*}
&\check{\varphi}_\tau (-z)e^{-W_{\tau}/z}\\
=& \sum_{\beta,l_i\geq 0} \sum_{\vec k\in \mathbb Z^m} (-z)^{d_\varphi}\frac{\prod_{i=1}^m(-1)^{l_i}}{\prod_{i =1}^ml_i!z^{l_i}}\langle \varphi(-\bar{\psi}), \prod_{i=1}^m[1]_{\vec e_i}^{l_i}, [1]_{\vec{\mathbf b}},[\on{pt}]_{-\vec {\mathbf b}+\vec k}\bar{\psi}^{\sum_{i=1}^m l_i-d_{\varphi}}\rangle^{(X,D)}_{0,\sum_{i=1}^m l_i+3,\beta}t^\beta x^{-\vec k}e^{\tau_{0,2}\cdot\beta}.
\end{align*}
Now, the LHS of (\ref{gamma-conj-pt}) is 
\[
\int_{|x_i|=1}\check{\varphi}_\tau(-z) e^{-W_{\tau}/z} \frac{\prod_{j\in I_\sigma}dx_j}{\prod_{j\in I_\sigma}x_j}.
\]
Computing the residue, we have
\begin{align}\label{coeff-x-0}
\sum_{\substack{\beta,l_i\geq 0 \\ \sum l_i \geq d_\varphi}} (-z)^{d_\varphi} \frac{\prod_{i=1}^m(-1)^{l_i}}{\prod_{i =1}^ml_i!z^{l_i}}\langle \varphi(-\bar{\psi}), \prod_{i=1}^m[1]_{\vec e_i}^{l_i}, [1]_{\vec{\mathbf b}},[\on{pt}]_{-\vec {\mathbf b}}\bar{\psi}^{\sum_{i=1}^m l_i-d_{\varphi}}\rangle^{(X,D)}_{0,\sum_{i=1}^m l_i+3,\beta} t^\beta e^{\tau_{0,2}\cdot\beta}x^{\vec 0}.
\end{align}
Note that we have $\vec k=\vec 0$. Therefore, we have
\[
d_i=D_i \cdot\beta=l_i\geq 0 \text{ for all } i=1,\ldots,m
\]
and
\[
-K_X\cdot\beta=\sum_{i=1}^m l_i.
\]

We consider the relative mirror theorem with the following extended data
\[
S=\{\vec e_i\}_{i=1}^m\cup \{\vec{\mathbf b}\}.
\]
In addition, for the $J$-function $J_{X,\beta}(\tau_{0,2}+\tau^\prime,z)$ that appears in the $I$-function of $(X,D)$, we take $\tau^\prime=\varphi x_{\varphi}$.  Then the residue (\ref{coeff-x-0}) is $(-z)$-times the sum of the $x_{\varphi}x_{\vec{\mathbf b}}\prod_{i=1}^m x_{\vec e_i}^{d_i}\cdot H^0(D_{\vec{\mathbf b}})$-coefficients of the relative $J$-function with
\[
\tau^\prime=\varphi(-z)x_{\varphi}+\sum_{i=1}^m [1]_{\vec e_i}x_{\vec e_i}+[1]_{\vec{\mathbf b}}x_{\vec{\mathbf b}}.
\]
Note that we need to multiply by $(-z)$ because we do not introduce a factor $\frac{1}{-z}$ for the insertion $[1]_{\vec {\mathbf b}}$ in (\ref{coeff-x-0}) while the extended $I$-function comes with a factor $\frac{1}{-z}$. By the relative mirror theorem, the corresponding coefficient in the extended relative $I$-function is
\[
1+\sum_{\beta: -K_X\cdot\beta\geq 1+d_{\varphi} }\langle \varphi, [\on{pt}]\psi^{-K_X\cdot\beta-1-d_{\varphi}}\rangle_{0,2,\beta}^Xe^{\tau_{0,2}\cdot\beta}t^\beta \frac{(-1)^{-K_X\cdot\beta-d_\varphi}}{z^{-K_X\cdot\beta-d_\varphi}}.
\]
That is precisely the RHS of (\ref{gamma-conj-pt}).

\end{proof}

We can specialize to the case where $\varphi=1$, then the RHS of (\ref{gamma-conj-pt}) is precisely the quantum period for the Fano variety $X$: 
\[
1+\sum_{\beta\in \NE(X),-K_X\cdot\beta\geq 2}\langle [\on{pt}]\psi^{-K_X\cdot\beta-2}\rangle_{0,1,\beta}^X t^\beta,
\]
after the change of variables $-z\mapsto z$.
We can also consider the classical period of $W$:  
\[
\int_{\Gamma_{c}}\frac{1}{1-W}\omega.
\]
Under a similar computation, we obtain the regularized quantum period
\[
1+\sum_{\beta\in \NE(X),-K_X\cdot\beta\geq 2}(-K_X\cdot\beta)!\langle [\on{pt}]\psi^{-K_X\cdot\beta-2}\rangle_{0,1,\beta}^X t^\beta.
\]
Therefore, we conclude that the regularized quantum period of $X$ coincides with the classical period of $W$, where $W:=\sum_{i=1}^m \vartheta_{[D_i]}$ is defined in terms of orbifold invariants of $(X,D)$.
This was proved in \cite{TY20b} with the assumption that all the $D_i$'s are nef. We remove this assumption by applying the more general version of the mirror theorem for snc pairs in Section \ref{sec:rel-mirror-deg} to compute the invariants. We conclude that 

\begin{theorem}\label{theorem-theta=x}
Let $X$ be a Fano variety and $D\subset X$ be an snc anticanonical divisor of $X$ such that $B$ is pure-dimensional with $\dim_{\mathbb R}B=\dim_{\mathbb C}X=n$. Then the regularized quantum period of $X$ coincides with the classical period of $W$.
\end{theorem}

\section{Differential equations}\label{sec:diff-equ}

We consider the integral
\begin{align}\label{integral-non-equiv}
\int_{\Gamma_{\mathbb R}}e^{-W_{\tau}/z}\omega.
\end{align}
We would like to explain that the integral (\ref{integral-non-equiv}) and the $J$-function of $X$ satisfies the same differential equation. We first explain it when the relative mirror map is trivial. When the mirror map is not trivial, the following statements hold after a change of variables.

Suppose the Landau--Ginzburg potential is just a Laurent polynomial 
\[
W=\sum_{i=1}^m \vartheta_{\vec e_i}=\sum_{i=1}^m x_i.
\]
Following Theorem \ref{theorem-theta=x}, we can consider the classical period of $W$ which is equal to the regularized quantum period of $X$. Under \cite{CCGGK}*{Definition 4.9}, this implies that the Laurent polynomial $W$ is the mirror-dual of the Fano variety. The oscillatory integral
\[
\int_{\Gamma_{c}} e^{-W_{\tau}/z}\omega
\]
 of $W$ is a solution of the GKZ system that comes from the Newton polytope $P_W$ of the Laurent polynomial $W$. The rank of the system is the normalized volume of the Newton polytope $P_W$ which is the number of maximal cones $\sigma \in \Sigma(X)$, hence the number of zero-dimensional strata of D. 

The quantum period is the $H^0(X)$-coefficient of the $J$-function of $X$. As the mirror map is trivial, the quantum period is equal to the classical period. Hence, it is a solution of the differential system.
Moreover, the components of the $J$-function of $X$ are also a solution of this differential system (cf. \cite{Iritani09}*{Lemma 4.6}). Furthermore, the components of the  $J$-function span a basis of solution to the differential system.

By a standard computation, the integral
\[
\int_{\Gamma_{\mathbb R}} e^{-W_{\tau}/z}\omega
\]
is also a solution to the same differential system.

In general, we can also consider it in terms of the quantum differential equation from the A-model side. We recall from (\ref{iden-mirror-theorem-X}) that the small $J$-function can be written in terms of invariants of $(X,D)$:
\begin{equation}\label{iden-mirror-theorem-X-1}
\begin{split}
&J_{X}(\tau_{0,2},Q, x,z)\\
=&e^{-\sum_{i=1}^m D_i\log x_i/z}\sum_{\beta\in \NE(X),[\gamma]}\langle \prod_{i=1}^{m_+}[1]_{\vec e_i}^{d_i}, \tau_{0,2},\ldots,\tau_{0,2},[\gamma]_{\sum_{i=m_++m_0+1}^m  d_i\vec e_i}\psi^{k} \rangle_{0,1+l+\sum_{i=1}^{m_+}d_i,\beta}^{(X,D)}\\
&\cdot [\gamma^\vee]\frac{z^{\sum_{i=1}^{m_+} d_i}}{l!}\prod_{j=1}^m  x_i^{-d_i}Q^\beta\frac{\prod_{a\leq 0}(D_j+az)}{\prod_{a\leq d_j}(D_j+az)}.
\end{split}
\end{equation}

Now we consider a quantum differential equation in terms of variables $Q_i$. The quantum differential operator is a polynomial differential operator of $Q_i\frac{\partial}{\partial Q_i}$ of the lowest order that annihilates the $J$-function. With expression (\ref{iden-mirror-theorem-X-1}), it is straightforward to check that a differential operator $\prod \left(Q_i\frac{\partial}{\partial Q_i}\right)^{a_i}$ has the same effect on the $J$-function and the oscillatory integral $\int_{\Gamma} e^{-W_{\tau}/z}\omega$ for a non-compact cycle $\Gamma$. Hence, the oscillatory integral is also a solution to the quantum differential equation. Similarly, we can consider the quantum differential operator for $\mathbb{J}_{X}(\tau_{0,2},-z)\varphi$, then the oscillatory integral (with the insertion of the mirror map for $\varphi$) is also a solution to the quantum differential equation.

We do not need an explicit description of the quantum differential operator here like in the toric case. Instead, the potential W is also defined in terms of Gromov--Witten invariants of $(X,D)$ in our set-up and these invariants are related to the $J$-function of $X$ via the relative mirror theorem. So, it is natural that the oscillatory integrals are also solutions to the quantum differential equations for the $J$-functions. We can also use the quantum connection in Remark \ref{rmk:quantum-connection} and the operator $\mathbb{J}_{X}(\tau_{0,2},-z)$ is a fundamental solution.

The same is true if we consider the equivariant potential $W_{\tau}^{\vec\lambda}$ and the integral $\int_{\Gamma_{\mathbb R}}\check{\varphi}^{\vec\lambda}_\tau(-z) e^{-W_{\tau}^{\vec\lambda}/z}\omega$ in Section \ref{sec:equiv-potential}. In this case, the $J$-function of $X$ is replaced by the equivariant $J$-function of the twisted Gromov--Witten theory of the bundle $P_m$ whose non-equivariant limit is the $J$-function of $X$. Note that the bundle $P_m$ does not have additional orbifold structures, so the non-equivariant limit of the twisted theory is the Gromov--Witten theory of $X$ not of $(X,D)$.

\section{Computing invariants via the mirror theorem}\label{sec:comp-inv-gamma-conj-phi}

In this section, we consider the $S$-extended $I$-function that will be used to compute the invariants: 
\begin{align}\label{coeff-J-n-i-phi}
\sum_{\beta}\sum_{l_i\geq 0,\sum l_i\geq d_\varphi}\frac{(-1)^{d_{\varphi}}\prod_{i=1}^m(-1)^{l_i}}{z^{-d_{\varphi}}\prod_{i=1}^m l_i!z^{l_i}}\langle \varphi(-\bar{\psi}),\prod[1]_{\vec e_i}^{l_i},[1]_{\vec {\mathbf b}},[\on{pt}]_{-\vec {\mathbf b}+\vec k}\bar{\psi}^{-d_{\varphi}+\sum l_i}\rangle t^{\beta}e^{\tau_{0,2}\cdot\beta}x^{-\vec k},
\end{align}
where there are $l_i$ markings with contact orders $1$ to $D_i$ and contact orders $0$ to $D_j$ for $j\neq i$, and two mid-age markings with contact orders $\vec {\mathbf b}$ and $-\vec{\mathbf b}+\vec k$. For a fixed $\beta$, the numbers $l_i$ are fixed for $i\not\in I_\sigma$. 

\begin{remark}
For Theorem \ref{thm-gamma-struc-sheaf}, we will restrict ourselves to the case where there is no class $\varphi$. For Theorem \ref{thm-gamma-struc-sheaf-phi}, we will include the class $\varphi$. Although the mirror map associated with $\varphi$ is less explicit in general, we present our computation with the class $\varphi$ as the calculation is the same, up to this mirror map. One may also choose to focus on Theorem \ref{thm-gamma-struc-sheaf} by ignoring $\varphi$ in the remaining sections.
\end{remark}

Set 
\[
d_i:=D_i\cdot\beta,
\]
then 
\[
l_i=d_i \text{ for } i\not\in I_\sigma.
\]
 In particular, $d_i\geq 0$ for $i\not\in I_\sigma$. Furthermore, we set
\[
t=1.
\]

We can rewrite (\ref{coeff-J-n-i-phi}) as
\begin{align*}
&\sum_{\beta}(-z)^{d_{\varphi}}\frac{\prod_{i\not\in I_\sigma}(-1)^{d_i}}{\prod_{i\not\in I_\sigma} d_i!z^{d_i}}\\
&\cdot\sum_{l_i\geq 0,i\in I_\sigma, \sum_{i=1}^m l_i\geq d_\varphi}\frac{\prod_{i\in I_\sigma}(-1)^{l_i}}{\prod_{i\in I_\sigma} l_i!z^{l_i}}\langle \varphi(-\bar{\psi}),\prod[1]_{\vec e_i}^{l_i},[1]_{\vec {\mathbf b}},[\on{pt}]_{-\vec {\mathbf b}+\vec k}\bar{\psi}^{-d_{\varphi}+\sum l_i}\rangle e^{\tau_{0,2}\cdot\beta}x^{-\vec k},
\end{align*}
where $l_i=d_i$ for $i\not\in I_\sigma$.

These invariants are $(-z)$ times the coefficients of $x_{\varphi}x_{\vec{\mathbf b}}\prod_{i\not\in I_\sigma}x_{\vec e_i}^{d_i}$ of the $J$-function $J_{(X,D)}(\tau, Q, -z)$ that take values in $\mathfrak H_{\vec{\mathbf b}-\vec k}$ and with 
\[
\tau^\prime=\varphi(-z)x_{\varphi}+\sum_{i=1}^m [1]_{\vec e_i}x_{\vec e_i}+[1]_{\vec{\mathbf b}}x_{\vec{\mathbf b}}.
\]
Note that we need to multiply by $(-z)$ because we do not introduce a factor $\frac{1}{-z}$ for the insertion $[1]_{\vec {\mathbf b}}$ in  (\ref{coeff-J-n-i-phi}).

Note that the invariants here are equivariant twisted invariants of $(P_m,X_\infty)$ and the $J$-function that we consider here is the equivariant twisted $J$-function of $(P_m,X_\infty)$:
\[
J_{(P_m,X_\infty)}^{\on{tw}}\cup_{i=1}^m h_i.
\]
We consider the part that takes values in the twisted sector $\mathfrak H_{\vec{\mathbf b}-\vec k,{\vec\lambda}}$ for the pair $(P_m,X_\infty)$. Therefore, the corresponding invariants in $J_{(P_m,X_\infty)}^{\on{tw}}$ take values in $H^{2n}(X)$. This gives the integral over $X$.

\begin{remark}
In (\ref{coeff-J-n-i-phi}), $\beta$ runs over $\NE(X)$ or $\NE(X)\setminus I$ for a monoid idea $I\subset \NE(X)$. The computation in this section works for both $\beta\in \NE(X)$ and $\beta\in\NE(X)\setminus I$. We do not specify it here. In the proof of the mirror symmetric Gamma conjecture in Section \ref{sec:proof-main}, we will work with either $\beta\in \NE(X)$ or $\beta\in \NE(X)\setminus I$. As we work with mid-age invariants, it is perhaps more clear to first consider $\beta\in \NE(X)\setminus I$ so that we only deal with finitely many invariants at a time and then take the limit.

  We would like to recall that it is indeed more subtle here if we consider $\beta \in \NE(X)$ because we have mid-age invariants. A more precise way to extract these invariants is that, for each $\{N_i\}$, we choose a sufficiently large $\vec{\mathbf b}$ such that we have a finite sum of mid-age invariants for $l_i\leq N_i$. Then we take $\{N_i\}\rightarrow \infty$. If we consider $\beta \in \NE(X)\setminus I$, then there are only finitely many $\beta$ and we can choose one sufficiently large $\vec{\mathbf b}$ for all $\beta \in \NE(X)\setminus I$.
\end{remark}

For the $S$-extended $I$-function, we need
\[
S=\{\vec e_i\}_{i=1}^m\cup \{\vec{\mathbf b}\},
\]
where $\vec e_i$ represents the contact order $1$ with $D_i$ and $0$ with $D_j$, for $j\neq i$.

We recall that (multi-)root stacks are complete intersections of (iterative-)$\mathbb P^1[r_i]$-bundles over $X$. The $I$-functions of the root stacks are obtained by the twisted $I$-functions of the toric stack bundles $P_{m,X_{\infty},\vec r}$. Then we take $\vec r\rightarrow \infty$. Here we are using the equivariant $I$-function of $(P_m,X_\infty)$ and taking the cup product $\cup_{i=1}^m (h_i+\lambda_i)$. So, the equivariant $I$-function is restricted to $H^*_{\vec{\lambda}}(D_\sigma)$.

We take $\vec{\mathbf b}$ to be sufficiently large and only consider the part of the $I$-function that takes value in $\mathfrak H_{\vec{\mathbf b}-\vec k,\vec{\lambda}}$ with $\vec{\mathbf b}\gg\vec k$. Note that we have 
\[
\mathfrak H_{\vec{\mathbf b}-\vec k,\vec\lambda}=\mathfrak H_{\vec{\mathbf b},\vec\lambda}=H^*_{\vec\lambda}(D_\sigma).
\]
The equivariant $I$-function takes values in $\mathfrak H_{\vec\lambda}$. Recall that we also have $h_i=D_i$ in $H^*_{\on{CR}}(X_{D,\vec r})$. In $\mathfrak H_{\vec{\mathbf b}}:=H^*(D_{\vec{\mathbf b}})=H^*(D_\sigma)\hookrightarrow H^*(X)$, set
\[
I_\sigma:=I_{\sigma_{\vec{\mathbf b}}}:=\{i| \mathbf b_i\neq 0\}\subset \{1,\ldots,m\}.
\]

Recall that the changes of variables for $x_{\vec e_i}$ are trivial in this case, so we use the same $x_{\vec e_i}$ for both the $I$-function and the $J$-function.  We denote this part of the equivariant $I$-function as $I^S_{(P_m,X_\infty)}(x,\vec\lambda,x_{\vec a},z)$:
\begin{align}\label{coeff-equiv-I}
\int_X e^{-\sum_{i=1}^m D_i^{\vec\lambda}(\sigma)\log x_i/z}\sum_{\beta}\sum_{l_i,c_0\geq 0}& J_{X,\beta}([\varphi]_0x_{\varphi}+\tau_{0,2},z)x^{-\vec d} \frac{\prod_{i=1}^m x_{\vec e_i}^{l_i}}{z^{\sum_{i=1}^m l_i}\prod_{i=1}^ml_i!}\\
\notag & \cdot\frac{x_{\vec{\mathbf b}}^{c_0}}{z^{c_0}c_0!}\left(\prod_{i=1}^{m}\frac{\prod_{a_i\leq d_i} (D_i^{\vec\lambda}(\sigma)+a_iz)}{\prod_{a_i\leq 0} (D_i^{\vec\lambda}(\sigma)+a_iz)}\right) \cup_{i\in I_\sigma}h_i^{\vec\lambda}.
\end{align}

\begin{remark}
Recall that we have $D_i=h_i$ in $H^*(X)$. In (\ref{coeff-J-n-i-phi}), we have $x^{\vec k}=x^{\vec d}x^{-\vec l}$. In the non-equivariant setting, we consider $x^{-\vec d}$ as coming from $-\sum_{i=1}^mD_i\log x_i$ whose equivariant lift is $-\sum_{i=1}^m h_i^{\vec\lambda}\log x_i$. Here $x_i^{-1}$ is $q_i$ in Corollary \ref{cor-extended-rel-I-func} and we set $t=1$. This part of the $I$-function (\ref{coeff-equiv-I}) actually takes value in the cohomology of a trivial $(\mathbb P^1)^{m-n}$-bundle over $X$ and we need to cup with $\cup_{i=1}^m h_i$ instead of $\cup_{i\in I_\sigma}h_i^{\vec\lambda}$. But it is still equal to the above formula (\ref{coeff-equiv-I}).
\end{remark}

The $x_{\varphi}x_{\vec {\mathbf b}}$-coefficient of $I_{(P_m,X_\infty)}^S(x,\vec\lambda,x_{\vec a},-z)$ is
\begin{align*}
\int_X e^{\sum_{i=1}^m D_i^{\vec\lambda}(\sigma)\log x_i/z}\sum_{\beta}\sum_{l_i\geq 0}& \mathbb{J}_{X,\beta}(\tau_{0,2},-z)\varphi x^{-\vec d} \frac{\prod_{i=1}^m x_{\vec e_i}^{l_i}}{(-z)^{\sum_{i=1}^m l_i}\prod_{i=1}^ml_i!}\\
&\cdot\frac{1}{-z}\left(\prod_{i=1}^{m}\frac{\prod_{a_i\leq d_i} (D_i^{\vec\lambda}(\sigma)-a_iz)}{\prod_{a_i\leq 0} (D_i^{\vec\lambda}(\sigma)-a_iz)}\right) \cup_{i\in I_\sigma}h_i^{\vec\lambda}.
\end{align*}

We can write it in terms of Gamma functions:
\begin{align}\label{coeff-I-Gamma}
\int_X e^{\sum_{i=1}^m D_i^{\vec\lambda}(\sigma)\log x_i/z}\sum_{\beta}\sum_{l_i\geq 0}& \mathbb{J}_{X,\beta}(\tau_{0,2},-z)\varphi x^{-\vec d} \frac{\prod_{i=1}^m x_{\vec e_i}^{l_i}}{(-z)^{\sum_{i=1}^m l_i}\prod_{i=1}^ml_i!}\\
\notag &\cdot\frac{1}{(-z)^{1-\sum_{i=1}^m d_i}}\left(\frac{\prod_{i=1}^{m}\Gamma(1-\frac{D_i^{\vec\lambda}(\sigma)}{z}+d_i)}{\prod_{i=1}^{m}\Gamma(1-\frac{D_i^{\vec\lambda}(\sigma)}{z})}\right)\cup_{i\in I_\sigma}h_i^{\vec\lambda}.
\end{align}

Moreover, we have 
\[
D_i^{\vec\lambda}(\sigma)=0 \text{ for } i\not\in I_\sigma
\]
 and 
\[
l_i=d_i:=D_i\cdot\beta, k_i=0 \text{ for } i\not\in I_\sigma.
\]
\[
k_i=d_i-l_i, \text{ for } i \in I_\sigma.
\] 
Therefore, we can write (\ref{coeff-I-Gamma}) as:

\begin{align*}
 \int_X e^{\sum_{i\in I_\sigma} D_i^{\vec\lambda}(\sigma)\log x_i/z}\sum_{\beta}\sum_{l_i\geq 0}& \mathbb{J}_{X,\beta}(\tau_{0,2},-z)\varphi x^{-\vec d} \frac{ \prod_{i\in I_\sigma} x_{\vec e_i}^{l_i}}{(-z)^{\sum_{i\in I_\sigma} l_i}\prod_{i\in I_\sigma}l_i!}\prod_{i\not\in I_\sigma}x_{\vec e_i}^{d_i}\\
&\cdot\frac{1}{(-z)^{1-\sum_{i\in I_\sigma} d_i}}\left(\frac{\prod_{i\in I_\sigma}\Gamma(1-\frac{D_i^{\vec\lambda}(\sigma)}{z}+d_i)}{\prod_{i\in I_\sigma}\Gamma(1-\frac{D_i^{\vec\lambda}(\sigma)}{z})}\right)\cup_{i\in I_\sigma}h_i^{\vec\lambda},
\end{align*}
where the factor $\prod_{i\not\in I_\sigma}(-z)^{l_i}l_i!$ is canceled.
Note that $x_i$ and $x_{\vec e_i}$ are different variables here.

On the other hand, we recall the gamma function identity:
\[
\Gamma(1-c)\Gamma(1+c)=\frac{2\pi \sqrt{-1}ce^{-\pi \sqrt{-1}c} }{1-e^{-2\pi \sqrt{-1}c }}.
\]

Therefore, we have
\begin{align*}
&\prod_{i\in I_\sigma}\frac{\Gamma(1-\frac{D_i^{\vec\lambda}(\sigma)}{z}+d_i)}{\Gamma(1-\frac{D_i^{\vec\lambda}(\sigma)}{z})}h_i^{\vec\lambda}\\
=& \prod_{i\in I_\sigma} \frac{2\pi \sqrt{-1}(D_i^{\vec\lambda}(\sigma)/z-d_i)e^{-\pi \sqrt{-1}(D_i^{\vec\lambda}(\sigma)/z-d_i)} }{(1-e^{-2\pi \sqrt{-1}(D_i^{\vec\lambda}(\sigma)/z-d_i) })\Gamma(1+D_i^{\vec\lambda}(\sigma)/z-d_i)}\\
&\cdot \frac{z\Gamma(1+D_i^{\vec\lambda}(\sigma)/z)(1-e^{-2\pi \sqrt{-1}D_i^{\vec\lambda}(\sigma)/z})}{2\pi \sqrt{-1} e^{-\pi \sqrt{-1} D_i^{\vec\lambda}(\sigma)/z}}\\
=& \prod_{i\in I_\sigma} \frac{\Gamma(1+D_i^{\vec\lambda}(\sigma)/z)}{\Gamma(1+D_i^{\vec\lambda}(\sigma)/z-d_i)} z(D_i^{\vec\lambda}(\sigma)/z-d_i)(-1)^{d_i} \\
=& (-1)^{d_i}\prod_{i\in I_\sigma} \frac{z\Gamma(1+D_i^{\vec\lambda}(\sigma)/z)}{\Gamma(D_i^{\vec\lambda}(\sigma)/z-d_i)}.
\end{align*}

In conclusion, the $x_{\varphi}x_{\vec {\mathbf b}}$-coefficient of the $I$-function is
\begin{align*}
\int_X e^{\sum_{i\in I_\sigma} D_i^{\vec\lambda}(\sigma)\log x_i/z}\sum_{\beta}\sum_{l_i\geq 0}&\mathbb{J}_{X,\beta}(\tau_{0,2},-z)\varphi x^{-\vec d}\frac{\prod_{i\in I_\sigma}x_{\vec e_i}^{l_i}}{(-z)^{\sum_{i\in I_\sigma} l_i}\prod_{i\in I_\sigma}l_i!}\prod_{i\not\in I_\sigma}x_{\vec e_i}^{d_i}\\
&\cdot\frac{1}{(-z)\cdot z^{-\sum_{i\in I_\sigma} d_i}}\prod_{i\in I_\sigma} \frac{z\Gamma(1+D_i^{\vec\lambda}(\sigma)/z)}{\Gamma(D_i^{\vec\lambda}(\sigma)/z-d_i)}.
\end{align*}

Recall that in (\ref{coeff-J-n-i-phi}) we have  
\[
\vec x^{\vec k}=\prod_{i=1}^m x_i^{k_i}=\prod_{i\in I_\sigma} x_i^{k_i}=\prod_{i\in I_\sigma}x_i^{d_i-l_i}.
\]
 The coefficient of $x_{\varphi}x_{\vec {\mathbf b}}$ of the $I$-function is precisely (\ref{coeff-J-n-i-phi}) after setting
\[
x_{\vec e_i}=x_i.
\]
Therefore, (\ref{coeff-J-n-i-phi}) is
\begin{align*}
\int_X e^{\sum_{i\in I_\sigma} D_i^{\vec\lambda}(\sigma)\log x_i/z}\sum_{\beta}\sum_{l_i\geq 0}&\mathbb{J}_{X,\beta}(\tau_{0,2},-z)\varphi \frac{\prod_{i\in I_\sigma}x_{i}^{l_i}}{(-z)^{\sum_{i\in I_\sigma} l_i}\prod_{i\in I_\sigma}l_i!}\\
&\cdot\frac{\prod_{i\in I_\sigma} x_i^{-d_i}}{ z^{-\sum_{i\in I_\sigma} d_i}}\prod_{i\in I_\sigma} \frac{z\Gamma(1+D_i^{\vec\lambda}(\sigma)/z)}{\Gamma(D_i^{\vec\lambda}(\sigma)/z-d_i)} \\
=\int_X e^{\sum_{i\in I_\sigma} D_i^{\vec\lambda}(\sigma)\log x_i/z}\sum_{\beta}\sum_{l_i\geq 0}&\mathbb{J}_{X,\beta}(\tau_{0,2},-z)\varphi \frac{\prod_{i\in I_\sigma}x_{i}^{l_i}}{(-z)^{\sum_{i\in I_\sigma} l_i}\prod_{i\in I_\sigma}l_i!}\\
&\cdot\frac{\prod_{i\in I_\sigma} x_i^{-d_i}}{ z^{-\sum_{i\in I_\sigma} d_i}}\prod_{i\in I_\sigma} \frac{z\Gamma(1+D_i^{\vec\lambda}(\sigma)/z)}{\Gamma(D_i^{\vec\lambda}(\sigma)/z-d_i)}.
\end{align*}

\section{Mirror symmetric Gamma conjecture for $\Gamma_{\mathbb R}$}\label{sec:proof-main}

In this section, we will prove Theorem \ref{thm-gamma-struc-sheaf} and Theorem \ref{thm-gamma-struc-sheaf-phi}. We do not mention whether the relative mirror map is trivial here. In Section \ref{sec:mirror-map}, we will explain how to define theta functions with corrections when the mirror map is not trivial. The computation we present in this section works when the relative mirror map is not trivial. We present the proof with the class $\varphi$. For Theorem  \ref{thm-gamma-struc-sheaf}, one can simply remove $\varphi$ everywhere in the proof.

 Recall that 
\[
D_\sigma=\cap_{i\in I_\sigma}D_i.
\]
Let $\sigma$ be a maximal dimensional cone of $\Sigma(X)$. We assume that the strata of $D$ is connected. Therefore, $D_\sigma$ is a point. 

Recall that the values of the functions depend on the chamber that contains the end point of the broken line. For a given maximal cone $\sigma$, the integral (\ref{integral-omega}) becomes
\[
\int_{(\mathbb R_{>0})^n}\check{\varphi}^{\vec\lambda}_\tau(-z) (p) e^{-W_{\tau}^{\vec\lambda}(p)/z}\frac{\prod_{j\in I_\sigma}d x_j}{\prod_{j\in I_\sigma} x_j},
\]
where $p\in\sigma$. We compute the integral for a fixed maximal cone $\sigma$, the computations for other maximal cones are the same.

Now we need to compute
\begin{align}\label{integral-sigma}
\int_{(\mathbb R_{>0})^n} \check{\varphi}^{\vec\lambda}_\tau(-z)(p) \exp\left(-\sum_{i=1}^m\vartheta_{\tau,[D_i]}^{\vec\lambda}(p)/z\right)\frac{\prod_{j\in I_\sigma}dx_j}{\prod_{j\in I_\sigma}x_j}.
\end{align}
We recall the Taylor expansion:
\begin{align}\label{W-taylor-expansion}
\exp\left(-\sum_{i=1}^m\vartheta_{\tau,[D_i]}^{\vec\lambda}(p)/z\right)=\sum_{l_i\geq 0}\frac{\prod_{i=1}^m(-1)^{l_i}\vartheta_{\tau,[D_i]}^{l_i,\vec\lambda}(p)}{\prod_{i=1}^ml_i!z^{l_i}}.
\end{align}

By Proposition \ref{prop-prod-theta}, which is a result of the product rule of the theta functions, we can write (\ref{W-taylor-expansion}) as
\[
\exp\left(-\sum_{i=1}^m\vartheta_{\tau,[D_i]}^{\vec\lambda}(p)/z\right)=\sum_{\beta, l_i\geq 0, \vec r}\frac{\prod_{i=1}^m(-1)^{l_i}N^{\beta,\vec\lambda}_{\tau,\vec e_i^{l_i},-\vec r}\vartheta_{\tau,\vec r}^{\vec\lambda}(p)}{\prod_{i =1}^ml_i!z^{l_i}},
\]
where $\vec r\in B(\mathbb Z)$ and $\beta$ is an effective curve class.

We recall the definition of theta functions in Definition \ref{def-theta-tau} in terms of orbifold invariants with mid-ages:
\begin{align*}
   \vartheta_{\tau,\vec r}^{\vec\lambda}(p)&:=
   \sum_{\vec k\in {\mathbb Z}^m} \sum_{\beta: D_i\cdot \beta=k_i+r_i}N_{\tau,\vec r, \vec{\mathbf b}, -\vec{\mathbf{b}}+\vec k}^{\beta,\vec\lambda} t^{\beta} x^{-\vec k},
    \end{align*}
where $p\in \sigma$ and $D_{\vec{\mathbf b}}=D_\sigma$. Here we take $\beta\in \NE(X)\setminus I$ in the definition of theta functions. As mentioned in Remark \ref{rmk-theta}, there are only finitely many terms in the definition of theta functions. Therefore, the Taylor expansion (\ref{W-taylor-expansion}) is also a finite sum and we take large complex numbers $\lambda_i$ such that the gamma function $\Gamma(D_i^{\vec\lambda}(\sigma)/z-d_i)$ is well-defined.

By Proposition \ref{prop-prod-phi-theta}, we have the following:

\begin{equation}\label{iden-integrand}
    \begin{split}
    &\check{\varphi}_\tau^{\vec\lambda} (-z)(p)\exp\left(-\sum_{i=1}^m\vartheta_{\tau,[D_i]}^{\vec\lambda}(p)/z\right)\\
    =&\sum_{\beta,l_i\geq 0} \sum_{\vec k\in \mathbb Z^m} (-z)^{d_{\varphi}}\frac{\prod_{i=1}^m(-1)^{l_i}}{\prod_{i =1}^ml_i!z^{l_i}}\langle \varphi(-\bar{\psi}), \prod_{i=1}^m[1]_{\vec e_i}^{l_i}, [1]_{\vec{\mathbf b}},[\on{pt}]_{-\vec {\mathbf b}+\vec k}\bar{\psi}^{\sum_{i=1}^m l_i}\rangle_{0,d_{\varphi}+\sum_{i=1}^m l_i+3,\beta}^{\on{tw},\vec\lambda}t^\beta x^{-\vec k}e^{\tau_{0,2}\cdot\beta}.
    \end{split}
    \end{equation}


For each $\beta$, the number $l_i$, for $i\not\in I_\sigma$, is fixed. In fact $l_i=d_i:=D_i\cdot\beta$ for $i\not\in I_\sigma$. Recall that we also set $t=1$. We need to sum over $l_i$ for $i\in I_\sigma$. These invariants
\begin{align*}
&\sum_{\beta}(-z)^{d_\varphi}\left(\prod_{i\not\in I_\sigma}\frac{(-1)^{l_i}}{l_i!z^{l_i}}\right)\sum_{l_i:i\in I_\sigma}\left(\prod_{i\in I_\sigma}\frac{(-1)^{l_i}}{l_i!z^{l_i}}\right)\\
&\quad \quad \cdot\langle \varphi(-\bar{\psi}), \prod_{i=1}^m[1]_{\vec e_i}^{l_i}, [1]_{\vec{\mathbf b}},[\on{pt}]_{-\vec {\mathbf b}+\vec k}\bar{\psi}^{\sum_{i=1}^m l_i}\rangle_{0,d_{\varphi}+\sum_{i=1}^m l_i+3,\beta}^{\on{tw},\vec\lambda}x^{-\vec k} e^{\tau_{0,2}\cdot\beta}
\end{align*}
 are computed via the relative mirror theorem in Section \ref{sec:comp-inv-gamma-conj-phi}.

 Now, (\ref{iden-integrand}) becomes
\begin{align*}
\int_X e^{\sum_{i\in I_\sigma} D_i^{\vec\lambda}(\sigma)\log x_i/z}\sum_{\beta} & \mathbb{J}_{X,\beta}(\tau_{0,2},-z)\varphi \frac{1}{(z)^{-\sum_{i\in I_\sigma} d_i}}\left(\sum_{l_i\geq 0:i\in I_\sigma}\prod_{i\in I_\sigma}\frac{(-1)^{l_i}x_i^{l_i}}{l_i!z^{l_i}}\right)\\
&\cdot \prod_{i\in I_\sigma} x_i^{-d_i}\prod_{i\in I_\sigma} \frac{z\Gamma(1+D_i^{\vec\lambda}(\sigma)/z)}{\Gamma(D_i^{\vec\lambda}(\sigma)/z-d_i)}\\
= \int_X e^{\sum_{i\in I_\sigma} D_i^{\vec\lambda}(\sigma)\log x_i/z}\sum_{\beta} &\mathbb{J}_{X,\beta}(\tau_{0,2},-z)\varphi \frac{1}{(z)^{-\sum_{i\in I_\sigma} d_i}}e^{\sum_{i\in I_\sigma}-x_i/z}\\
& \cdot \prod_{i\in I_\sigma} x_i^{-d_i}\prod_{i\in I_\sigma} \frac{z\Gamma(1+D_i^{\vec\lambda}(\sigma)/z)}{\Gamma(D_i^{\vec\lambda}(\sigma)/z-d_i)},
\end{align*}
where $k_i=d_i-l_i$ for $i\in I_\sigma$ and $d_i:=D_i\cdot\beta$ for $i\in\{1,\ldots,m\}$.

Extracting the factors in (\ref{integral-sigma}) that are relevant to $x_i$ together with the factor $\prod_{i\in I_\sigma} \frac{1}{\Gamma(D_i^{\vec\lambda}(\sigma)/z-d_i)}$, we consider the integral
\[
\prod_{i\in I_\sigma} \frac{1}{\Gamma(D_i^{\vec\lambda}(\sigma)/z-d_i)}\int_{(\mathbb R_{>0})^n}\prod_{i\in I_\sigma}x_i^{-d_i} \prod_{i\in I_\sigma}x_i^{D_i^{\vec\lambda}(\sigma)/z}e^{-x_i/z}\frac{\prod_{i\in I_\sigma}d x_i}{\prod_{i\in I_\sigma} x_i}.
\]
\if{With a change of variable $x_i\mapsto x_i^{-1}$, we have
\begin{align*}
    &\prod_{i\in I_\sigma} \frac{1}{\Gamma(D_i^{\vec\lambda}(\sigma)/z-d_i)}\int_{(\mathbb R_{>0})^n}\prod_{i\in I_\sigma}x_i^{d_i} \prod_{i\in I_\sigma}x_i^{-D_i^{\vec\lambda}(\sigma)/z}e^{x_i^{-1}/z}\frac{\prod_{i\in I_\sigma}d x_i}{\prod_{i\in I_\sigma} x_i}\\
    =&\prod_{i\in I_\sigma} \frac{1}{\Gamma(D_i^{\vec\lambda}(\sigma)/z-d_i)}\int_{(\mathbb R_{>0})^n}\prod_{i\in I_\sigma}x_i^{-d_i} \prod_{i\in I_\sigma}x_i^{D_i^{\vec\lambda}(\sigma)/z}e^{-x_i/z}\frac{\prod_{i\in I_\sigma}d x_i}{\prod_{i\in I_\sigma} x_i}.\\
\end{align*}
}\fi

Considering a change of variables $x_i\mapsto zx_i$, we have
\begin{align*}
&\prod_{i\in I_\sigma} \frac{1}{\Gamma(D_i^{\vec\lambda}(\sigma)/z-d_i)}\int_{(\mathbb R_{>0})^n}\prod_{i\in I_\sigma}x_i^{-d_i} \prod_{i\in I_\sigma}x_i^{D_i^{\vec\lambda}(\sigma)/z}e^{-x_i/z}\frac{\prod_{i\in I_\sigma}d x_i}{\prod_{i\in I_\sigma} x_i}\\
=&\prod_{i\in I_\sigma} \frac{1}{\Gamma(D_i^{\vec\lambda}(\sigma)/z-d_i)}\int_{(\mathbb R_{>0})^n}\prod_{i\in I_\sigma}(zx_i)^{-d_i} \prod_{i\in I_\sigma}(zx_i)^{D_i^{\vec\lambda}(\sigma)/z}e^{-x_i}\frac{\prod_{i\in I_\sigma}d x_i}{\prod_{i\in I_\sigma} x_i}\\
   =&\prod_{i\in I_\sigma} \frac{1}{\Gamma(D_i^{\vec\lambda}(\sigma)/z-d_i)}\prod_{i\in I_\sigma}z^{-d_i}z^{D_i^{\vec\lambda}(\sigma)/z}\Gamma(D_i^{\vec\lambda}(\sigma)/z-d_i)\\
   =&\prod_{i\in I_\sigma}z^{-d_i}z^{D_i^{\vec\lambda}(\sigma)/z}.
\end{align*}

Therefore,
\begin{align*}
    &\int_{(\mathbb R_{>0})^n} \check{\varphi}_\tau^{\vec\lambda}(-z) (p)\exp\left(-\sum_{i=1}^m\vartheta_{\tau,[D_i]}^{\vec\lambda}(p)/z\right)\frac{\prod_{i\in I_\sigma}dx_i}{\prod_{i\in I_\sigma}x_i}\\
    =&\int_X \sum_{\beta}z^{-\sum_{i\in I_\sigma}d_i}z^{c_1/z}\frac{\mathbb{J}_{X,\beta}(\tau_{0,2},-z)\varphi}{(z)^{-\sum_{i\in I_\sigma}d_i}}\prod_{i\in I_\sigma} z\Gamma(1+D_i^{\vec\lambda}(\sigma)/z)\\
    =&  \int_X \sum_{\beta} \left(z^{c_1}z^{\deg/2}\mathbb{J}_{X,\beta}(\tau_{0,2},-z)\varphi\right)\cup \prod_{i\in I_\sigma}\Gamma(1+D_i^{\vec\lambda}(\sigma)).
\end{align*}
where $c_1:=c_1(TX)$ and $p\in\sigma$. Note that it takes values in $H^*_{\vec\lambda}(D_\sigma)$. The third line is straightforward by checking the power of $z$.  As we mentioned earlier, the same computation works for each maximal cone $\sigma$.

The oscillatory integral $\int_{\Gamma_{\mathbb R}}\check{\varphi}^{\vec\lambda}_\tau(-z) e^{-W^{\vec\lambda}_{\tau}/z}\omega$ is a solution to a differential equation and the equivariant $\mathbb J$-function is a basis of the solution for the differential equation. The oscillatory integral $\int_{\Gamma_{\mathbb R}}\check{\varphi}^{\vec\lambda}_\tau(-z) e^{-W^{\vec\lambda}_{\tau}/z}\omega$ is a linear combination of the components of the $\mathbb J$-function. We must have
\[
\int_{\Gamma_{\mathbb R}}\check{\varphi}_\tau(-z) e^{-W_{\tau}/z}\omega=\int_X \left(z^{c_1}z^{\deg/2} \mathbb{J}_{X}(\tau_{0,2},-z)\varphi\right)\cup \hat{\Gamma}_X,
\]
where $\beta\in \NE(X)\setminus I$. Taking the limit over $I$, we prove Theorem \ref{thm-gamma-struc-sheaf}.

\section{Theta functions with quantum corrections}\label{sec:mirror-map}

Now we consider Theorem \ref{thm-gamma-struc-sheaf} and Theorem \ref{thm-gamma-pt} when the relative mirror map is not trivial. As we mentioned earlier, the proofs in Section \ref{sec:gamma-pt-sheaf}-\ref{sec:proof-main} still work and the identities hold after including the mirror maps. In this section, we would like to adjust the setting slightly by including mirror maps in the definition of the theta functions. In other words, instead of considering $\vartheta_{\vec e_i}$, we consider theta functions that encode the relative mirror map. Recall that, in Definition \ref{def-theta-tau}, we defined
 \begin{align*}
   \vartheta_{\tau,\vec e_i}(p)&:=
   \sum_{\vec k\in {\mathbb Z}^n} \sum_{\beta: D_i\cdot \beta=k_i+e_i}N_{\tau,\vec e_i, \vec{\mathbf b}, -\vec{\mathbf{b}}+\vec k}^\beta t^{\beta} x^{\vec k},
    \end{align*}
with
\[
N_{\tau,\vec e_i, \vec{\mathbf b}, -\vec{\mathbf{b}}+\vec k}^\beta=\sum_{l\geq 0}\frac{1}{l!}\langle [1]_{\vec r},\tau,\ldots,\tau, [1]_{\vec{\mathbf b}}, [\on{pt}]_{-\vec{\mathbf{b}}+\vec k}\rangle_{0,l+3,\beta}^{(X,D)},
\]
where $\tau \in \mathfrak H$ and we assume that $\deg^0(\tau)\leq 1$.
Now, let 
\[
\tau=\tau(Q,q)
\]
 be the non-extended relative mirror map (\ref{mirror-map}). Note that $$\deg^0(\tau(Q,q))=1$$. Then we define the superpotential as
\[
W_\tau=\sum_{i=1}^{m} \vartheta_{\tau,\vec e_i}.
\]

To define the theta function for $\vartheta_{\tau,\vec r}$ for $\vec r\in B(\mathbb Z)$ we need an additional mirror map associated with $[1]_{\vec r}$. As we computed in (\ref{extended-mirror-map}), the extended mirror map associated with $[1]_{\vec e_i}$ is trivial. One can compute that the extended mirror map for $[1]_{\vec r}$ is
\begin{align}\label{mirror-map-r}
\tilde{[1]}_{\vec r}x_{\vec r}:= \left(\sum_{\substack{\beta:d_i= D_i\cdot\beta\leq r_i, \\ k\geq 2, \vec r-\vec d\in B(\mathbb Z)}} \langle [\on{pt}]_{D_{\beta,-}}\psi^{k-2}\rangle_{0,1,\beta}^X \frac{\prod_{i: d_i> 0} d_i!}{\prod_{j: d_j<0}(-d_j-1)!}Q^\beta q^{\vec d}[1]_{\vec r-\vec d}\right)x_{\vec r}.
\end{align}
Then we can define the theta functions that encode these mirror maps, we may call them the theta functions with quantum corrections. If we increase the contact orders $r_i$ to be sufficiently large, then the mirror map (\ref{mirror-map-r}) becomes the mirror map for the corresponding mid-age marking.

\begin{definition}\label{def-theta-tau-mirror}
For a Fano variety $X$ with an snc anticanonical divisor $D$, let $\tau\in \mathfrak H$ with $\deg^0(\tau)\leq 1$ be the non-extended mirror map (\ref{mirror-map}). Theta functions with quantum corrections are defined as follows: Fix $\vec r \in B(\mathbb Z)\setminus \{0\}$. Let $\sigma_{\on{max}}\in \Sigma(X)$ be a maximal cone of $\Sigma(X)$ such that $p\in \sigma_{\on{max}}$, then
    \begin{align}\label{iden-theta-tau}
   \tilde{\vartheta}_{\tau,\vec r}(p)&:=
   \sum_{\vec k\in {\mathbb Z}^m} \sum_{\beta: D_i\cdot \beta=k_i+r_i}N_{\tau,\vec r, \vec{\mathbf b}, -\vec{\mathbf{b}}+\vec k}^\beta t^{\beta} x^{-\vec k},
    \end{align}
with
\[
N_{\tau,\vec r, \vec{\mathbf b}, -\vec{\mathbf{b}}+\vec k}^\beta=\sum_{l\geq 0}\frac{1}{l!}\langle \tilde{[1]}_{\vec r},\tau,\ldots,\tau, [1]_{\vec{\mathbf b}}, [\on{pt}]_{-\vec{\mathbf{b}}+\vec k}\rangle_{0,l+3,\beta}^{(X,D)},
\]
where the last two markings are mid-age markings along $D_i$ for $i\in I_\sigma$.  
\end{definition}

The theta functions with quantum corrections satisfy the following simple relation.

\begin{proposition}\label{prod-theta-tau}
The theta functions with quantum corrections $\tilde{\vartheta}_{\tau,\vec r}$ satisfy
\[
\tilde{\vartheta}_{\tau,\vec r}=\prod_{i\in I_{\vec r}}\tilde{\vartheta}_{\tau,\vec e_i}^{r_i}.
\]  
\end{proposition}

\begin{proof}
Let $\tau_{\vec{\mathbf b}}$ be the extended mirror map for $[1]_{\vec{\mathbf b}}$. We consider the $S$-extended $I$-function with the extended data
\[
S=\{\vec r, \vec{\mathbf b}\}.
\]
Then, we sum over the $\frac{x_{\vec r}x_{\vec{\mathbf b}}}{z}$-coefficient of the $J$-function that takes value in $[1]_{-\vec {\mathbf b}+\vec k}$, we have
\[
\tau_{\vec{\mathbf b}} \sum_{\vec k\in {\mathbb Z}^m} \sum_{\beta: D_i\cdot \beta=k_i+r_i}\sum_{l\geq 0}\frac{1}{l!}\langle \tilde{[1]}_{\vec r},\tau,\ldots,\tau, [1]_{\vec{\mathbf b}}, [\on{pt}]_{-\vec{\mathbf{b}}+\vec k}\rangle_{0,l+3,\beta}^{(X,D)}Q^\beta q^{\vec d}.
\]
This is precisely,
\[
\tau_{\vec{\mathbf b}} \tilde{\vartheta}_{\tau,\vec r}(p).
\]
On the other hand, the corresponding coefficient of the $S$-extended $I$-function is
\[
\sum_{\beta} \langle [\on{pt}]_{D_{\beta,-}}\bar{\psi}^{k-2} \rangle_{0,1,\beta}^{X}\frac{\prod_{j:d_j>0}d_j!}{\prod_{j:d_j<0}(-d_j-1)!}Q^\beta q^{D\cdot\beta}.
\]

For $\prod_{i\in I_{\vec r}}\tilde{\vartheta}_{\tau,\vec e_i}^{r_i}$, the mirror maps associated with $[1]_{\vec e_i}$ are trivial.  The product rule of theta functions with parameter $\tau$ (Proposition \ref{prop-prod-theta}) implies that: 
\[
\prod_{i\in I_{\vec r}}\tilde{\vartheta}_{\tau,\vec e_i}^{r_i}(p)=  \sum_{\vec k\in {\mathbb Z}^m} \sum_{\beta: D_i\cdot \beta=k_i+r_i}\sum_{l\geq 0}\frac{1}{l!}\langle \prod_{i\in I_{\vec r}}[1]_{\vec e_i}^{r_i},\tau,\ldots,\tau, [1]_{\vec{\mathbf b}}, [\on{pt}]_{-\vec{\mathbf{b}}+\vec k}\rangle_{0,l+3+\sum r_i,\beta}^{(X,D)}Q^\beta q^{\vec d}.
\]
To compute $\prod_{i\in I_{\vec r}}\tilde{\vartheta}_{\tau,\vec e_i}^{r_i}$, we need to consider the $S$-extended $I$-function with the extended data
\[
S=\{\vec e_i\}_{i\in I_{\vec r}}\cup\{\vec{\mathbf b}\}.
\]
The summation of the $\frac{x_{\vec{\mathbf b}}\prod_{i\in I_{\vec e_i}}x_{\vec e_i}^{r_i}}{z^{\sum_{i\in I_{\vec r}} r_i}}$-coefficient of the $J$-function that takes value in  $[1]_{-\vec {\mathbf b}+\vec k}$ is
\[
\frac{1}{\prod_{i\in I_{\vec r}}r_i!} \tau_{\vec{\mathbf b}} \sum_{\vec k\in {\mathbb Z}^m} \sum_{\beta: D_i\cdot \beta=k_i+r_i}\sum_{l\geq 0}\frac{1}{l!}\langle \prod_{i\in I_{\vec r}}[1]_{\vec e_i}^{r_i},\tau,\ldots,\tau, [1]_{\vec{\mathbf b}}, [\on{pt}]_{-\vec{\mathbf{b}}+\vec k}\rangle_{0,l+2+\sum r_i,\beta}^{(X,D)}Q^\beta q^{\vec d}.
\]
The corresponding coefficient of the $S$-extended $I$-function is
\[
\frac{1}{\prod_{i\in I_{\vec r}}r_i!}\sum_{\beta} \langle [\on{pt}]_{D_{\beta,-}}\bar{\psi}^{k-2} \rangle_{0,1,\beta}^{X}\frac{\prod_{j:d_j>0}d_j!}{\prod_{j:d_j<0}(-d_j-1)!}Q^\beta q^{D\cdot\beta}.
\]
After canceling the factor $\frac{1}{\prod_{i\in I_{\vec r}}r_i!}$, we have 
\[
\tilde{\vartheta}_{\tau,\vec r}=\prod_{i\in I_{\vec r}}\tilde{\vartheta}_{\tau,\vec e_i}^{r_i}.
\]
\end{proof}

\begin{remark}
When we use the relative mirror theorem to compute theta functions, different $S$-extended $I$-functions have different extended mirror maps. As we already encode the extended mirror map in Definition \ref{def-theta-tau-mirror}, we have a simple relation for theta functions in Proposition \ref{prod-theta-tau}.
\end{remark}

Now we consider the integrals. We start with the classical period of $W$.  The terms in the classical periods are the $\vartheta_{\vec 0}$-coefficient of $W_\tau^d$ if the coordinates are $\vartheta_{\vec e_i}$. They will be the $x^0$-coefficient of $W_\tau^d$ if the coordinates are chosen as $x_i$. Now, we examine the difference. On one hand, we have
\begin{align}\label{coeff-theta-0}
[W_\tau^d]_{\vartheta_{\vec 0}}=\sum_{\beta, \sum_{i=1}^m d_i=d} \frac{1}{k!\prod_{i=1}^md_i!}\langle \prod_{i=1}^m [1]_{\vec e_i}^{d_i}, \tau,\ldots,\tau,[\on{pt}]_{\vec 0}\bar{\psi}^{d-2}\rangle_{0,d+k+1,\beta}^{(X,D)}
\end{align}
On the other hand, we have
\begin{align}\label{coeff-x-0-tau}
[W_\tau^d(p)]_{x^{\vec 0}}=\sum_{\beta, \sum_{i=1}^m d_i=d} \frac{1}{k!\prod_{i=1}^md_i!}\langle \prod_{i=1}^m [1]_{\vec e_i}^{d_i}, \tau,\ldots,\tau,[1]_{\vec{\mathbf b}},[\on{pt}]_{-\vec{\mathbf b}}\bar{\psi}^{d-1}\rangle_{0,d+k+2,\beta}^{(X,D)},
\end{align}
for $p$ in a maximal cone $\sigma$.
 
Both coefficients (\ref{coeff-theta-0}) and (\ref{coeff-x-0-tau}) can be computed using the relative mirror theorem. In particular, (\ref{coeff-theta-0}) is equal to the corresponding coefficient of the regularized quantum period. While (\ref{coeff-x-0-tau}) is equal to the corresponding coefficient of the regularized quantum period after an additional mirror map for $[1]_{\vec{\mathbf b}}$ which was computed in (\ref{extended-mirror-map}).  We consider this extended mirror map as a function $\tau$ such that its restriction $\tau_\sigma$ to a maximal cone $\sigma$ is:
\[
\tau_\sigma:= \tau_{\vec{\mathbf b}}:=1+\sum_{\substack{\beta: D_j\cdot\beta=0, \text{ for } j\not\in I_\sigma, \\ k\geq 2}} \langle [\on{pt}]_{D_{\beta,-}}\psi^{k-2}\rangle_{0,1,\beta}^X \frac{\prod_{j: d_j>0} d_j!}{\prod_{j:d_j<0} (-d_j-1)!}Q^\beta q^{\vec d}.
\] 
When integrating over the coordinates $x_i$, we replace $\frac{\prod_{i\in I_\sigma}dx_i}{\prod_{i\in I_\sigma}x_i}$ by $\tau_\sigma\frac{\prod_{i\in I_\sigma}dx_i}{\prod_{i\in I_\sigma}x_i}$. Then the same proofs of Theorem \ref{thm-gamma-struc-sheaf} and Theorem \ref{thm-gamma-pt} in Section \ref{sec:gamma-pt-sheaf} and Section \ref{sec:proof-main} when the mirror maps are trival still hold in this setting.

We explain a bit more about the extended mirror map $\tau_\sigma$. When computing the integrals at different coordinates, the difference becomes from whether we consider the coefficient of $\vartheta_{\vec r}$ or the coefficient of $x^{\vec r}$. Now, to compute $\vartheta_{\vec r}$ using the relative mirror theorem, we need an additional marking with insertion $[1]_{\vec r}$. When $[1]_{\vec r}=[1]_{\vec e_i}$ the extended mirror map associated with this insertion is trivial. In general, the mirror map associated with $[1]_{\vec r}$ is not trivial. The extended mirror map $\tau_\sigma$ associated with a mid-age marking can be considered as a limit of the mirror map for $[1]_{\vec r}$ when $D_{\vec {\mathbf b}}\in D_{\vec r}$ as we may consider the mid-ages as having infinite contact orders along the corresponding divisors. 

When $D$ is a smooth anticanonical divisor of $X$. By \cite{You22}, we have
\[
W=\vartheta_1=x\exp\left(g(Q(t,x))\right),
\]
where
\[
g(Q)=\sum_{\substack{\beta\in \on{NE}(X)\\ D\cdot \beta \geq 2}}\langle [\on{pt}]\psi^{D\cdot \beta-2}\rangle_{0,1,\beta}^X t^\beta x^{-D\cdot\beta} (D\cdot \beta-1)!
\]
and $Q=Q(t,x)$ is the inverse of the relative mirror map. In other words,
\[
\vartheta_1\exp\left(-g(Q(t,x))\right)=x,
\]
where the LHS is equal to the theta function $\tilde{\vartheta}_1$ that encodes the relative mirror map. The mirror map for the mid-age marking is
\[
\left([1]_{\mathbf b}+\sum_{\substack{\beta\in \on{NE}(X)\\ D\cdot \beta \geq 2}}\langle [\on{pt}]\psi^{D\cdot \beta-2}\rangle_{0,1,\beta}^X t^\beta x^{-D\cdot\beta} (D\cdot \beta)![1]_{\mathbf b-D\cdot\beta}\right),
\]
which contains all the information about the regularized quantum period. While,
the integral $\int_{|x|=1}e^x \frac{dx}{x}$ does not encode any useful information of the Gromov--Witten invariants.

\section{More general cases}\label{sec:relative-case}

Theorem \ref{thm-gamma-pt} and Theorem \ref{thm-gamma-struc-sheaf} also hold when $X$ is not Fano after imposing additional non-trivial mirror maps in $D$, because the only place where the Fano assumption is used is when we apply the relative mirror theorem in the computation.

Theorem \ref{thm-gamma-struc-sheaf} is based on the assumption that one can find an snc anticanonical divisor $D$ that has zero-dimensional strata. If there is no such $D$ in a Fano variety $X$, then we need to do extra work.

When the Fano variety $X$ is a complete intersection in another variety $Y$ that has a divisor $E$ that satisfies the assumption of Theorem \ref{thm-gamma-struc-sheaf}. Then one can apply the relative mirror construction of Section \ref{sec:mirror-constr} to the pair $(Y,E)$ and the mirror symmetric Gamma conjecture holds for $Y$ by Theorem \ref{thm-gamma-struc-sheaf}. If $Y$ is not Fano, the equality holds after imposing additional mirror maps. Following \cite{Iritani23gamma}, mirror symmetric Gamma conjecture for $X$ follows from that of $Y$ via a Laplace transform. 

In general, there are at least two ways in which Theorem \ref{thm-gamma-struc-sheaf} can be applied. One can try to find another birational model $(X^\prime,D^\prime)$ such that $X^\prime \setminus D^\prime\cong X\setminus D$. Then study the conjecture for the pair $(X^\prime,D^\prime)$. Some results for the birational invariance of the invariants are probably required here. Another way is to consider a maximally unipotent monodromy degeneration of the log Calabi--Yau manifold $X\setminus D$. This is called the relative case in intrinsic mirror construction \cite{GS19}. 

We explain the case where there is a semi-stable degeneration of $X\setminus D$ over a curve $S$
\[
\mathcal X \rightarrow S
\]
 If it is a maximally unipotent monodromy of $X\setminus D$, then we can apply the relative mirror construction of Section \ref{sec:mirror-constr} to $(\mathcal X,\mathcal D)$ to obtain the mirror of $X\setminus D$.

Now we just consider a semi-stable degeneration $\pi: \mathcal X\rightarrow S$, where the central fiber is snc. Let $[X]$ be the divisor class for a smooth fiber. Then the divisor $[X]$ is nef. We assume that $\mathcal X$ has an snc anticanonical divisor $\mathcal D$ that has zero-dimensional strata and the restriction of $(\mathcal X,\mathcal D)$ to a smooth fiber is $(X,D)$. Although $\mathcal X$ is not Fano in general, the proof of the mirror symmetric Gamma conjecture (Theorem \ref{thm-gamma-struc-sheaf}) still holds for $\mathcal X$ (possibly with mirror maps). More precisely, we have

\begin{align}\label{gamma-conj-deg}
\int_{\tilde \Gamma_{\mathbb R}} e^{-\tilde W_{\tau}/z}\tilde \omega=\int_{\mathcal X} \left(z^{c_1}z^{\deg/2}J_{\mathcal X}(\tau_{0,2},-z)\right)\cup \hat{\Gamma}_{\mathcal X},
\end{align}
where the equality holds after a possibly non-trivial change of variables (i.e. mirror maps). Then, following \cite{Iritani11} (see also \cite{Iritani23gamma}), after applying a Laplace transform, one obtains the mirror symmetric Gamma conjecture for $X$:
\begin{align}
\int_{\Gamma_{\mathbb R}} e^{- W_{\tau}/z} \omega=\int_{X} \left(z^{c_1}z^{\deg/2}J_{X}(\tau_{0,2},-z)\right)\cup \hat{\Gamma}_{X}.
\end{align}
This is Theorem \ref{thm-mum-deg}.

\bibliographystyle{amsxport}

\bibliography{main}

\end{document}